\documentclass[final,hidelinks,onefignum,onetabnum]{siamart220329}

\usepackage{lipsum}
\usepackage{amsfonts}
\usepackage{amsmath}
\usepackage{graphicx}
\usepackage{epstopdf}
\usepackage{algorithmic}
\usepackage{bm}

\usepackage{booktabs}
\usepackage{multirow} 
\usepackage{subcaption}
\usepackage{caption}
\newcommand{\alphaLF}{\alpha}
\usepackage{tikz}
\usetikzlibrary{positioning}
\usetikzlibrary{arrows.meta}
\usepackage{extarrows}
\newtheorem{expl}{Example}[section]
\ifpdf
  \DeclareGraphicsExtensions{.eps,.pdf,.png,.jpg}
\else
  \DeclareGraphicsExtensions{.eps}
\fi

\newsiamremark{remark}{Remark}
\newsiamremark{hypothesis}{Hypothesis}
\crefname{hypothesis}{Hypothesis}{Hypotheses}
\newsiamthm{claim}{Claim}

\headers{PPCT for Ideal MHD}{Dongwen Pang, Kailiang Wu}

\title{Provably Positivity-Preserving Constrained Transport (PPCT) Second-Order Scheme for Ideal Magnetohydrodynamics}

\author{Dongwen Pang\thanks{Department of Mathematics, Southern University of Science and Technology, Shenzhen 518055, China
  (\email{12331010@mail.sustech.edu.cn}).}
\and Kailiang Wu\thanks{Department of Mathematics and Shenzhen International Center
for Mathematics, Southern University of Science and Technology, Shenzhen
518055, China
  (\email{wukl@sustech.edu.cn}).}
}

\usepackage{amsopn}

\ifpdf
\hypersetup{
  pdftitle={Provably Positivity-Preserving Constrained Transport Second-Order Scheme for Ideal Magnetohydrodynamics},
  pdfauthor={Dongwen Pang, Kailiang Wu}
}
\fi

\usepackage[left=1.26in, right=1.26in, top=1.25in, bottom=1.25in]{geometry} 

\begin{document}
	
	\large 

\maketitle

\begin{abstract}
This paper proposes and analyzes a robust and efficient second-order positivity-preserving constrained transport (PPCT) scheme  for ideal magnetohydrodynamics (MHD) on non-staggered Cartesian meshes. The PPCT scheme provably preserves two crucial physical constraints: a globally discrete divergence-free (DDF) constraint on the magnetic field and the positivity of density and pressure. It is motivated by a novel splitting technique proposed in [T.A.~Dao,~M. Nazarov \& I.~Tomas, {\it J. Comput. Phys.}, 508: 113009, 2024], which splits the MHD system into an Euler subsystem with a steady magnetic field and a magnetic subsystem with steady density and internal energy. 
			To achieve the structure-preserving properties, the PPCT scheme couples a PP finite volume method for the Euler subsystem with a finite difference CT method for the magnetic subsystem via Strang splitting. The finite volume method is based on a new PP limiter, which is proven to maintain the second-order accuracy of the reconstruction. The PP limiter enforces the positivity of the reconstructed values for density and pressure, as well as an a priori condition for the PP property of the updated cell averages. Rigorous theoretical proof of the PP property is provided using the geometric quasilinearization (GQL) approach [K.~Wu \& C.-W.~Shu, {\it SIAM Review}, 65:1031--1073, 2023].
			For the magnetic subsystem, we construct an implicit finite difference CT method, which conserves energy and preserves a globally DDF constraint on non-staggered Cartesian meshes. The resulting nonlinear algebraic system is solved efficiently with an iterative algorithm, reducing the residual error to machine precision within a few iterations. Unique solvability and convergence of this algorithm are theoretically proven under a CFL-like condition.  Since the finite difference CT method for the magnetic subsystem is unconditionally energy-stable and preserves steady density and internal energy, the time step for the PP property and stability of the PPCT scheme is restricted only by a mild CFL condition for the finite volume method applied to the Euler subsystem. 
			While the primary focus is on the 2D case for clarity, the extension to 3D is also presented. Several challenging numerical experiments, including highly magnetized MHD jets with extremely high Mach numbers, are presented to validate the accuracy, robustness, and high resolution of the PPCT scheme.
			\vspace{3mm}
\end{abstract}



\section{Introduction}
The ideal magnetohydrodynamic (MHD) equations are a fundamental model in astrophysics, space physics, and plasma physics. They play a crucial role in studying phenomena such as coronal mass ejections, the evolution of solar coronal magnetic fields, and the dynamics of Earth's magnetosphere. Due to their nonlinear and hyperbolic nature, deriving analytical solutions is often infeasible. Consequently, numerical methods have become the primary tools for advancing research in this field. 
	In MHD simulations, preserving the underlying physical structures is essential for the reliability of numerical methods. Failure to maintain these structures can lead to numerical instabilities, non-physical results, and the eventual breakdown of simulations
	
	One of the most critical constraints in MHD is the divergence-free (DF) condition of the magnetic field. In numerical simulations, failing to satisfy this condition can result in numerical instabilities and introduce nonphysical artifacts \cite{Toth2000, Brackbill1980, Evans1988, Li2005}. Various techniques have been developed over the past few decades to enforce the DF condition or control the errors arising from its violation. The projection method \cite{Brackbill1980} employs a post-processing step based on Hodge decomposition, projecting the magnetic field onto a DF subspace by solving an elliptic Poisson equation. The widely used constrained transport (CT) method \cite{Londrillo2004divergence, Helzel2011unstaggered} preserves a discrete DF constraint on staggered or non-staggered grids, with numerous variants developed for different numerical frameworks \cite{Dai1998, DeVore1991flux, Gardiner2005}. Another approach, the eight-wave scheme \cite{Powell1995, Powell1997approximate}, introduces non-conservative source terms proportional to the divergence of the magnetic field. This method allows divergence errors to be advected with the flow, thereby mitigating their impact on the solution while maintaining Galilean invariance. The hyperbolic divergence cleaning method \cite{Dedner2002} employs a mixed hyperbolic-parabolic equation to dynamically dampen divergence errors. Additionally, locally divergence-free (LDF) methods \cite{Li2005, Yakovlev2013,  DingWu2024SISCMHD} ensure that the divergence of the magnetic field is zero within each computational cell. Globally divergence-free (GDF) methods \cite{Li2011, Li2012, Balsara2009, Fu2018} extend this enforcement across the entire computational domain by ensuring the continuity of the normal component of the magnetic field across cell interfaces, thereby maintaining an exactly DF magnetic field throughout the simulation.
	
	Another essential physical property that is challenging to preserve in MHD simulations is the positivity of density and pressure. In extreme conditions---such as low densities, low pressures, high Mach numbers, or strong magnetic fields---numerical methods may produce negative values for these quantities. Such non-physical results can lead to the loss of hyperbolicity in the system, causing severe numerical instabilities and potentially leading to the breakdown of the simulation. 
	Early efforts to cope with this issue focused on developing robust one-dimensional approximate Riemann solvers \cite{BalsaraSpicer1999a, Janhunen2000, Bouchut2007, Bouchut2010}, which provided reliable solutions in certain cases. The second-order MUSCL--Hancock method proposed by Waagan \cite{Waagan2009} incorporates a relaxation Riemann solver to ensure the positivity of density and pressure, demonstrating robust performance in benchmark tests. More recently, attention has shifted towards designing positivity-preserving (PP) limiters for higher-order schemes \cite{Zhang2010, Zhang2010b, ZHANG2017301}. 
	For instance, Balsara's self-adjusting PP limiter \cite{Balsara2012} dynamically adjusts the reconstructed variables in finite volume schemes to maintain positivity. Cheng \textit{et~al.}\ \cite{Cheng2013} extended the Zhang--Shu PP limiter \cite{Zhang2010b} to discontinuous Galerkin methods for MHD simulations. 
	In high-order finite volume or discontinuous Galerkin methods, the PP property may be lost in two situations: when the reconstructed or DG solution polynomials fail to be positive, or when the cell averages become negative during the updating process \cite{Zhang2010b, Zhang2010}. While the PP limiters \cite{Balsara2012, Cheng2013} can recover positivity in the first case, ensuring the positivity of the cell averages during the updating process is more challenging and critical for obtaining a genuinely PP scheme. The validity of the PP limiters \cite{Balsara2012, Cheng2013} relies on the positivity of the cell averages in the updating process; however, this was not rigorously proven for these high-order schemes. 
	Additionally, high-order finite difference methods employing parameterized flux limiters to maintain positivity have been introduced \cite{Christlieb2015PP, Christlieb2016}. These methods leverage the presumed PP property of certain low-order schemes, such as the Lax--Friedrichs scheme---whose PP property was later rigorously analyzed in \cite{Wu2017SIAMJNAppMHD}---to adjust the high-order numerical flux and ensure positivity. 
	Despite these advances, the PP property of most of these schemes has been justified primarily through numerical experiments, and few schemes have rigorous proofs of their PP property for the MHD equations.

	Recent advances \cite{Wu2017SIAMJNAppMHD, WuShu2018, WuShu2019, WuTangM3AS, WuJiangShu2022, WuShu2020NumMath} have theoretically revealed the key relationship between the PP property---an algebraic structure---and the DF condition---a differential structure---in both ideal and relativistic compressible MHD systems. 
	It has been rigorously proven for conservative numerical schemes that the PP property is closely linked to maintaining a discrete divergence-free (DDF) condition; see the theoretical analyses and findings in \cite{Wu2017SIAMJNAppMHD} for finite volume and discontinuous Galerkin schemes, and \cite{WuJiangShu2022} for central discontinuous Galerkin methods. 
	It was observed that even slight violations of the DDF condition might compromise the PP property during numerical simulations \cite{Wu2017SIAMJNAppMHD}. Unfortunately, the globally coupled nature of the DDF condition makes it incompatible with standard local scaling PP limiters in multidimensional cases. Maintaining one property can inadvertently compromise the other, complicating the development of provably PP numerical methods for MHD. 
	On the other hand, researchers have also discovered the relationship between the PP and DF properties at the continuous level. It has been shown that even the exact smooth solutions of the conservative MHD system can lead to negative pressure when the magnetic field divergence is nonzero \cite{WuShu2018, WuShu2020NumMath}. This finding further confirms the intricate coupling between the positivity constraint and the DF condition, making the simultaneous preservation of both properties in multidimensional MHD simulations a challenging yet highly desirable task.

	To address this challenge, researchers have attempted to reformulate the standard conservation form of the MHD equations into different forms to reduce the intricate coupling between the PP and DF constraints. 
	Notably, it was observed that the symmetrizable MHD system, which incorporates the Godunov--Powell source term \cite{Godunov1972, Powell1995}, always has the PP property at the continuous level, regardless of whether the DF condition is satisfied \cite{WuShu2019}. 
	Motivated by this crucial finding, it was demonstrated that a suitable discretization of the symmetrizable MHD system can result in high-order schemes whose PP property relies only on a locally DDF condition, which is compatible with the local scaling PP limiter. 
	Based on the symmetrizable MHD system, high-order provably PP schemes were successfully designed in \cite{WuShu2018} using LDF discontinuous finite elements, and in \cite{DingWu2024SISCMHD} using a novel discrete LDF projection for finite volume methods. 
	Subsequently, a framework for provably positive high-order finite volume and discontinuous Galerkin schemes was established on general meshes \cite{WuShu2019}. At the heart of these advancements is the geometric quasilinearization (GQL) approach \cite{Wu2017SIAMJNAppMHD,WuShu2021GQL}, which converts complex nonlinear constraints---such as the positivity of pressure---into more manageable linear ones. GQL achieves this by integrating appropriate auxiliary variables and leveraging the geometric characteristics of convex sets. Building on this framework, researchers have systematically developed a range of high-order PP and physical-constraint-preserving numerical methods, such as \cite{Wu2017SIAMJNAppMHD,WuTangM3AS,WuShu2018, WuShu2019,  DingWu2024JCPGQLBased, Wu2021Minimum} and the more recent PP work \cite{LiuWu2024OEDGforMHD} within the oscillation-eliminating framework \cite{PSW2024}. 
	Besides adding the Godunov--Powell source term, another useful approach is operator splitting, i.e., splitting the ideal MHD system into the Euler equations and a magnetic system \cite{Dao2024Structure}. Based on such a novel splitting, a second-order accurate structure-preserving finite element method was proposed by Dao, Nazarov, and Tomas \cite{Dao2024Structure}. Specifically, this method features a non-divergence formulation, treating the magnetic field's impact on momentum and energy as source terms and using the Strang splitting technique to derive a novel splitting form that includes an Euler part and a magnetic part. Using this splitting form, the PP property and the DF condition can be considered independently. Besides preserving the positivity of density and pressure using convex limiting techniques, this method also maintains the conservation properties for total mass and total energy \cite{Dao2024Structure}.

	The aim of this paper is to design and analyze a novel, efficient, and robust second-order scheme (termed PPCT) for ideal MHD, which  simultaneously preserves positivity via the GQL approach and enforces a global DDF constraint via the CT approach on non-staggered Cartesian meshes, with rigorous theoretical proofs. The key contributions, novelty, and significance of this work include:
	\begin{itemize}
		\item  Inspired by \cite{Dao2024Structure}, we split the ideal MHD equations into two subsystems: a compressible Euler system with a steady magnetic field and a magnetic system with steady density and internal energy. We propose a new finite volume–finite difference (FV-FD) hybrid scheme (i.e., PPCT), which couples a novel, provably PP finite volume method for the Euler subsystem with an unconditionally energy-stable finite difference CT method on non-staggered grids for the magnetic subsystem. A second-order Strang splitting approach is utilized to alternate the evolution of these two methods, efficiently achieving both the PP and globally DDF properties. 
		
		\item  The finite volume method for the Euler subsystem is based on second-order reconstruction with a slope limiter. We opt to reconstruct primitive variables, as reconstructing conservative variables tends to introduce overshoots or nonphysical oscillations by numerical evidence. The PP property of the finite volume method is ensured in two key aspects: enforcing the positivity of the reconstructed limiting values for density and pressure, and preserving the positivity of updated cell averages during time-stepping. 
		The first aspect is addressed by a carefully designed PP limiter, which serves as a correction step following the reconstruction. We prove that this PP limiter preserves second-order accuracy. The second aspect is guaranteed by rigorous theoretical analysis and a velocity slope suppression mechanism embedded in the PP limiter, which provides an a priori condition for the PP property of updated cell averages. 
		Since we reconstruct primitive variables, the commonly used cell average decomposition  \cite{Zhang2010,CuiDingWu2022SINUM,CuiDingWu2023JCP} does not hold automatically. Instead, our a priori condition plays a crucial role in proving the PP property. The rigorous proof is based on the GQL approach \cite{WuShu2021GQL}, which equivalently transforms the nonlinear positivity constraint on pressure into a set of linear constraints, thus overcoming the challenges posed by nonlinearity.
		
		\item We develop an implicit finite difference CT scheme for the magnetic subsystem, ensuring the globally DDF condition for the magnetic field on non-staggered grids. Unlike the structure-preserving finite element method in \cite{Dao2024Structure}, our method does not involve matrix computations, making it computationally more efficient. To solve the resulting nonlinear algebraic system, we introduce a simple yet efficient iterative algorithm, which shows exponential convergence. Specifically, the residual error is reduced to machine precision within 5 to 9 iterations on average, and never exceeds 20 steps. Unique solvability and convergence of this algorithm are theoretically proven under a CFL-like condition.  Additionally, we rigorously prove that the CT scheme conserves energy. 
		Since the implicit finite difference method for the magnetic subsystem is unconditionally energy-stable and preserves both density and internal energy, the time step for the PP property and stability of the PPCT scheme is restricted only by a mild CFL condition for the finite volume method applied to the Euler subsystem.
		
		\item We extend the PPCT method to three dimensions, maintaining the same numerical framework as in the two-dimensional (2D) case. All the desirable structure-preserving properties, such as positivity preservation, global DDF enforcement, and energy conservation, are retained in 3D.
		
		\item Finally, we validate the proposed schemes on a variety of benchmark and challenging problems. The numerical results demonstrate the excellent performance of the PPCT method, showing high resolution and robustness.
	\end{itemize}
	In short, the proposed PPCT scheme is an efficient and robust second-order FV-FD hybrid method with only a mild time step restriction. It is straightforward to implement on non-staggered Cartesian meshes. The scheme successfully achieves both positivity preservation and the globally DDF property. Moreover, the PPCT scheme conserves both mass and total energy, the latter of which is often lost in some other CT methods with positivity enforcement \cite{Christlieb2015PP, Christlieb2016}.

	This paper is organized as follows. In \Cref{S.2}, we briefly review the ideal MHD system and the novel splitting method introduced in \cite{Dao2024Structure}. In \Cref{S.3}, we develop the PPCT numerical method with rigorous theoretical proofs for the 2D ideal MHD system. \Cref{S.4} extends the scheme to the 3D case. \Cref{S.5} provides numerical examples to demonstrate the accuracy and robustness of our method. Finally, \Cref{S.6} concludes the paper.
	\section{Compressible Ideal MHD Equations and Operator Splitting}\label{S.2}
	The ideal MHD system can be expressed in divergence form as
	\begin{equation}\label{MHDSystem}
		{\bm U}_t + \nabla \cdot {\bm F}({\bm U}) = \mathbf{0},
	\end{equation}
	where ${\bm U} = \left( \rho, {\bm m}, {\bm B}, E_{tot} \right)^{\top} \in \mathbb{R}^8$ represents the conservative variables. Here, $\rho$ is the density, ${\bm m} = \left(m_1, m_2, m_3\right)^{\top}$ is the momentum vector, ${\bm B} = \left(B_1, B_2, B_3\right)^{\top}$ denotes the magnetic field vector, and $E_{tot}$ stands for the total energy, which includes internal, kinetic, and magnetic energies. 
	In the $d$-dimensional case ($d \in \{1, 2, 3\}$), the flux ${\bm F} = \left( {\bm F}_1, \cdots , {\bm F}_d \right)$ is defined by
	\begin{equation*}
		{\bm F}_i({\bm U}) = \left(m_i, m_i {\bm v} - B_i {\bm B} + p_{tot} \mathbf{e}_i, v_i {\bm B} - B_i {\bm v}, v_i \left(E_{tot} + p_{tot}\right) - B_i({\bm v} \cdot {\bm B})\right)^{\top}, 
	\end{equation*}
	where $\mathbf{e}_i$ is a unit vector of size 3 with the $i$-th element being 1 and others being 0. The fluid velocity is given by ${\bm v} = \left(v_1, v_2, v_3\right)^{\top} = {\bm m}/\rho$, and the total pressure is defined as $p_{tot} = p + \frac{1}{2} \left|{\bm B}\right|^2$, which includes both thermal and magnetic pressure. 
	Denoting $\gamma$ as the ratio of specific heats, the system \eqref{MHDSystem} can be closed by the ideal equation of state (EOS):
	\begin{equation*}
		E_{tot} = \frac{p}{\gamma - 1} + \frac{1}{2} \rho \left|{\bm v}\right|^2 + \frac{1}{2} \left|{\bm B}\right|^2,
	\end{equation*}
	where $\left|\cdot\right|$ denotes the Euclidean norm. If the magnetic field is DF initially, it follows that the exact solution will remain DF at any time, i.e.,
	\begin{equation*}
		\nabla \cdot {\bm B} = 0.
	\end{equation*}
	In physics, the DF property of the magnetic field implies the non-existence of magnetic monopoles; therefore, a physical solution must satisfy this constraint.

	As shown in \cite{Wu2017SIAMJNAppMHD,WuShu2019}, the PP property and the DF property are closely connected for the ideal MHD system \eqref{MHDSystem}.  This intricate connection  makes the simultaneous preservation of both properties for multidimensional MHD schemes  a challenging task. 
	To address this challenge,  
	instead of solving the divergence form \eqref{MHDSystem} directly, we follow  \cite{Dao2024Structure} and reformulate the MHD system \eqref{MHDSystem} into  a non-divergence form as follows:
	\begin{equation}\label{inductionForm}
		\begin{aligned}
			\partial_t \rho + \operatorname{div} \boldsymbol{m} & = 0, \\
			\partial_t \boldsymbol{m} + \operatorname{div}\left(\rho^{-1} \boldsymbol{m} \boldsymbol{m}^{\top} + \mathbf{I} p\right) & = -\boldsymbol{B} \times \operatorname{curl} \boldsymbol{B}, \\
			\partial_t E + \operatorname{div}\left(\frac{\boldsymbol{m}}{\rho}(E + p)\right) & = -\left(\boldsymbol{B} \times \operatorname{curl} \boldsymbol{B}\right) \cdot \frac{\boldsymbol{B}}{\rho}, \\
			\partial_t \boldsymbol{B} - \operatorname{curl}\left(\frac{\boldsymbol{m}}{\rho} \times \boldsymbol{B}\right) & = \mathbf{0},
		\end{aligned}
	\end{equation}
	where $E = \rho e + \frac{1}{2} \rho \left| \boldsymbol{v} \right|^2$ represents the total mechanical energy (excluding contributions from magnetic energy), and $\mathbf{I} \in \mathbb{R}^3$ is the identity matrix. In this formulation, the magnetic forces are treated as source terms influencing the momentum and total mechanical energy, while the governing equations for the magnetic field are represented by the induction equation. For sufficiently smooth solutions, the divergence form \eqref{MHDSystem} and the non-divergence form \eqref{inductionForm} are equivalent. However, for solutions that may exhibit discontinuities or weak differentiability, these two forms should be regarded as distinct mathematical representations.
	
	Utilizing the operator-splitting technique \cite{Dao2024Structure}, we can separate the non-divergence form \eqref{inductionForm} into two subsystems:
	\begin{equation}\label{Operator1}
		\text{System-A}\left\{\begin{aligned}
			\partial_t \rho + \operatorname{div} \boldsymbol{m} & = 0, \\
			\partial_t \boldsymbol{m} + \operatorname{div}\left(\rho^{-1} \boldsymbol{m} \boldsymbol{m}^{\top} + \mathbf{I} p\right) & = \mathbf{0}, \\
			\partial_t E + \operatorname{div}\left(\frac{\boldsymbol{m}}{\rho}(E + p)\right) & = 0, \\
			\partial_t \boldsymbol{B} & = \mathbf{0},
		\end{aligned}\right.
	\end{equation}
	\begin{equation}\label{Operator2}
		\text{System-B}\left\{\begin{aligned}
			\partial_t \rho & = 0, \\
			\partial_t \boldsymbol{m} & = -\boldsymbol{B} \times \operatorname{curl} \boldsymbol{B}, \\
			\partial_t E & = -\left(\boldsymbol{B} \times \operatorname{curl} \boldsymbol{B}\right) \cdot \frac{\boldsymbol{m}}{\rho}, \\
			\partial_t \boldsymbol{B} & = \operatorname{curl}(\boldsymbol{v} \times \boldsymbol{B}).
		\end{aligned}\right.
	\end{equation}
	System-A \eqref{Operator1} corresponds to the compressible Euler equations under a steady magnetic field. A key property of System-B is the invariance of density and internal energy \cite[Lemma 2.4]{Dao2024Structure}. Utilizing this property, System-B can be simplified to:
	\begin{equation}\label{Subsystem}
		\text{System-B}\left\{\begin{aligned}
			\partial_t \rho & = 0, \\
			\partial_t e & = 0, \\
			\rho \partial_t \boldsymbol{v} & = -\boldsymbol{B} \times \operatorname{curl} \boldsymbol{B}, \\
			\partial_t \boldsymbol{B} & = \operatorname{curl}(\boldsymbol{v} \times \boldsymbol{B}).
		\end{aligned}\right.
	\end{equation}
	This implies that the solutions of System-B always preserve positivity for both density and internal energy, allowing us to focus on constructing provably PP schemes only for System-A. Moreover, since the evolution of the magnetic field is confined to System-B, the DF constraint must only be preserved when solving System-B. This innovative splitting strategy \cite{Dao2024Structure} effectively decouples the preservation of positivity from the DF constraint.

	\section{Two-Dimensional PPCT Scheme}\label{S.3}
	In this section, we describe our PPCT scheme for the 2D ideal MHD system based on the above  operator-splitting approach. Let ${\bm U}^n$ denote the numerical solution at time level $n$. The PPCT scheme, based on the second-order Strang splitting method, is expressed as
	\begin{equation}\label{eq:fullscheme}
		{\bm U}^{n+1} = S_{A}^{\frac{\Delta t}{2}} \circ S_{B}^{\Delta t} \circ S_{A}^{\frac{\Delta t}{2}} {\bm U}^n,
	\end{equation}
	where $\Delta t$ is the time step from time level $n$ to time level $n+1$, and $S_{A}^{\Delta t}$ and $S_{B}^{\Delta t}$ represent the forward evolution operators corresponding to the numerical schemes for for System-A \eqref{Operator1} and System-B \eqref{Subsystem}, respectively. 
	
	For System-A \eqref{Operator1}, we will develop a second-order finite volume method, denoted as $S_{A}^{\Delta t}$, which is rigorously proven to preserve the positivity of both density and pressure. 
	For System-B \eqref{Subsystem}, we will introduce a second-order finite difference CT scheme, denoted as $S_{B}^{\Delta t}$. This scheme is designed to automatically maintain the DDF condition on the magnetic field, while also ensuring the conservation of total energy. 
	As a result, the overall PPCT scheme \eqref{eq:fullscheme} achieves second-order accuracy in both time and space. 
	
	\begin{remark}\label{rem:stability}
		Since the CT scheme $S_{B}^{\Delta t}$ for System-B \eqref{Subsystem} employs an implicit time discretization, it demonstrates unconditional stability in our numerical experiments. Therefore, the time step for stability of the PPCT scheme is only restricted by the standard CFL condition of $S_{A}^{\frac{\Delta t}{2}}$ for the Euler part \eqref{Operator1}. This is similar to the structure-preserving finite element method of Dao, Nazarov, and Tomas \cite{Dao2024Structure}. 
	\end{remark}
	
	\begin{remark}
		An alternative second-order Strang splitting strategy is
		\[
		{\bm U}^{n+1} = S_{B}^{\frac{\Delta t}{2}} \circ S_{A}^{\Delta t} \circ S_{B}^{\frac{\Delta t}{2}} {\bm U}^n,
		\]
		which requires two applications of $S_{B}^{\frac{\Delta t}{2}}$. However, since the computational cost of $S_{B}^{\Delta t}$ is generally higher than that of $S_{A}^{\Delta t}$, we opt for the splitting strategy \eqref{eq:fullscheme}, which reduces the overall computational expense.
	\end{remark}

	Assume the computational domain is partitioned into a uniform rectangular mesh, with the cell centered at $(x_i, y_j)$ denoted by $I_{ij} = (x_{i-\frac{1}{2}}, x_{i+\frac{1}{2}}) \times (y_{j-\frac{1}{2}}, y_{j+\frac{1}{2}})$. The center coordinates are $x_i = \frac{1}{2} \left( x_{i-\frac{1}{2}} + x_{i+\frac{1}{2}} \right)$ and $y_j = \frac{1}{2} \left( y_{j-\frac{1}{2}} + y_{j+\frac{1}{2}} \right)$. The constant spatial step sizes in the $x$- and $y$-directions are denoted by $\Delta x$ and $\Delta y$, respectively. 
	We will detail the provably PP finite volume scheme $S_{A}^{{\Delta t}}$ and the CT finite difference scheme $S_{B}^{{\Delta t}}$ in Sections \ref{section:operator1} and \ref{section:operator2}, respectively. Finally, we summarize the overall PPCT operator-splitting scheme for the full MHD system in Section \ref{eq:summary}.

	\subsection{Provably PP Finite Volume Method $S_{A}^{{\Delta t}}$ for System-A}\label{section:operator1}
	
	System-A consists of the Euler equations coupled with a steady magnetic field. We begin by considering the 2D Euler equations:
	\begin{equation}\label{Eulersystem}
		\bm{Q}_t +  \nabla \cdot \bm{G}(\bm{Q}) = \mathbf{0},
	\end{equation}
	where $\bm{Q} = (\rho, \bm{m}, E)^\top$, and the flux $\bm{G} = (\bm{G}_1, \bm{G}_2)$ is defined as
	\[
	\bm{G}_i(\bm{Q}) = \left( m_i, m_i \bm{v} + p \mathbf{e}_i, (E + p) v_i \right)^\top, \quad i = 1, 2,
	\]
	with $\bm{m}, \bm{v} \in \mathbb{R}^3$ to maintain consistency with the MHD system.
	
	\subsubsection{Outline of the Finite Volume Method}
	Denote the approximate cell averages as 
	$\overline{\bm{Q}}_{ij}(t) \approx \frac{1}{\Delta x \Delta y} \int_{I_{ij}} \bm{Q}(x, y, t) \, \mathrm{d} x \, \mathrm{d} y.$ 
	Then the semi-discrete finite volume method for \eqref{Eulersystem} can be written as
	\begin{equation}\label{FVM2DEulerSemi}
		\frac{\mathrm{d}}{\mathrm{d} t} \overline{\bm{Q}}_{ij}(t) = \mathcal{L}_{ij}(\overline{\bm{Q}}(t)),
	\end{equation}
	with
	\[
	\mathcal{L}_{ij}(\overline{\bm{Q}}(t)) := - \frac{1}{\Delta x} \left( \hat{\bm{G}}_{1, i+\frac{1}{2}, j}(t) - \hat{\bm{G}}_{1, i-\frac{1}{2}, j}(t) \right) - \frac{1}{\Delta y} \left( \hat{\bm{G}}_{2, i, j+\frac{1}{2}}(t) - \hat{\bm{G}}_{2, i, j-\frac{1}{2}}(t) \right).
	\]
	Here, $\hat{\bm{G}}_{1,i+\frac{1}{2},j}(t)$ and $\hat{\bm{G}}_{2,i,j+\frac{1}{2}}(t)$ represent numerical flux approximations at the cell interfaces:
	\begin{equation}\label{NumericalFlux}
		\begin{aligned}
			\hat{\bm{G}}_{1,i+\frac{1}{2},j}(t) &= \hat{\bm{G}}_1\left( \bm{Q}^{-}_{i+\frac{1}{2},j}, \bm{Q}^{+}_{i+\frac{1}{2},j} \right) \approx \frac{1}{\Delta y} \int_{y_{j-\frac{1}{2}}}^{y_{j+\frac{1}{2}}} \bm{G}_1\left( \bm{Q}(x_{i+\frac{1}{2}}, y, t) \right) \, \mathrm{d} y, \\
			\hat{\bm{G}}_{2, i, j+\frac{1}{2}}(t) &= \hat{\bm{G}}_2\left( \bm{Q}^{-}_{i, j+\frac{1}{2}}, \bm{Q}^{+}_{i, j+\frac{1}{2}} \right) \approx \frac{1}{\Delta x} \int_{x_{i-\frac{1}{2}}}^{x_{i+\frac{1}{2}}} \bm{G}_2\left( \bm{Q}(x, y_{j+\frac{1}{2}}, t) \right) \, \mathrm{d} x,
		\end{aligned}
	\end{equation}
	where $\bm{Q}^{\pm}_{i+\frac{1}{2},j}$ (resp. $\bm{Q}^{\pm}_{i,j+\frac{1}{2}}$) denotes the left and right limiting values of the approximate solutions at $(x_{i+\frac{1}{2}}, y_j)$ (resp. $(x_i, y_{j+\frac{1}{2}})$) on the cell interfaces. These limiting values are computed using the second-order reconstruction, discussed later. 
	While various numerical fluxes can be used, in this work, we adopt the Lax--Friedrichs flux:
	\begin{equation}\label{LFflux}
		\hat{\bm{G}}_k(\bm{Q}^-, \bm{Q}^+) = \frac{1}{2} \left( \bm{G}_k(\bm{Q}^-) + \bm{G}_k(\bm{Q}^+) - \alphaLF_k (\bm{Q}^+ - \bm{Q}^-) \right), \quad k = 1, 2,
	\end{equation}
	where $\alphaLF_k$ is the numerical viscosity parameter, defined as
	\[
	\alphaLF_k = \left\|\left| v_k \right| + c \right\|_{L^\infty}, \quad k = 1, 2,
	\]
	with $c = \sqrt{\frac{\gamma p}{\rho}}$ denoting the sound speed.

	To achieve a fully discrete second-order scheme, we apply the second-order explicit strong stability-preserving (SSP) Runge--Kutta time discretization to the semi-discrete form \eqref{FVM2DEulerSemi} as follows:
	\begin{equation}\label{SSPRK2}
		\begin{aligned}
			\overline{\bm{Q}}_{ij}^{(1)} &= \overline{\bm{Q}}_{ij}^{n} + \Delta t \, \mathcal{L}_{ij}(\overline{\bm{Q}}^n), \\
			\overline{\bm{Q}}_{ij}^{n+1} &= \frac{1}{2} \overline{\bm{Q}}_{ij}^{n} + \frac{1}{2} \left( \overline{\bm{Q}}_{ij}^{(1)} + \Delta t \, \mathcal{L}_{ij}(\overline{\bm{Q}}^{(1)}) \right),
		\end{aligned}
	\end{equation}
	where $\Delta t$ represents the time step size, and $\overline{\bm{Q}}_{ij}^{n} = \overline{\bm{Q}}_{ij}(t_n)$ denotes the approximate cell average over $I_{ij}$ at time $t_n$.

	\subsubsection{Second-Order Reconstruction}\label{sec:reconstruction}
	As an example, consider the reconstruction procedure at time level $n$. We reconstruct the limiting values from the cell averages $\left\{\overline{\bm{Q}}_{ij}^n\right\}$. In this work, we reconstruct the primitive variables $\bm{W} = (\rho, \bm{v}, p)^\top$. For a given $\overline{\bm{Q}}_{ij}^n = \left( \overline\rho_{ij}^n,\overline{ \bm{m}}_{ij}^n, \overline E_{ij}^n \right)^\top$, the values of the corresponding primitive variables are given by
	\[
	\overline{\bm{W}}_{ij}^n = \left( \overline \rho_{ij}^n,~ \frac{ \overline{  \bm{m}}_{ij}^n}{\overline \rho_{ij}^n},~ (\gamma - 1) \left( \overline E_{ij}^n - \frac{| \overline{\bm{m}}_{ij}^n|^2}{2 \overline \rho_{ij}^n} \right) \right)^\top.
	\]
	
	We reconstruct a piecewise linear function as follows:
	\[
	\bm{W}_{ij}^n(x, y) = \overline{\bm{W}}_{ij}^n + \delta \bm{W}_{ij}^x (x - x_i) + \delta \bm{W}_{ij}^y (y - y_j), \quad (x, y) \in I_{ij},
	\]
	where $\overline{\bm{W}}_{ij}^n$ is computed from the cell average $\overline{\bm{Q}}_{ij}^n$ rather than the actual average of the primitive variables. The slopes $\delta \bm{W}_{ij}^x$ and $\delta \bm{W}_{ij}^y$ are obtained using a slope limiter (e.g., minmod, superbee). In this work, we use the van Albada limiter, where the slopes are defined as
	\begin{equation}\label{VanAlbada}
		\begin{aligned}
			\delta {\bm W}_{ij}^x &= \frac{\left( \left[ \frac{\overline{\bm{W}}_{i+1,j}^n - \overline{\bm{W}}_{ij}^n}{\Delta x} \right]^2 + \epsilon_x \right)\odot \frac{\overline{\bm{W}}_{ij}^n - \overline{\bm{W}}_{i-1,j}^n}{\Delta x} + \left( \left[ \frac{\overline{\bm{W}}_{ij}^n - \overline{\bm{W}}_{i-1,j}^n}{\Delta x} \right]^2 + \epsilon_x \right) \odot \frac{\overline{\bm{W}}_{i+1,j}^n - \overline{\bm{W}}_{ij}^n}{\Delta x}}{\left[ \frac{\overline{\bm{W}}_{ij}^n - \overline{\bm{W}}_{i-1,j}^n}{\Delta x} \right]^2 + \left[ \frac{\overline{\bm{W}}_{i+1,j}^n - \overline{\bm{W}}_{ij}^n}{\Delta x} \right]^2 + 2 \epsilon_x }, \\
			\delta {\bm W}_{ij}^y &= \frac{\left( \left[ \frac{\overline{\bm{W}}_{i,j+1}^n - \overline{\bm{W}}_{ij}^n}{\Delta y} \right]^2 + \epsilon_y \right) \odot \frac{\overline{\bm{W}}_{ij}^n - \overline{\bm{W}}_{i,j-1}^n}{\Delta y} + \left( \left[ \frac{\overline{\bm{W}}_{ij}^n - \overline{\bm{W}}_{i,j-1}^n}{\Delta y} \right]^2 + \epsilon_y \right) \odot \frac{\overline{\bm{W}}_{i,j+1}^n - \overline{\bm{W}}_{ij}^n}{\Delta y}}{\left[ \frac{\overline{\bm{W}}_{ij}^n - \overline{\bm{W}}_{i,j-1}^n}{\Delta y} \right]^2 + \left[ \frac{\overline{\bm{W}}_{i,j+1}^n - \overline{\bm{W}}_{ij}^n}{\Delta y} \right]^2 + 2\epsilon_y}.
		\end{aligned}
	\end{equation}
	Here, ``$\odot$'' denotes the Hadamard product, $\epsilon_x = 3 \Delta x$, and  $\epsilon_y = 3 \Delta y$. Define 
	$$\Delta \bm{W}_{ij}^x = \left( \Delta \rho_{ij}^x, \Delta \bm{v}_{ij}^x, \Delta p_{ij}^x \right)^\top = \frac{\Delta x}{2} \delta \bm{W}_{ij}^x, \qquad \Delta \bm{W}_{ij}^y = \left( \Delta \rho_{ij}^y, \Delta \bm{v}_{ij}^y, \Delta p_{ij}^y \right)^\top = \frac{\Delta y}{2} \delta \bm{W}_{ij}^y.$$
	The limiting values are then expressed as
	\begin{equation}\label{deltaXY}
		\begin{aligned}
			\bm{W}^{n,-}_{i+\frac{1}{2},j} &= \overline{\bm{W}}_{ij}^n + \Delta \bm{W}_{ij}^x, \qquad 
			\bm{W}^{n,+}_{i-\frac{1}{2},j} = \overline{\bm{W}}_{ij}^n - \Delta \bm{W}_{ij}^x, \\
			\bm{W}^{n,-}_{i,j+\frac{1}{2}} &= \overline{\bm{W}}_{ij}^n + \Delta \bm{W}_{ij}^y, \qquad 
			\bm{W}^{n,+}_{i,j-\frac{1}{2}} = \overline{\bm{W}}_{ij}^n - \Delta \bm{W}_{ij}^y.
		\end{aligned}
	\end{equation}
	
	\begin{remark}
		Our reconstruction procedure is applied to the primitive variables for reasons of simplicity, efficiency, and stability. Although direct reconstruction of the conservative variables is straightforward, numerical experiments have shown that it can lead to some nonphysical overshoots and oscillations (see, e.g., \Cref{fig:Ex-Rotor2}). In contrast, reconstructing the primitive variables results in a scheme that exhibits no such spurious structures or oscillations.
	\end{remark}
	
	Using the van Albada limiter ensures the non-oscillatory property of the piecewise linear reconstruction in \eqref{deltaXY}. However, directly using these limiting values does not guarantee a PP scheme. 
	
	\subsubsection{Positivity-Preserving Goals}

	For physical solutions, both the density and internal energy (or pressure) must remain positive. Let $\mathcal{G}$ denote the set of physically admissible states, defined as:
	\begin{equation*}
		\mathcal{G}=\left\{ \bm{Q}=(\rho, {\bm m},  E)^{\top}:~ \rho>0,\quad  \mathcal{E}(\bm{Q}):=E-\frac{|{\bm m}|^2}{2\rho}>0\right\}.
	\end{equation*}
	It is well known that $\mathcal{G}$ is a convex set \cite{ShuChiwang1996PPforEuler}. The second constraint, $\mathcal{E}(\bm{Q})>0$, is nonlinear with respect to $\bm{Q}$, making it difficult to analyze the PP property. Using the GQL representation \cite{WuShu2021GQL}, we transform this nonlinear constraint into equivalent linear constraints. This novel representation, denoted as $\mathcal{G}_*$, is crucial for the proof presented later.
	
	\begin{lemma}[GQL Representation \cite{WuShu2021GQL}]\label{Lemma:GQLrepresentation}
		The admissible state set $\mathcal{G}$ is exactly equivalent to 
		\begin{equation*}
			\mathcal{G}_*=\left\{\bm{Q}=(\rho, {\bm m}, E)^{\top}:~ \bm{Q} \cdot {\bm n}_1>0, \quad \bm{Q} \cdot \boldsymbol{n}_* >0 \quad \forall \boldsymbol{v}_* \in \mathbb{R}^3\right\},
		\end{equation*}
		where ${\bm n}_1=(1,0,0,0,0)^\top$ and $\boldsymbol{n}_* = \left(\frac{ \left| {\bm v}_* \right|^2}{2}, -{\bm v}_*, 1\right)^\top$, with $\boldsymbol{v}_*$ as free auxiliary variables independent of $\bm{Q}$.
	\end{lemma}

	In the finite volume framework, given that $\overline{\bm{Q}}^{n}_{ij} \in \mathcal{G}$ for all $i$ and $j$, the overall goal of maintaining positivity can be broken down into two sub-goals:
	\begin{enumerate}
		\item Ensure the positivity of the reconstructed limiting values;
		\item Preserve the positivity of the cell averages at the next time step, i.e., $\overline{\bm{Q}}^{n+1}_{ij} \in \mathcal{G}$.
	\end{enumerate}
	
	In the following, we introduce a simple PP limiter to ensure that the limiting values for both density and pressure remain positive, thereby achieving the first sub-goal. 
	Later, we will rigorously prove that the scheme, when equipped with this limiter, is PP for the updated cell averages, addressing the second sub-goal. This, in turn, guarantees the effectiveness of the PP limiter in future time steps as well.

	\subsubsection{Positivity-Preserving Limiter for Point Values}

	Given $\overline{\rho}_{ij}^n > 0$ and $\overline{p}_{ij}^n > 0$, our PP limiter ensures the positivity of point values by applying slope suppression. To achieve this, we introduce parameters $\alpha^{\star}, \beta, \kappa^{\star} \in [0, 1]$ (where "$\star$" can be $x$ or $y$) to control the modification of the slopes. The primitive variables, modified by the PP limiter, are defined as follows:
	\begin{equation}\label{2D_Euler_limiter}
		\begin{aligned}
			\widetilde{{\bm W}}^{n,+}_{i-\frac{1}{2},j} &= \begin{pmatrix}
				\tilde{\rho}_{i-\frac{1}{2}, j}^{n,+} \\
				\tilde{{\bm v}}_{i-\frac{1}{2}, j}^{n,+} \\
				\tilde{p}_{i-\frac{1}{2}, j}^{n,+}
			\end{pmatrix} = 
			\begin{pmatrix}
				\overline{\rho}_{ij}^n - \alpha^x \Delta \rho_{ij}^{x} \\
				\overline{{\bm v}}_{ij}^n - \beta \Delta {\bm v}_{ij}^{x} \\
				\overline{p}_{ij}^n - \kappa^x \Delta p_{ij}^{x}
			\end{pmatrix}, \qquad 
			\widetilde{{\bm W}}^{n,-}_{i+\frac{1}{2},j} = \begin{pmatrix}
				\tilde{\rho}_{i+\frac{1}{2}, j}^{n,-} \\
				\tilde{{\bm v}}_{i+\frac{1}{2}, j}^{n,-} \\
				\tilde{p}_{i+\frac{1}{2}, j}^{n,-}
			\end{pmatrix} = 
			\begin{pmatrix}
				\overline{\rho}_{ij}^n + \alpha^x \Delta \rho_{ij}^{x} \\
				\overline{{\bm v}}_{ij}^n + \beta \Delta {\bm v}_{ij}^{x} \\
				\overline{p}_{ij}^n + \kappa^x \Delta p_{ij}^{x}
			\end{pmatrix}, \\
			\widetilde{{\bm W}}^{n,+}_{i,j-\frac{1}{2}} &= \begin{pmatrix}
				\tilde{\rho}_{i, j-\frac{1}{2}}^{n,+} \\
				\tilde{{\bm v}}_{i, j-\frac{1}{2}}^{n,+} \\
				\tilde{p}_{i, j-\frac{1}{2}}^{n,+}
			\end{pmatrix} = 
			\begin{pmatrix}
				\overline{\rho}_{ij}^n - \alpha^y \Delta \rho_{ij}^{y} \\
				\overline{{\bm v}}_{ij}^n - \beta \Delta {\bm v}_{ij}^{y} \\
				\overline{p}_{ij}^n - \kappa^y \Delta p_{ij}^{y}
			\end{pmatrix}, \qquad 
			\widetilde{{\bm W}}^{n,-}_{i,j+\frac{1}{2}} = \begin{pmatrix}
				\tilde{\rho}_{i, j+\frac{1}{2}}^{n,-} \\
				\tilde{{\bm v}}_{i, j+\frac{1}{2}}^{n,-} \\
				\tilde{p}_{i, j+\frac{1}{2}}^{n,-}
			\end{pmatrix} = 
			\begin{pmatrix}
				\overline{\rho}_{ij}^n + \alpha^y \Delta \rho_{ij}^{y} \\
				\overline{{\bm v}}_{ij}^n + \beta \Delta {\bm v}_{ij}^{y} \\
				\overline{p}_{ij}^n + \kappa^y \Delta p_{ij}^{y}
			\end{pmatrix}.
		\end{aligned}
	\end{equation}
	If all the parameters $\{\alpha^{\star}, \beta, \kappa^{\star}\}$ are set to $1$, the scheme remains the original second-order method without modification. However, if all parameters are set to $0$, the scheme reduces to a first-order Lax--Friedrichs scheme, which is provably PP under a CFL condition for the Euler equations \cite{Zhang2010b}. 
	We aim to determine appropriate parameters such that the scheme is provably PP while maintaining second-order accuracy. 
	
	We first specify a fixed $q > 2$, which is closely related to the PP CFL condition for the updated cell averages. As the value of $q$ decreases, the BP limiter becomes more restrictive, but the corresponding BP CFL number increases. 
	For a given $q > 2$, the PP limiting  parameters in \eqref{2D_Euler_limiter} are computed via the following three steps:  
	
	\noindent
	{\bf Step 1:} Modify the density slopes and ensure density positivity:
		\begin{equation}\label{rho_2D}
			\alpha^x  =\begin{cases}
				\min\left(\frac{\overline{\rho}_{ij}^n}{\left| \Delta \rho_{ij}^x \right|(1+\epsilon)}, 1\right),   &  \mbox{if} ~  \Delta \rho_{ij}^x  \neq 0, \\
				1, & \mbox{otherwise},
			\end{cases} \quad
			\alpha^y  =\begin{cases}
				\min\left(\frac{\overline{\rho}_{ij}^n}{\left| \Delta \rho_{ij}^y \right|(1+\epsilon)}, 1\right),   &  \mbox{if} ~  \Delta \rho_{ij}^y  \neq 0, \\
				1, & \mbox{otherwise},
			\end{cases}
		\end{equation} 
		where $\epsilon$ is a small positive number and can be taken as $10^{-14}$. 
		
			\noindent
	{\bf Step 2:} Modify the pressure slopes and ensure pressure positivity:
		\begin{equation}\label{p_2D}
			\kappa^x   = \begin{cases}
				\min\left(\frac{\overline{p}_{ij}^n}{\left| \Delta p_{ij}^x \right|(1+\epsilon)}, 1\right),  &  \mbox{if} ~  \Delta p_{ij}^x  \neq 0 ,\\
				1, & \mbox{otherwise},
			\end{cases}  \quad
			\kappa^y  =\begin{cases}
				\min\left(\frac{ \overline{p}_{ij}^n}{\left| \Delta p_{ij}^y \right|(1+\epsilon)}, 1\right), & \mbox{if} ~  \Delta p_{ij}^y  \neq 0, \\
				1, & \mbox{otherwise}.
			\end{cases}
		\end{equation}
		
			\noindent
	{\bf Step 3:} Modify the velocity slopes:
		\begin{equation}\label{u_2D}
			\beta = \begin{cases}
				\min\left( \sqrt{\frac{(q - 2)^2  \overline{\rho}_{ij}^n  \overline{p}_{ij}^n}{(\gamma - 1) \left(  2  \left| C_x \alpha^x \Delta \rho_{ij}^x \Delta {\bm v}^x_{ij} + C_y \alpha^y \Delta \rho_{ij}^y \Delta {\bm v}_{ij}^y \right|^2  + (q - 2) (\overline{\rho}_{ij}^n)^2 \left( C_x \left|\Delta {\bm v}_{ij}^x \right|^2 + C_y \left|\Delta {\bm v}_{ij}^y \right|^2 \right) \right)}}, 1\right),  & \mbox{if} ~ \left| \Delta {\bm v}_{ij}^x \right| + \left| \Delta {\bm v}_{ij}^y \right| \neq 0, \\
				1, & \mbox{otherwise},
			\end{cases} 
		\end{equation}
		where the constants $C_x = \frac{\alpha_1 \Delta y}{ \alpha_1 \Delta y + \alpha_2 \Delta x }$ and $C_y = \frac{\alpha_2 \Delta x}{ \alpha_1 \Delta y + \alpha_2 \Delta x }$.

	In the above PP limiter, it is evident that \textbf{Step 1} and \textbf{Step 2} ensure the positivity of the limiting values for density and pressure. 
	\textbf{Step 3} is an additional limiter designed to guarantee condition \eqref{eq:cond2} in Theorem \ref{LimitingValuesPP}. This step, together with \textbf{Step 1} and \textbf{Step 2}, plays a crucial role in rigorously establishing the provably PP property of the updated cell averages at the next time step.
	
	Let $\widetilde{{\bm Q}}^{n,+}_{i-\frac{1}{2},j}, \widetilde{{\bm Q}}^{n,-}_{i+\frac{1}{2},j}, \widetilde{{\bm Q}}^{n,+}_{i,j-\frac{1}{2}}, \widetilde{{\bm Q}}^{n,-}_{i,j+\frac{1}{2}}$ denote the conservative variables corresponding to the modified limiting values of the primitive variables in \eqref{2D_Euler_limiter}. Specifically, we have the following properties. 
	
	\begin{theorem}\label{LimitingValuesPP}
		The conservative variables corresponding to the limiting values \eqref{2D_Euler_limiter} modified by the PP limiter satisfy the following two properties:
		\begin{align}
		& \qquad	\widetilde{{\bm Q}}^{n,+}_{i-\frac{1}{2},j}, \widetilde{{\bm Q}}^{n,-}_{i+\frac{1}{2},j}, \widetilde{{\bm Q}}^{n,+}_{i,j-\frac{1}{2}}, \widetilde{{\bm Q}}^{n,-}_{i,j+\frac{1}{2}} \in {\mathcal{G}}.
		\\
		\label{eq:cond2}
		&	q \overline{{\bm Q}}_{ij}^{n} \cdot \mathbf{n}_{*} \geq C_x \left( \widetilde{{\bm Q}}^{n, -}_{i+\frac{1}{2},j} + \widetilde{{\bm Q}}^{n, +}_{i-\frac{1}{2},j} \right) \cdot \mathbf{n}_{*} 
			+ C_y \left( \widetilde{{\bm Q}}^{n,+}_{i,j-\frac{1}{2}} + \widetilde{{\bm Q}}^{n, -}_{i,j+\frac{1}{2}} \right) \cdot \mathbf{n}_{*} \quad
			\forall \bm{v}_* \in \mathbb{R}^3,
		\end{align}
		where $\mathbf{n}_* = \left(\frac{\left| {\bm v}_* \right|^2}{2}, -{\bm v}_*, 1\right)^\top$, and $q > 2$ is the fixed parameter in {\bf Step 3} of the PP limiter.
	\end{theorem}
	
	\begin{proof}
		
		First, we verify that the modified limiting values satisfy:
		\begin{equation*}
			\tilde{\rho}_{i-\frac{1}{2}, j}^{n,+},\ \tilde{\rho}_{i+\frac{1}{2}, j}^{n,-},\ \tilde{\rho}_{i, j-\frac{1}{2}}^{n,+},\ \tilde{\rho}_{i, j+\frac{1}{2}}^{n,-} > 0, \qquad
			\tilde{p}_{i-\frac{1}{2}, j}^{n,+},\ \tilde{p}_{i+\frac{1}{2}, j}^{n,-},\ \tilde{p}_{i, j-\frac{1}{2}}^{n,+},\ \tilde{p}_{i, j+\frac{1}{2}}^{n,-} > 0.
		\end{equation*}
		This follows directly from the definitions in \eqref{2D_Euler_limiter} and the construction of the parameters in \textbf{Steps 1} and \textbf{2} of the PP limiter.
		
		Next, from the definition of $\beta$ in \eqref{u_2D}, we have
		\begin{equation}\label{eq:defd}
			\left( 2 \left| C_x \alpha^x \Delta \rho_{ij}^x \Delta \bm{v}_{ij}^x + C_y \alpha^y \Delta \rho_{ij}^y \Delta \bm{v}_{ij}^y \right|^2 
			+ (q - 2) (\overline{\rho}_{ij}^n)^2 \left( C_x \left| \Delta \bm{v}_{ij}^x \right|^2 
			+ C_y \left| \Delta \bm{v}_{ij}^y \right|^2 \right) \right) \beta^2 \leq \frac{(q - 2)^2 \overline{\rho}_{ij}^n \overline{p}_{ij}^n}{\gamma - 1}.
		\end{equation}
		We use the following relationships:
		\begin{equation}\label{limiterrelation}
			\left\{
			\begin{aligned}
				& \widetilde{\bm{m}}^{n,-}_{i+\frac{1}{2},j} = \tilde{\rho}^{n,-}_{i+\frac{1}{2},j} \, \widetilde{\bm{v}}^{n,-}_{i+\frac{1}{2},j} 
				= \left( \overline{\rho}_{ij}^n + \alpha^x \Delta \rho_{ij}^x \right) \left( \overline{\bm{v}}_{ij}^n + \beta \Delta \bm{v}_{ij}^x \right), \\
				& \widetilde{\bm{m}}^{n,+}_{i-\frac{1}{2},j} = \tilde{\rho}^{n,+}_{i-\frac{1}{2},j} \, \widetilde{\bm{v}}^{n,+}_{i-\frac{1}{2},j} 
				= \left( \overline{\rho}_{ij}^n - \alpha^x \Delta \rho_{ij}^x \right) \left( \overline{\bm{v}}_{ij}^n - \beta \Delta \bm{v}_{ij}^x \right), \\
				& \tilde{E}^{n,-}_{i+\frac{1}{2},j} = \dfrac{\tilde{p}^{n,-}_{i+\frac{1}{2},j}}{\gamma - 1} + \dfrac{1}{2} \tilde{\rho}^{n,-}_{i+\frac{1}{2},j} \left| \widetilde{\bm{v}}^{n,-}_{i+\frac{1}{2},j} \right|^2, \\
				& \tilde{E}^{n,+}_{i-\frac{1}{2},j} = \dfrac{\tilde{p}^{n,+}_{i-\frac{1}{2},j}}{\gamma - 1} + \dfrac{1}{2} \tilde{\rho}^{n,+}_{i-\frac{1}{2},j} \left| \widetilde{\bm{v}}^{n,+}_{i-\frac{1}{2},j} \right|^2.
			\end{aligned}
			\right.
		\end{equation}
		Substituting \eqref{limiterrelation} into the inequality \eqref{eq:defd}, and after simplification, we derive
		\begin{equation}\label{quadratic}
			\begin{aligned}
				& \left| C_x \left( \widetilde{\bm{m}}^{n,-}_{i+\frac{1}{2},j} + \widetilde{\bm{m}}^{n,+}_{i-\frac{1}{2},j} \right)
				+ C_y \left( \widetilde{\bm{m}}^{n,-}_{i,j+\frac{1}{2}} + \widetilde{\bm{m}}^{n,+}_{i,j-\frac{1}{2}} \right)
				- q \overline{\bm{m}}_{ij}^n \right|^2 \\
				& \quad - 2 (q - 2) \overline{\rho}_{ij}^n \left( q \overline{E}_{ij}^n 
				- \left( C_x \left( \tilde{E}^{n,-}_{i+\frac{1}{2},j} + \tilde{E}^{n,+}_{i-\frac{1}{2},j} \right)
				+ C_y \left( \tilde{E}^{n,-}_{i,j+\frac{1}{2}} + \tilde{E}^{n,+}_{i,j-\frac{1}{2}} \right) \right) \right) \leq 0.
			\end{aligned}
		\end{equation}
		This inequality implies that the discriminant of the following quadratic function $\Phi(|\bm{v}_*|)$ is non-positive, and thus $\Phi(|\bm{v}_*|) \geq 0$ for any $\bm{v}_* \in \mathbb{R}^3$, where
		\begin{equation*}
			\begin{aligned}
				\Phi(|\bm{v}_*|) := & \dfrac{1}{2}(q - 2) \overline{\rho}_{ij}^n \left| \bm{v}_* \right|^2 
				-  \left| C_x \left( \widetilde{\bm{m}}^{n,-}_{i+\frac{1}{2},j} + \widetilde{\bm{m}}^{n,+}_{i-\frac{1}{2},j} \right)
				+ C_y \left( \widetilde{\bm{m}}^{n,-}_{i,j+\frac{1}{2}} + \widetilde{\bm{m}}^{n,+}_{i,j-\frac{1}{2}} \right)
				- q \overline{\bm{m}}_{ij}^n \right| \left| \bm{v}_* \right|  \\
				& + q \overline{E}_{ij}^n 
				- \left( C_x \left( \tilde{E}^{n,-}_{i+\frac{1}{2},j} + \tilde{E}^{n,+}_{i-\frac{1}{2},j} \right)
				+ C_y \left( \tilde{E}^{n,-}_{i,j+\frac{1}{2}} + \tilde{E}^{n,+}_{i,j-\frac{1}{2}} \right) \right).
			\end{aligned}
		\end{equation*}
		Therefore, for any $\bm{v}_* \in \mathbb{R}^3$, we have
		\begin{equation*}
			\begin{aligned}
				& q \overline{\bm{Q}}_{ij}^n \cdot \mathbf{n}_* - C_x \left( \widetilde{\bm{Q}}^{n,-}_{i+\frac{1}{2},j} + \widetilde{\bm{Q}}^{n,+}_{i-\frac{1}{2},j} \right) \cdot \mathbf{n}_* 
				- C_y \left( \widetilde{\bm{Q}}^{n,+}_{i,j-\frac{1}{2}} + \widetilde{\bm{Q}}^{n,-}_{i,j+\frac{1}{2}} \right) \cdot \mathbf{n}_* \\
				& = \dfrac{1}{2}(q - 2) \overline{\rho}_{ij}^n \left| \bm{v}_* \right|^2 
				+ \left( C_x \left( \widetilde{\bm{m}}^{n,-}_{i+\frac{1}{2},j} + \widetilde{\bm{m}}^{n,+}_{i-\frac{1}{2},j} \right)
				+ C_y \left( \widetilde{\bm{m}}^{n,-}_{i,j+\frac{1}{2}} + \widetilde{\bm{m}}^{n,+}_{i,j-\frac{1}{2}} \right)
				- q \overline{\bm{m}}_{ij}^n \right) \cdot \bm{v}_* \\
				& \quad + q \overline{E}_{ij}^n 
				- \left( C_x \left( \tilde{E}^{n,-}_{i+\frac{1}{2},j} + \tilde{E}^{n,+}_{i-\frac{1}{2},j} \right)
				+ C_y \left( \tilde{E}^{n,-}_{i,j+\frac{1}{2}} + \tilde{E}^{n,+}_{i,j-\frac{1}{2}} \right) \right) \\
				& \geq \Phi\left( \left| \bm{v}_* \right| \right) \geq 0.
			\end{aligned}
		\end{equation*}
		This establishes the inequality \eqref{eq:cond2}, completing the proof.
	\end{proof}

	\begin{remark}
		The initial idea of introducing slope parameters to ensure positive limiting values was inspired by Berthon’s work in \cite{Berthon2005}. However, unlike Berthon, we do not depend on a fixed decomposition of cell average values. Instead, our analysis is grounded in the GQL approach \cite{WuShu2021GQL}, which offers a more flexible and generalized framework. This approach allows us to explicitly establish the relationship between parameter selection, the PP property, and the CFL condition.
	\end{remark}
	
	\begin{remark}
		The property \eqref{eq:cond2} in Theorem \ref{LimitingValuesPP} implies that there exists a state $\widehat{\bm{Q}}_{ij} \in \overline{\mathcal{G}}$ such that
		\begin{equation}\label{decompositionform}
			\overline{\bm{Q}}_{ij}^n = \dfrac{C_x}{q} \left( \widetilde{\bm{Q}}^{n,-}_{i+\frac{1}{2},j} + \widetilde{\bm{Q}}^{n,+}_{i-\frac{1}{2},j} \right) 
			+ \dfrac{C_y}{q} \left( \widetilde{\bm{Q}}^{n,-}_{i,j+\frac{1}{2}} + \widetilde{\bm{Q}}^{n,+}_{i,j-\frac{1}{2}} \right) 
			+ \dfrac{q - 2}{q} \widehat{\bm{Q}}_{ij}.
		\end{equation}
	\end{remark}

	If the exact solution of density and pressure has a uniform positive lower bound, we can prove that the PP limiter does not degrade the second-order accuracy of the scheme, as established in the following theorem.

	\begin{theorem}\label{accuracy}
		Assume the exact density $\rho$, pressure $p$, and velocity $\bm{v} = \left(v_1, v_2, v_3\right)^\top$ are $C^2$ functions in the computational domain $\Omega$, and that there exists a constant $C > 0$ such that
		$$\rho(x,y,t) > C>0, \quad p(x,y,t) > C>0, \qquad \forall (x,y)\in \Omega,~~ \forall t > 0.$$
		Let the reconstructed point values be second-order accurate approximations to the exact solution, i.e.,
		\begin{align*}
			\left|\rho^{n,\mp}_{i \pm \frac{1}{2},j} - \rho(x_{i \pm \frac{1}{2}}, y_j, t_n) \right| &\le \mathcal{O}(h^2), \qquad \left|\rho^{n,\mp}_{i,j \pm \frac{1}{2}} - \rho(x_{i}, y_{j\pm \frac{1}{2}}, t_n)\right| \le \mathcal{O}(h^2), \\
			\left|\bm{v}^{n,\mp}_{i\pm\frac{1}{2},j} - \bm{v}(x_{i \pm \frac{1}{2}}, y_j, t_n)\right| &\le \mathcal{O}(h^2), \qquad \left|\bm{v}^{n,\mp}_{i,j\pm\frac{1}{2}} - \bm{v}(x_{i}, y_{ j \pm \frac{1}{2}}, t_n)\right| \le \mathcal{O}(h^2), \\
			\left | p^{n,\mp}_{i \pm \frac{1}{2},j} - p(x_{i \pm \frac{1}{2}}, y_j, t_n) \right| & \le \mathcal{O}(h^2), \qquad \left| p^{n,\mp}_{i,j \pm \frac{1}{2}} - p(x_{i}, y_{j\pm \frac{1}{2}}, t_n) \right| \le \mathcal{O}(h^2),
		\end{align*}
		where $h = \max(\Delta x, \Delta y)$. Under this assumption, the point values modified by the PP limiter \eqref{rho_2D}--\eqref{u_2D} will still be second-order accurate approximations of the exact solution. Specifically, we have
		\begin{align}\label{ddd}
			\left| \tilde{\rho}^{n,\mp}_{i \pm \frac{1}{2},j} - \rho(x_{i \pm \frac{1}{2}}, y_j, t_n) \right| & \le \mathcal{O}(h^2), \qquad \left| \tilde{\rho}^{n,\mp}_{i,j \pm \frac{1}{2}} - \rho(x_{i}, y_{j\pm \frac{1}{2}}, t_n) \right| \le \mathcal{O}(h^2), \\ \label{vvv}
			\left|\tilde{\bm{v}}^{n,\mp}_{i\pm\frac{1}{2},j} - \bm{v}(x_{i \pm \frac{1}{2}}, y_j, t_n)\right| & \le \mathcal{O}(h^2), \qquad \left|\tilde{\bm{v}}^{n,\mp}_{i,j\pm\frac{1}{2}} - \bm{v}(x_{i}, y_{ j \pm \frac{1}{2}}, t_n)\right| \le \mathcal{O}(h^2), \\ \label{ppp}
			\left| \tilde{p}^{n,\mp}_{i \pm \frac{1}{2},j} - p(x_{i \pm \frac{1}{2}}, y_j, t_n) \right| & \le \mathcal{O}(h^2), \qquad \left| \tilde{p}^{n,\mp}_{i,j \pm \frac{1}{2}} - p(x_{i}, y_{j\pm \frac{1}{2}}, t_n) \right| \le \mathcal{O}(h^2).
		\end{align}
		This confirms that the PP limiter maintains the second-order accuracy of the reconstruction.
	\end{theorem}

	\begin{proof}
		We will focus on the proof for the limiting values at $(x_{i \pm \frac{1}{2}}, y_j, t_n)$, while the proof for the limiting values at $(x_{i}, y_{j\pm \frac{1}{2}}, t_n)$ is similar and thus omitted. Using the triangle inequality gives 
		\begin{equation}\label{Therm3.6Eq1}
			\begin{aligned}
				\left| \tilde{\rho}^{n,\mp}_{i \pm \frac{1}{2},j} - \rho(x_{i \pm \frac{1}{2}}, y_j, t_n) \right| 
				&= \left|\overline{\rho}_{ij}^n \pm \alpha^x \Delta \rho^x_{ij} -  \rho(x_{i \pm \frac{1}{2}}, y_j, t_n) \right| \\
				&= \left|\overline{\rho}_{ij}^n \pm  \Delta \rho^x_{ij} - \rho(x_{i \pm \frac{1}{2}}, y_j, t_n) \mp \Delta \rho^x_{ij} \pm \alpha^x   \Delta \rho^x_{ij} \right| \\
				&\leq \left| \rho^{n,\mp}_{i \pm \frac{1}{2},j} - \rho(x_{i \pm \frac{1}{2}}, y_j, t_n) \right| + (1 - \alpha^x) \left|\Delta \rho^x_{ij} \right|
				\\
				& \le {\cal O}(h^2) + (1 - \alpha^x) \left|\Delta \rho^x_{ij} \right|,
			\end{aligned}
		\end{equation}
		There exists a constant $C_1 > 0$ such that
		\begin{equation*}
			\overline{\rho}_{ij}^n - \left|\Delta \rho^x_{ij} \right| 
			=\min \left( \rho^{n,-}_{i + \frac{1}{2},j}, \rho^{n,+}_{i - \frac{1}{2},j}  \right)
			\ge m - C_1 h^2,
		\end{equation*}
		where $m = \min\left( \rho(x_{i - \frac{1}{2}}, y_j, t_n), \rho(x_{i + \frac{1}{2}}, y_j, t_n) \right)>0$. 
		This along with $\epsilon=10^{-14} \ll h^2$ implies: if $\alpha^x = \frac{\overline{\rho}_{ij}^n}{\left|\Delta \rho^x_{ij}\right| (1+\epsilon)  } < 1$ then 
		\begin{equation*}
			0\le (1 - \alpha^x) \left|\Delta \rho^x_{ij} \right| = \left|\Delta \rho^x_{ij}\right| - \overline{\rho}_{ij}^n 
			+  \overline{\rho}_{ij}^n \frac{\epsilon}{1+\epsilon} \le C_1 h^2 - m + \overline{\rho}_{ij}^n \frac{\epsilon}{1+\epsilon} < C_2 h^2.
		\end{equation*}
		Substituting this estimate into \eqref{Therm3.6Eq1}, we obtain 
		\begin{equation*}
			\left| \tilde{\rho}^{n,\mp}_{i \pm \frac{1}{2},j} - \rho(x_{i \pm \frac{1}{2}}, y_j, t_n) \right| \leq \mathcal{O}(h^2).
		\end{equation*}
		Similarly, we have 
		\begin{equation*}
			\left| \tilde{p}^{n,\mp}_{i \pm \frac{1}{2},j} - p(x_{i \pm \frac{1}{2}}, y_j, t_n) \right| \leq \mathcal{O}(h^2).
		\end{equation*}
		This completes the accuracy proof for density and pressure in \eqref{ddd} and \eqref{ppp}. 
		For velocity, similar to \eqref{Therm3.6Eq1}, we only need to show:
		\begin{equation*}
			1 - \beta \leq \mathcal{O}(h^2). 
		\end{equation*}
		Denote the modified conservative values after Step 1 and Step 2 of the PP limiter by $\check{\bm{Q}}^{n,-}_{i+\frac{1}{2},j}, \check{\bm{Q}}^{n,+}_{i-\frac{1}{2},j}, \check{\bm{Q}}^{n,-}_{i,j+\frac{1}{2}}, \check{\bm{Q}}^{n,+}_{i,j-\frac{1}{2}}$. 
		Define $\check{\bm{Q}}_{ij}$ by 
		\begin{equation}\label{decomposition_beta1}
			\dfrac{q - 2}{q} \check{\bm{Q}}_{ij}	: = \dfrac{C_x}{q} \left( \check{\bm{Q}}^{n,-}_{i+\frac{1}{2},j} + \check{\bm{Q}}^{n,+}_{i-\frac{1}{2},j} \right) 
			+ \dfrac{C_y}{q} \left( \check{\bm{Q}}^{n,-}_{i,j+\frac{1}{2}} + \check{\bm{Q}}^{n,+}_{i,j-\frac{1}{2}} \right) - \overline{\bm{Q}}_{ij}^n.
		\end{equation}
		According to the second-order accuracy of the modified density and pressure proven above,  we have
		\begin{equation*}
			\check{\bm{Q}}^{n,-}_{i+\frac{1}{2},j} = \bm{Q}(x_{i+\frac{1}{2}}, y_{j}, t_n) + \mathcal{O}(h^2), \qquad \check{\bm{Q}}^{n,+}_{i-\frac{1}{2},j} = \bm{Q}(x_{i-\frac{1}{2}}, y_{j}, t_n) + \mathcal{O}(h^2).
		\end{equation*}
		This implies 
		\begin{equation*}
			\frac{1}{2} \left( \check{\bm{Q}}^{n,-}_{i+\frac{1}{2},j} + \check{\bm{Q}}^{n,+}_{i-\frac{1}{2},j} \right) = \frac{1}{2} \left( \bm{Q}(x_{i+\frac{1}{2}}, y_{j}, t_n) + \bm{Q}(x_{i-\frac{1}{2}}, y_{j}, t_n) \right) + \mathcal{O}(h^2) 
			= {\bm{Q}}^e_{ij} + \mathcal{O}(h^2),
		\end{equation*}
		where ${\bm{Q}}^e_{ij} =  \bm{Q}(x_i, y_j, t_n)$ denotes the exact solution at the cell center.  Similarly, 
		\begin{equation*}
			\frac{1}{2} \left( \check{\bm{Q}}^{n,-}_{i,j+\frac{1}{2}} + \check{\bm{Q}}^{n,+}_{i,j-\frac{1}{2}} \right) = {\bm{Q}}^e_{ij} + \mathcal{O}(h^2).
		\end{equation*}
		According to \eqref{decomposition_beta1}, we obtain
		\begin{equation*}
			\frac{q - 2}{q } \check{\bm{Q}}_{ij} =  \left( 1 - \frac{2 (C_x + C_y) }{q} \right) {\bm{Q}}^e_{ij} + \mathcal{O}(h^2).
		\end{equation*}
		This gives the following approximation relation
		\begin{equation*}
			\check{\bm{Q}}_{ij} = {\bm{Q}}^e_{ij} + \mathcal{O}(h^2).
		\end{equation*}
		Let $\check{p}_{ij}$ be the pressure of $\check{\bm{Q}}_{ij}$ and ${p}^e_{ij}$ represent the pressure of $\overline{\bm{Q}}^e_{ij}$, then we have 
		\begin{equation*}
			\check{p}_{ij} = {p}^e_{ij}+ \mathcal{O}(h^2).
		\end{equation*}
		\textbf{Case 1:} If $\check{p}_{ij} \ge 0$, which means $\check{\bm{Q}}_{ij} \in \overline{\mathcal{G}}$, then Step 3 of the PP limiter is not activated, i.e., $\beta = 1$. 
		\\
		\textbf{Case 2:} If $\check{p}_{ij} < 0$, combined with the fact that ${p}^e_{ij}>0$, we conclude that there exists a constant $C_3 > 0$ such that
		\begin{equation*} -C_3 h^2 \le \check{p}_{ij} < 0.
		\end{equation*}	
		Substituting the relationships from \eqref{limiterrelation} into \eqref{decomposition_beta1} gives
		\begin{equation}\label{convex_equation}
			\begin{aligned}
				2 \left| C_x \alpha^x \Delta \rho_{ij}^x \Delta \bm{v}^x_{ij} + C_y \alpha^y \Delta \rho_{ij}^y \Delta \bm{v}_{ij}^y \right|^2 + (q - 2) (\overline{\rho}_{ij}^n)^2 \left( C_x \left|\Delta \bm{v}_{ij}^x \right|^2 + C_y \left|\Delta \bm{v}_{ij}^y \right|^2 \right) 
				= \frac{(q - 2)^2 \overline{\rho}_{ij}^n (\overline{p}_{ij}^n - \check{p}_{ij})}{\gamma - 1}.
			\end{aligned}
		\end{equation}
		By the expression for $\beta$ in \eqref{u_2D}, we have
		\begin{equation}\label{beta_equation}
			\begin{aligned}
				\left( 2 \left| C_x \alpha^x \Delta \rho_{ij}^x \Delta \bm{v}^x_{ij} + C_y \alpha^y \Delta \rho_{ij}^y \Delta \bm{v}_{ij}^y \right|^2 + (q - 2) (\overline{\rho}_{ij}^n)^2 \left( C_x \left|\Delta \bm{v}_{ij}^x \right|^2 + C_y \left|\Delta \bm{v}_{ij}^y \right|^2 \right) \right) \beta^2 
				= \frac{(q - 2)^2 \overline{\rho}_{ij}^n \overline{p}_{ij}^n}{\gamma - 1}.
			\end{aligned}
		\end{equation}
		Combining \eqref{beta_equation} and \eqref{convex_equation}, we obtain 
		\begin{equation*}
			\beta = \sqrt{ \dfrac{\overline{p}^n_{ij}}{\overline{p}^n_{ij} - \check{p}_{ij}} } = \sqrt{ \dfrac{1}{1 - \frac{\check{p}_{ij}}{\overline{p}^n_{ij} } } }.
		\end{equation*}
		Note that $\overline{p}_{ij}^n>0$ and 
		\begin{equation*}
			\overline{p}_{ij}^n  = p(x_i, y_j, t_n) +  \mathcal{O}(h^2) \ge C + \mathcal{O}(h^2). 
		\end{equation*}
		Thus, we have
		\begin{equation*}
			\beta  = \sqrt{ \dfrac{1}{1 - \frac{\check{p}_{ij}}{\overline{p}^n_{ij} } } } \ge \sqrt{\dfrac{1}{1 + \frac{C_3 h^2}{C + \mathcal{O}(h^2)}}} \ge 1 - \mathcal{O}(h^2), 
		\end{equation*}
		where the last inequality follows from the Taylor expansion.  
		By considering both \textbf{Case 1} and \textbf{Case 2}, we have completed the proof.
	\end{proof}

	\subsubsection{Proof of Positivity-Preserving Property for Updated Cell Averages}
	We now employ the GQL approach \cite{WuShu2021GQL} to theoretically prove that our scheme, with the proposed PP limiter, ensures that the updated cell averages $\overline{{\bm Q}}_{ij}^{n+1} \in \mathcal{G}$ for all $i$ and $j$. This guarantees the PP property at the next time step.
	
	For simplicity, we focus on the PP finite volume method with forward Euler time discretization. However, the PP analysis and conclusions remain valid for higher-order SSP Runge--Kutta methods, which are convex combinations of forward Euler steps. Replacing the limiting values in \eqref{NumericalFlux} with the PP-modified values \eqref{2D_Euler_limiter}, the resulting finite volume method is written as
	\begin{equation}\label{PPForwardEuler}
\begin{aligned}
			\overline{{\bm Q}}_{ij}^{n+1} = & \overline{{\bm Q}}_{ij}^{n} -  \lambda_1 \left( \mathbf{\hat{G}}_{1} \left( \widetilde{{\bm Q}}^{n,-}_{i+\frac{1}{2}, j}, \widetilde{{\bm Q}}^{n,+}_{i+\frac{1}{2}, j}\right)  - \mathbf{\hat{G}}_{1}\left(\widetilde{{\bm Q}}^{n,-}_{i-\frac{1}{2}, j}, \widetilde{{\bm Q}}^{n,+}_{i-\frac{1}{2}, j}\right)  \right) 
			\\
		& - \lambda_2 \left(\mathbf{\hat{G}}_{2}\left( \widetilde{{\bm Q}}^{n,-}_{i, j+\frac{1}{2}}, \widetilde{{\bm Q}}^{n,+}_{i, j+\frac{1}{2}} \right)  - \mathbf{\hat{G}}_{2}\left( \widetilde{{\bm Q}}^{n,-}_{i, j-\frac{1}{2}}, \widetilde{{\bm Q}}^{n,+}_{i, j-\frac{1}{2}} \right)  \right),
\end{aligned}
	\end{equation}
	where $\lambda_1 = \frac{\Delta t}{\Delta x}$ and $\lambda_2 = \frac{\Delta t}{\Delta y}$. The quantities $\widetilde{{\bm Q}}^{n, \pm}$ are the conservative variables associated with the PP-limited primitive variables $\widetilde{{\bm W}}^{n, \pm}$, and $\mathbf{G}_1$ and $\mathbf{G}_2$ represent the Lax--Friedrichs numerical fluxes in the $x$ and $y$ directions, respectively.
	
	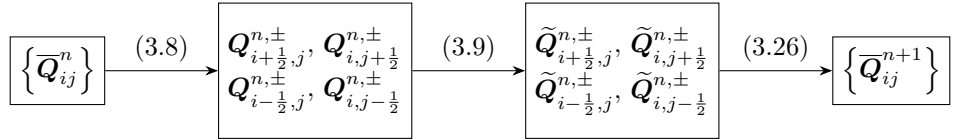
\begin{figure}[ht]
		\centering
		\begin{tikzpicture}[
			->, 
			>=Stealth, 
			node distance=1.5cm, 
			every node/.style={align=center}, 
			]
			\node (A) [draw, rectangle, minimum width=0.9cm, minimum height=0.9cm] {$\left\{\overline{{\bm Q}}_{ij}^{n}\right\}$};
			\node (B) [right=of A, draw, rectangle, minimum width=2.5cm, minimum height=1.8cm] {${\bm Q}_{i+\frac{1}{2},j}^{n,\pm}$, ${\bm Q}_{i,j+\frac{1}{2}}^{n,\pm}$\\
				${\bm Q}_{i-\frac{1}{2},j}^{n,\pm}$, ${\bm Q}_{i,j-\frac{1}{2}}^{n,\pm}$};
			\node (C) [right=1.5 of B, draw, rectangle, minimum width=2.5cm, minimum height=1.8cm] {$\widetilde{{\bm Q}}^{n,\pm}_{i+\frac{1}{2},j}$, $\widetilde{{\bm Q}}^{n,\pm}_{i,j+\frac{1}{2}}$\\
				$\widetilde{{\bm Q}}^{n,\pm}_{i-\frac{1}{2},j}$, $\widetilde{{\bm Q}}^{n,\pm}_{i,j-\frac{1}{2}}$};
			\node (D) [right=1.5 of C, draw, rectangle, minimum width=0.9cm, minimum height=0.9cm] {$\left\{\overline{{\bm Q}}_{ij}^{n+1}\right\}$};    
			\draw[->] (A) -- (B) node[midway, above] {\eqref{deltaXY}}; 
			\draw[->] (B) -- (C) node[midway, above] {\eqref{2D_Euler_limiter}}; 
			\draw[->] (C) -- (D) node[midway, above] {\eqref{PPForwardEuler}}; 
		\end{tikzpicture}
		\caption{Illustration of the provably PP finite volume method with forward Euler time discretization.}
		\label{fig:flowchartFVM}
	\end{figure}
	
	The overall procedure of the PP finite volume method \eqref{PPForwardEuler} is illustrated in \Cref{fig:flowchartFVM}. We now present the theoretical proof of its PP property.
	
	\begin{lemma}\label{1DLFscheme}
		For any $\bm{Q}_1, \bm{Q}_2, \bm{Q}_3, \bm{Q}_4 \in \mathcal{G}$ and for ${\bm n} = {\bm n}_1$ or ${\bm n}_*$ with any auxiliary variable ${\bm v}_* \in \mathbb{R}^3$, the following inequality holds:
		\begin{equation}\label{eq:3ddfd}
			- \left( \hat{\bm{G}}_\ell(\bm{Q}_3, \bm{Q}_4) - \hat{\bm{G}}_\ell(\bm{Q}_1, \bm{Q}_2) \right) \cdot {\bm n} > - \alphaLF_\ell \left( \bm{Q}_2 + \bm{Q}_3 \right) \cdot {\bm n}, \qquad \ell = 1, 2,
		\end{equation}
		where $\hat{\bm{G}}_\ell$ represents the Lax--Friedrichs flux, and $\alphaLF_\ell$ is the numerical viscosity coefficient.
	\end{lemma}
	
	\begin{proof}
		It was shown in \cite{ShuChiwang1996PPforEuler} that the first-order finite volume scheme using the Lax--Friedrichs flux is PP under the CFL condition $0 < \lambda_\ell \alphaLF_\ell \leq 1$. Specifically, for $\bm{Q}_1, \bm{Q}_2, \bm{Q}_3 \in \mathcal{G}$, the scheme ensures
		\[
		\bm{Q}_2 - \lambda_\ell \left( \hat{\bm{G}}_\ell(\bm{Q}_2, \bm{Q}_3) - \hat{\bm{G}}_\ell(\bm{Q}_1, \bm{Q}_2) \right) \in \mathcal{G}, \qquad \ell = 1, 2.
		\]
		According to Lemma \ref{Lemma:GQLrepresentation}, we know that
		\[
		\bm{Q}_2 \cdot {\bm n} - \lambda_\ell \left( \hat{\bm{G}}_\ell(\bm{Q}_2, \bm{Q}_3) - \hat{\bm{G}}_\ell(\bm{Q}_1, \bm{Q}_2) \right) \cdot {\bm n} > 0,
		\]
		which implies
		\begin{equation}\label{LFeq1}
			- \left( \hat{\bm{G}}_\ell(\bm{Q}_2, \bm{Q}_3) - \hat{\bm{G}}_\ell(\bm{Q}_1, \bm{Q}_2) \right) \cdot {\bm n} > - \frac{1}{\lambda_\ell} \bm{Q}_2 \cdot {\bm n} \geq -\alphaLF_\ell \bm{Q}_2 \cdot {\bm n}.
		\end{equation}
		Similarly, we can show
		\begin{equation}\label{LFeq2}
			- \left( \hat{\bm{G}}_\ell(\bm{Q}_3, \bm{Q}_4) - \hat{\bm{G}}_\ell(\bm{Q}_2, \bm{Q}_3) \right) \cdot {\bm n} > - \frac{1}{\lambda_\ell} \bm{Q}_3 \cdot {\bm n} \geq -\alphaLF_\ell \bm{Q}_3 \cdot {\bm n}.
		\end{equation}
		Combining \eqref{LFeq1} and \eqref{LFeq2} yields
		\[
		- \left( \hat{\bm{G}}_\ell(\bm{Q}_3, \bm{Q}_4) - \hat{\bm{G}}_\ell(\bm{Q}_1, \bm{Q}_2) \right) \cdot {\bm n} > -\alphaLF_\ell \left( \bm{Q}_2 + \bm{Q}_3 \right) \cdot {\bm n}.
		\]
		This completes the proof of \eqref{eq:3ddfd}.
	\end{proof}
	
	\begin{theorem}\label{PPTheorem}
		If $\overline{{\bm Q}}^n_{ij} \in \mathcal{G}$ for all $i$ and $j$, then $\overline{{\bm Q}}^{n+1}_{ij}$ obtained by \eqref{PPForwardEuler} always belongs to $\mathcal{G}$ under the CFL condition:
		\begin{equation}\label{CFLcondtion}
			\Delta t \left( \frac{\alpha_1}{\Delta x} + \frac{\alpha_2}{\Delta y}  \right)  \le \frac{1}{q},
		\end{equation}
		where $q>2$ is a fixed constant used in the PP limiter and can be arbitrarily specified in $(2, +\infty)$. 
	\end{theorem}
	
	\begin{proof}
		By the scheme \eqref{PPForwardEuler} and \Cref{1DLFscheme}, we have
		\begin{equation*}
			\begin{aligned}
				q \overline{{\bm Q}}_{ij}^{n+1} \cdot \bm{n}_* &= 
				q \overline{\bm{Q}}^n_{ij} \cdot \bm{n}_* -  q \lambda_1 \left( \mathbf{\hat{G}}_{1} \left( \widetilde{{\bm Q}}^{n,-}_{i+\frac{1}{2}, j}, \widetilde{{\bm Q}}^{n,+}_{i+\frac{1}{2}, j}\right)  - \mathbf{\hat{G}}_{1}\left(\widetilde{{\bm Q}}^{n,-}_{i-\frac{1}{2}, j}, \widetilde{{\bm Q}}^{n,+}_{i-\frac{1}{2}, j}\right)  \right) \cdot \bm{n}_* \\ 
				&\quad - q \lambda_2 \left(\mathbf{\hat{G}}_{2}\left( \widetilde{{\bm Q}}^{n,-}_{i, j+\frac{1}{2}}, \widetilde{{\bm Q}}^{n,+}_{i, j+\frac{1}{2}} \right)  - \mathbf{\hat{G}}_{2}\left( \widetilde{{\bm Q}}^{n,-}_{i, j-\frac{1}{2}}, \widetilde{{\bm Q}}^{n,+}_{i, j-\frac{1}{2}} \right)  \right) \cdot \bm{n}_* \\
				& > q \overline{\bm{Q}}^n_{ij} \cdot \bm{n}_* - q \lambda_1 \alpha_1 \left(\widetilde{{\bm Q}}^{n,-}_{i+\frac{1}{2}, j}+ \widetilde{{\bm Q}}^{n,+}_{i-\frac{1}{2}, j}  \right) \cdot \bm{n}_* \\
				&\quad - q \lambda_2 \alpha_2 \left( \widetilde{{\bm Q}}^{n,-}_{i, j+\frac{1}{2}} + \widetilde{{\bm Q}}^{n,+}_{i, j-\frac{1}{2}}  \right) \cdot \bm{n}_* \\
				& \geq q \overline{\bm{Q}}^n_{ij} \cdot \bm{n}_* -   C_x \left( \widetilde{{\bm Q}}^{n, -}_{i+\frac{1}{2},j} + \widetilde{{\bm Q}}^{n, +}_{i-\frac{1}{2},j} \right) \cdot \mathbf{n}_{*} \\
				&\quad - C_y \left( \widetilde{{\bm Q}}^{n,+}_{i,j-\frac{1}{2}} + \widetilde{{\bm Q}}^{n, -}_{i,j+\frac{1}{2}} \right) \cdot \mathbf{n}_{*},
			\end{aligned}
		\end{equation*}
		where the last inequality follows from the CFL condition \eqref{CFLcondtion}. Using the result in \Cref{LimitingValuesPP}, we obtain 
		\begin{equation*}
			\overline{{\bm Q}}_{ij}^{n+1} \cdot \bm{n}_* > 0.
		\end{equation*}
		Similarly,  
		\begin{equation*}
			\begin{aligned}
				\overline{\bm{Q}}_{ij}^{n+1} \cdot \bm{n}_1 &> \overline{\bm{Q}}^n_{ij} \cdot \bm{n}_1 -   \frac{C_x}{q} \left( \widetilde{{\bm Q}}^{n, -}_{i+\frac{1}{2},j} + \widetilde{{\bm Q}}^{n, +}_{i-\frac{1}{2},j} \right) \cdot \mathbf{n}_{1}-  \frac{C_y}{q} \left( \widetilde{{\bm Q}}^{n,-}_{i, j+\frac{1}{2}} + \widetilde{{\bm Q}}^{n,+}_{i, j-\frac{1}{2}}  \right) \cdot \bm{n}_1 \\ 
				&= \frac{q - 2}{q} \overline{\rho}^n_{ij} > 0.
			\end{aligned}
		\end{equation*}
		Thanks to the GQL representation, we have $\overline{\bm{Q}}_{ij}^{n+1} \in  {\mathcal G}_* = {\mathcal G}$. 
		The proof is completed.
	\end{proof}
	
	Denote the PP finite volume method \eqref{PPForwardEuler} with the forward Euler time discretization as 
	\begin{equation}\label{ForwardEuler2}
		\overline{\bm{Q}}_{ij}^{n+1} = \overline{\bm{Q}}_{ij}^{n} + \Delta t \, \widetilde{ \mathcal{L}}_{ij}(\overline{\bm{Q}}^n). 
	\end{equation}
	Then the PP finite volume method \eqref{PPForwardEuler} with the second-order SSP Runge--Kutta time stepping can be written as 
	\begin{equation}\label{SSPRK2new}
		\begin{aligned}
			\overline{\bm{Q}}_{ij}^{(1)} &= \overline{\bm{Q}}_{ij}^{n} + \Delta t \, \widetilde{ \mathcal{L}}_{ij}(\overline{\bm{Q}}^n), \\
			\overline{\bm{Q}}_{ij}^{n+1} &= \frac{1}{2} \overline{\bm{Q}}_{ij}^{n} + \frac{1}{2} \left( \overline{\bm{Q}}_{ij}^{(1)} + \Delta t \, \widetilde{ \mathcal{L}}_{ij}(\overline{\bm{Q}}^{(1)}) \right), 
		\end{aligned}
	\end{equation}
	which is a convex combination of two forward Euler steps and thus retains the PP property.

	\subsection{Finite Difference CT Method $S_{B}^{\Delta t}$ for System-B}\label{section:operator2}
	
	In this subsection, we present our second-order finite difference CT method for System-B \eqref{Subsystem} on non-staggered Cartesian grids. Since both the density and internal energy remain unchanged in System-B, the main task is to solve the equations:
	\begin{equation}\label{FDequations}
		\left\{
		\begin{aligned}
			\rho \partial_t {\bm v} &= - {\bm B} \times \operatorname{curl} {\bm B}, \\
			\partial_t {\bm B} &= \operatorname{curl}({\bm v} \times {\bm B}),
		\end{aligned}
		\right.
	\end{equation}
	and then update the mechanical energy $E$.
	
	\subsubsection{Finite Difference CT Method and Its Properties}

	To maintain the discrete divergence-free (DDF) condition for the magnetic field, a key feature of the CT method is the introduction of the electric field when discretizing the induction equation. For clarity, we denote ${\bm B} = (B^x, B^y, B^z)^\top$, ${\bm v} = (v^x, v^y, v^z)^\top$, and define the electric field $\mathbf{\Omega} = (\Omega^x, \Omega^y, \Omega^z)^\top$ by
	\begin{equation}\label{def:Omega}
		\mathbf{\Omega}  =	{\bm B} \times {\bm v} = 
		\begin{pmatrix}
			B^y v^z - B^z v^y \\
			B^z v^x - B^x v^z \\
			B^x v^y - B^y v^x
		\end{pmatrix}.
	\end{equation}
	In the 2D case, the induction equations \eqref{FDequations} can be reformulated as
	\begin{equation*}
		\frac{\partial }{\partial t}  \begin{pmatrix}
			B^x \\
			B^y \\
			B^z
		\end{pmatrix} = 
		\begin{pmatrix}
			-\frac{\partial \Omega^z}{\partial y} \\
			\frac{\partial \Omega^z}{\partial x} \\
			\frac{\partial \Omega^x}{\partial y} - \frac{\partial \Omega^y}{\partial x}
		\end{pmatrix}.
	\end{equation*}

	In traditional CT methods, either staggered meshes are used, or more sophisticated shock-capturing techniques are applied to compute the electric field. However, in this paper, we design an implicit method for time discretization, which significantly enhances stability while avoiding complex electric field calculations. 
	A remarkable advantage of using such an implicit method is that it enables us to automatically preserve the DDF property on non-staggered meshes and establish the energy conservation properties (see Theorems \ref{theorem:energyconservation} and \ref{thm:totalenergy}). Though this implicit time stepping introduces some additional computational cost due to iteratively solving a nonlinear system, our numerical experiments show that this increase is manageable because the iteration typically converges within a few steps (see Section~\ref{S.5} for the numerical evidence).
	
	Let $\bm{B}^{n + \frac{1}{2}}_{ij}$ denote the arithmetic average of $\bm{B}^{n}_{ij}$ and $\bm{B}^{n+1}_{ij}$:
	\begin{equation*}
		\bm{B}^{n + \frac{1}{2}}_{ij} = 
		\begin{pmatrix}
			B^{x, n + \frac{1}{2}}_{ij}, B^{y, n + \frac{1}{2}}_{ij}, B^{z, n + \frac{1}{2}}_{ij}
		\end{pmatrix}^\top 
		= \frac{\bm{B}^{n}_{ij} + \bm{B}^{n+1}_{ij}}{2}.
	\end{equation*} 
	The second-order finite difference CT scheme for \eqref{FDequations} is given by
	\begin{equation}\label{discretization2}
		\left\{
		\begin{aligned}
			B^{x, n+1}_{ij} &= B^{x,n}_{ij} - \Delta t \frac{ \Omega^{z, n + \frac{1}{2}}_{i, j +1} - \Omega^{z, n + \frac{1}{2}}_{i, j -1} }{2 \Delta y}, \\
			B^{y, n+1}_{ij} &= B^{y,n}_{ij} + \Delta t \frac{ \Omega^{z, n + \frac{1}{2}}_{i + 1, j} - \Omega^{z, n + \frac{1}{2}}_{i - 1, j} }{2 \Delta x}, \\
			B^{z, n + 1}_{ij} &= B^{z,n}_{ij} - \Delta t \frac{ \Omega^{y, n + \frac{1}{2}}_{i+1, j} - \Omega^{y, n + \frac{1}{2}}_{i-1, j } }{2 \Delta x} + \Delta t \frac{ \Omega^{x, n + \frac{1}{2}}_{i, j +1} - \Omega^{x, n + \frac{1}{2}}_{i, j -1} }{2 \Delta y}, \\
			\rho^n_{ij} {\bm v}^{n+1}_{ij} &= \rho^n_{ij} {\bm v}^n_{ij} - \Delta t {\bm B}^{n+\frac{1}{2}}_{ij} \times \operatorname{curl} {\bm B}^{n + \frac{1}{2}}_{ij},
		\end{aligned}
		\right.
	\end{equation}
	where $\mathbf{\Omega}^{n + \frac{1}{2}}_{ij} = {\bm B}^{n + \frac{1}{2}}_{ij} \times {\bm v}^{n + \frac{1}{2}}_{ij}$, ${\bm v}^{n + \frac{1}{2}}_{ij} =\dfrac{1}{2}( {\bm v}^{n}_{ij} + {\bm v}^{n + 1}_{ij})$, and the discrete curl $\operatorname{curl} {\bm B}^{n + \frac{1}{2}}_{ij}$ is defined as
	\begin{equation}\label{curlB_ij}
		\operatorname{curl} {\bm B}^{n + \frac{1}{2}}_{ij} := 
		\begin{pmatrix}
			\frac{B^{z, n + \frac{1}{2}}_{i,j+1} - B^{z, n + \frac{1}{2}}_{i,j-1}}{2 \Delta y} \\
			-\frac{B^{z, n + \frac{1}{2}}_{i+1,j} - B^{z, n + \frac{1}{2}}_{i-1,j}}{2 \Delta x} \\
			\frac{B^{y, n + \frac{1}{2}}_{i+1,j} - B^{y, n + \frac{1}{2}}_{i-1,j}}{2 \Delta x} - \frac{B^{x, n + \frac{1}{2}}_{i,j+1} - B^{x, n + \frac{1}{2}}_{i,j-1}}{2 \Delta y}
		\end{pmatrix}.
	\end{equation}

	Define $\bm{R}_{ij} := \left( \bm{B}_{ij}, \sqrt{\rho_{ij}} \bm{v}_{ij} \right)^\top$, and denote ${\bf R} =  \{\bm{R}_{ij}\}$ as the collection of all $\bm{R}_{ij}$ for all $i$ and $j$. Then scheme \eqref{discretization2} can be rewritten in abstract form:
	\begin{equation}\label{eq:nonlinears}
		{\bf R}^{n+1} = {\bf R}^n - \Delta t  {\bf \Psi} \left( \frac{{\bf R}^n + {\bf R}^{n+1} }{2}  \right).
	\end{equation}
	Given ${\bf R}^n$, we propose the following iterative algorithm to solve this nonlinear system:
	\begin{equation}\label{iter}
		\begin{aligned}
			{\bf R}^{(0)} &= {\bf R}^{n}, \\
			{\bf R}^{(k+1)} &= {\bf R}^{n} - \Delta t {\bf \Psi} \left( \frac{{\bf R}^n + {\bf R}^{(k)} } 2  \right), \quad k=0,1,\dots
		\end{aligned}
	\end{equation}
	The stopping criterion is set as follows:
	\begin{equation}\label{eq:stopping}
		{\mathcal E}_{k} := \max \left\{ \left\| {\bm B}^{(k+1)} - {\bm B}^{(k)} \right\|_{\infty}, \left\| {\bm v}^{(k+1)} - {\bm v}^{(k)} \right\|_{\infty} \right\} < \varepsilon_{tol},
	\end{equation}
	where $\varepsilon_{\text{tol}}$ is the iteration error tolerance typically set to $10^{-10}$ in our numerical experiments. We numerically observe that this iterative algorithm always converges in our tests. Moreover, it exhibits exponential convergence with respect to $k$. Typically, the error ${\mathcal E}_{k}$ is reduced to machine precision within 5 to 9 iterations on average (requiring fewer steps on finer meshes), and never exceeds 20 steps. This demonstrates the efficiency of the algorithm.
	
	After solving the nonlinear system \eqref{eq:nonlinears} to obtain $\bm{B}_{ij}^{n+1}$ and $\bm{v}_{ij}^{n+1}$, we update the mechanical energy $E$ using the invariance of internal energy:
	\begin{equation}\label{updatingsource}
		E^{n+1}_{ij} = E^{n}_{ij} - \frac{1}{2} \rho^n_{ij} \left|{\bm v}^n_{ij}\right|^2 + \frac{1}{2} \rho^{n+1}_{ij} \left|{\bm v}^{n+1}_{ij}\right|^2.
	\end{equation}
	Note that the density remains unchanged in System-B, i.e., $\rho^{n+1}_{ij} = \rho^n_{ij}$.
	
	The discrete divergence of $\bm{B}$ at time $t_n$ at the point $(x_i, y_j)$ is defined as
	\begin{equation}\label{definition:DDF}
		\left( \nabla \cdot {\bm B} \right)_{ij}^{n} := \frac{B^{x,n}_{i+1, j} - B^{x,n}_{i-1,j}}{2 \Delta x} + \frac{B^{y,n}_{i, j+1} - B^{y,n}_{i,j-1}}{2 \Delta y}.
	\end{equation}
	We now demonstrate that the DDF condition of the magnetic field, $\left( \nabla \cdot {\bm B} \right)_{ij}^{n} =0$, is naturally preserved by the CT discretization in \eqref{discretization2}. Additionally, an energy identity holds under periodic boundary conditions or when the solution has compact support.
	
	\begin{theorem}[Discrete Divergence-Free (DDF) Property]\label{theorem:DDF2D}
		Assume $\left( \nabla \cdot {\bm B} \right)_{ij}^{n} = 0$ for all $i$ and $j$. Then $\left( \nabla \cdot {\bm B} \right)_{ij}^{n+1} = 0$ for all $i$ and $j$.
	\end{theorem}
	
	\begin{proof}
		By direct computation, we have 
		\begin{equation*}
			\begin{aligned}
				\left( \nabla \cdot {\bm B} \right)_{ij}^{n+1} - \left( \nabla \cdot {\bm B} \right)_{ij}^{n} &= -\Delta t \frac{\Omega^{z, n + \frac{1}{2}}_{ i+1,j+1 } - \Omega^{z, n + \frac{1}{2}}_{ i+1,j-1 }}{2 \Delta y \Delta x } + \Delta t \frac{\Omega^{z, n + \frac{1}{2}}_{ i-1,j+1 } - \Omega^{z, n + \frac{1}{2}}_{ i-1,j-1 }}{2 \Delta x \Delta y} \\
				& \quad + \Delta t \frac{\Omega^{z, n + \frac{1}{2}}_{ i+1,j+1 } - \Omega^{z, n + \frac{1}{2}}_{ i-1,j+1 }}{2 \Delta x \Delta y} - \Delta t \frac{\Omega^{z, n + \frac{1}{2}}_{ i+1,j-1 } - \Omega^{z, n + \frac{1}{2}}_{ i-1,j-1 }}{2 \Delta x \Delta y} = 0.
			\end{aligned}
		\end{equation*}
		The proof is completed.
	\end{proof}

	\begin{theorem}[Energy Identity]\label{theorem:energyconservation}
		If the boundary condition is periodic or the solution has compact support, the CT scheme \eqref{discretization2} satisfies the following energy conservation identity:
		\begin{equation}\label{theorem:energyconservationequation}
			\sum_{i,j} \left( \frac{1}{2} \rho^{n}_{ij} \left| {\bm v}^{n}_{ij} \right|^2  + \frac{1}{2} \left| {\bm B}^{n}_{ij} \right|^2 \right) = \sum_{i,j} \left( \frac{1}{2} \rho^{n+1}_{ij} \left| {\bm v}^{n+1}_{ij} \right|^2  + \frac{1}{2} \left| {\bm B}^{n+1}_{ij} \right|^2 \right).
		\end{equation}
	\end{theorem}
	
	\begin{proof}
		In \eqref{discretization2}, we multiply both sides of each equation by $B^{x,n+\frac{1}{2}}_{ij}$, $B^{y,n+\frac{1}{2}}_{ij}$, $B^{z,n+\frac{1}{2}}_{ij}$, and ${\bm v}^{n+\frac{1}{2}}_{ij}$, respectively:
		\begin{equation}\label{multiplyDiscretizaion}
			\left\{
			\begin{aligned}
				\frac{1}{2} \left( \left| B^{x, n+1}_{ij} \right|^2 - \left| B^{x,n}_{ij} \right|^2 \right) &=  - \Delta t \frac{ \Omega^{z, n + \frac{1}{2}}_{i, j +1} - \Omega^{z, n + \frac{1}{2}}_{i, j -1} }{2 \Delta y} B^{x, n+\frac{1}{2}}_{ij}, \\
				\frac{1}{2} \left( \left| B^{y, n+1}_{ij} \right|^2 - \left| B^{y,n}_{ij} \right|^2 \right) &=   \Delta t \frac{ \Omega^{z, n + \frac{1}{2}}_{i + 1, j} - \Omega^{z, n + \frac{1}{2}}_{i - 1, j} }{2 \Delta x} B^{y, n+\frac{1}{2}}_{ij}, \\
				\frac{1}{2} \left( \left| B^{z, n + 1}_{ij} \right|^2 - \left| B^{z,n}_{ij} \right|^2 \right) &=  - \Delta t \frac{ \Omega^{y, n + \frac{1}{2}}_{i+1, j} - \Omega^{y, n + \frac{1}{2}}_{i-1, j } }{2 \Delta x} B^{z, n + \frac{1}{2}}_{ij} + \Delta t \frac{ \Omega^{x, n + \frac{1}{2}}_{i, j +1} - \Omega^{x, n + \frac{1}{2}}_{i, j -1} }{2 \Delta y} B^{z, n + \frac{1}{2}}_{ij}, \\
				\frac{1}{2} \rho^n_{ij} \left( \left| {\bm v}^{n+1}_{ij} \right|^2  - \left| {\bm v}^n_{ij} \right|^2 \right) &=    \Delta t \mathbf{\Omega}^{n+\frac{1}{2}}_{ij} \cdot \operatorname{curl} {\bm B}^{n + \frac{1}{2}}_{ij}, 
			\end{aligned}
			\right.
		\end{equation}
		where the last equation follows from the identity
		\begin{equation*}
			- \left( {\bm B}^{n+\frac{1}{2}}_{ij} \times \operatorname{curl} {\bm B}^{n + \frac{1}{2}}_{ij} \right) \cdot {\bm v}^{n+\frac{1}{2}}_{ij} = \mathbf{\Omega}^{n+\frac{1}{2}}_{ij} \cdot \operatorname{curl} {\bm B}^{n + \frac{1}{2}}_{ij}.
		\end{equation*}
		Plugging \eqref{curlB_ij} into \eqref{multiplyDiscretizaion} and summing both sides give
		\begin{equation}\label{energyequation}
			\begin{aligned}
				&\frac{1}{2} \rho^{n+1}_{ij} \left| {\bm v}^{n+1}_{ij} \right|^2  + \frac{1}{2} \left| {\bm B}^{n+1}_{ij} \right|^2 - \left( \frac{1}{2} \rho^{n}_{ij} \left| {\bm v}^{n}_{ij} \right|^2  + \frac{1}{2} \left| {\bm B}^{n}_{ij} \right|^2  \right) = \\
				&\quad \frac{\Delta t}{2 \Delta y} \left[  \Omega^{x,n+\frac{1}{2}}_{ij} \left( B^{z, n + \frac{1}{2}}_{i,j+1} - B^{z, n + \frac{1}{2}}_{i,j-1} \right) + B^{z, n + \frac{1}{2}}_{ij} \left( \Omega^{x,n+\frac{1}{2}}_{i,j+1} - \Omega^{x,n+\frac{1}{2}}_{i,j-1} \right) \right] \\
				& \quad - \frac{\Delta t}{2 \Delta x} \left[  \Omega^{y,n+\frac{1}{2}}_{ij} \left( B^{z, n + \frac{1}{2}}_{i+1,j} - B^{z, n + \frac{1}{2}}_{i-1,j} \right) + B^{z, n + \frac{1}{2}}_{ij} \left( \Omega^{y,n+\frac{1}{2}}_{i+1,j} - \Omega^{y,n+\frac{1}{2}}_{i-1,j} \right) \right] \\
				& \quad + \frac{\Delta t}{2 \Delta x} \left[  \Omega^{z,n+\frac{1}{2}}_{ij} \left( B^{y, n + \frac{1}{2}}_{i+1,j} - B^{y, n + \frac{1}{2}}_{i-1,j} \right) + B^{y, n + \frac{1}{2}}_{ij} \left( \Omega^{z,n+\frac{1}{2}}_{i+1,j} - \Omega^{z,n+\frac{1}{2}}_{i-1,j} \right) \right] \\
				& \quad - \frac{\Delta t}{2 \Delta y} \left[  \Omega^{z,n+\frac{1}{2}}_{ij} \left( B^{x, n + \frac{1}{2}}_{i,j+1} - B^{x, n + \frac{1}{2}}_{i,j-1} \right) + B^{x, n + \frac{1}{2}}_{ij} \left( \Omega^{z,n+\frac{1}{2}}_{i,j+1} - \Omega^{z,n+\frac{1}{2}}_{i,j-1} \right) \right].
			\end{aligned} 
		\end{equation}
		Based on periodic boundary conditions or compact support, the right-hand side cancels, completing the proof.
	\end{proof}
	
	\begin{theorem}[Total Energy Conservation]\label{thm:totalenergy}
		If the boundary condition is periodic or the solution has compact support, then the scheme \eqref{discretization2} with \eqref{updatingsource} conserves the total energy:
		\begin{equation*}
			\sum_{i,j} \left( E_{ij}^n + \frac{1}{2} \left| {\bm B}^{n}_{ij} \right|^2 \right) = \sum_{i,j} \left( E^{n+1}_{ij} + \frac{1}{2} \left| {\bm B}^{n+1}_{ij} \right|^2 \right).
		\end{equation*}
	\end{theorem}
	
	\begin{proof}
		Using \eqref{updatingsource}, we have
		\begin{equation*}
			\sum_{i,j} \left( E^{n+1}_{ij} - \frac{1}{2} \rho^{n+1}_{ij} \left|{\bm v}^{n+1}_{ij}\right|^2 \right) = \sum_{i,j} \left( E^{n}_{ij} - \frac{1}{2} \rho^n_{ij} \left| {\bm v}^n_{ij} \right|^2 \right).
		\end{equation*}
		Adding this and \eqref{theorem:energyconservationequation} completes the proof.
	\end{proof}
	
\subsubsection{Unique Solvability of System \eqref{eq:nonlinears} and Convergence Analysis of Algorithm \eqref{iter}}

We now analyze the uniqueness of the solution to the nonlinear algebraic system \eqref{eq:nonlinears} and the convergence of the iterative algorithm \eqref{iter}. Extensive numerical experiments have demonstrated that this algorithm consistently exhibits robust convergence, suggesting that it may be unconditionally convergent in practice. Yet, currently, we are only able to establish a conditional convergence theorem under a more restrictive CFL-like condition, which is not strictly required for practical computations.

Denote $\delta_x \mathbf{R}_{ij} := \frac{1}{2}\left( \mathbf{R}_{i+1,j} -\mathbf{R}_{i-1,j} \right)$ 
and $\delta_y \mathbf{R}_{ij} := \frac{1}{2}\left( \mathbf{R}_{i,j+1} - \mathbf{R}_{i,j-1} \right)$. 
Then, we have the following theorem.

\begin{theorem}
	If there exists a constant $\eta \in (0, 1)$ such that 
	$
	\Delta t \left(  \frac{\alpha_x}{\Delta x} + \frac{\alpha_y}{\Delta y}  \right) \le 2 \eta,
	$ 
	where $\alpha_x$ and $\alpha_y$ are defined as follows:
	\begin{equation*}
		\begin{aligned}
			\alpha_x &= \left\| \left\| \bm{v}  \right\|_1 + \frac{ \left\| \bm{B} \right\|_1 + \left|  \delta_x B^y \right| + \left| \delta_x B^z \right| }{\sqrt{\rho}}  \right\|_{L^{\infty}}, \\ 
			\alpha_y &=  \left\| \left\| \bm{v}  \right\|_1 + \frac{ \left\| \bm{B} \right\|_1 + \left|  \delta_y B^x \right| + \left| \delta_y B^z \right| }{\sqrt{\rho}}  \right\|_{L^{\infty}},
		\end{aligned}
	\end{equation*}
	then the algebraic system \eqref{eq:nonlinears} has a unique solution, and the iterative algorithm \eqref{iter} converges to this solution.
\end{theorem}

\begin{proof}
	Noting the sparsity of the Jacobian matrix $\frac{\partial \mathbf{\Psi}}{\partial \mathbf{R}}$, we conduct the estimate based on the infinity norm. By \eqref{eq:nonlinears} and \eqref{iter}, we have 
	\begin{equation}\label{differenceInequality}
		\begin{aligned}
			\left\| \mathbf{R}^{(k+1)} - \mathbf{R}^{n+1} \right\|_{\infty} &=\Delta t \left\| \mathbf{\Psi}\left(\frac{\mathbf{R}^n + \mathbf{R}^{(k)}}{2}\right)  - \mathbf{\Psi}\left(\frac{\mathbf{R}^n + \mathbf{R}^{n+1}}{2}\right) \right\|_{\infty}\\
			&\le \Delta t \left \|  \frac{\partial \mathbf{\Psi}}{\partial \mathbf{R}} \right \|_{\infty} \left \| \frac{\mathbf{R}^{(k)} - \mathbf{R}^{n+1}}{2} \right\|_{\infty}\\
			& = \frac{\Delta t}{2} \max_{i,j} \left\| \frac{\partial \mathbf{\Psi}_{ij}}{\partial \hat{\mathbf{R}}_{ij}} \right\|_{\infty}  \left \| \mathbf{R}^{(k)} - \mathbf{R}^{n+1} \right\|_{\infty},
		\end{aligned}
	\end{equation}
	where 
	\begin{align*}
		\hat{\mathbf{R}}_{ij} &= \left( \mathbf{R}_{i-1,j}; \mathbf{R}_{i,j-1} ; \mathbf{R}_{ij} ; \mathbf{R}_{i+1,j} ; \mathbf{R}_{i,j+1}  \right),
		\\
		\frac{\partial \mathbf{\Psi}_{ij}}{\partial \hat{\mathbf{R}}_{ij}} &= \left(  \frac{\partial \mathbf{\Psi}_{ij}}{\partial \mathbf{R}_{i-1,j}} , 
		\frac{\partial \mathbf{\Psi}_{ij}}{\partial \mathbf{R}_{i,j-1}}, 
		\frac{\partial \mathbf{\Psi}_{ij}}{\partial \mathbf{R}_{ij}}, 
		\frac{\partial \mathbf{\Psi}_{ij}}{\partial \mathbf{R}_{i+1,j}}, 
		\frac{\partial \mathbf{\Psi}_{ij}}{\partial \mathbf{R}_{i,j+1}} \right).
	\end{align*}
	After some calculations, we obtain these Jacobian matrices:		
	\[
	\frac{\partial \mathbf{\Psi}_{ij}}{\partial \mathbf{R}_{i-1,j}} = \begin{bmatrix}
		0 & 0 & 0 & 0 & 0 & 0 \\
		\frac{v^y_{i-1,j}}{2 \Delta x} & -\frac{v^x_{i-1,j}}{2 \Delta x} & 0 & -\frac{B^{y}_{i-1,j}}{2 \sqrt{\rho_{i-1,j}} \Delta x} & \frac{B^x_{i-1,j}}{2\sqrt{\rho_{i-1,j}} \Delta x } & 0 \\
		\frac{v^z_{i-1,j}}{2 \Delta x} & 0 & -\frac{v^x_{i-1,j}}{2 \Delta x}  &  -\frac{B^z_{i-1,j}}{2\sqrt{\rho_{i-1,j}} \Delta x} & 0 & \frac{B^x_{i-1,j}}{2\sqrt{\rho_{i-1,j}} \Delta x} \\
		0 & -\frac{B^y_{ij}}{2\sqrt{\rho_{i,j}} \Delta x } & -\frac{B^z_{ij}}{2\sqrt{\rho_{i,j}} \Delta x} & 0 & 0 & 0 \\
		0 & \frac{B^x_{ij}}{2\sqrt{\rho_{i,j}} \Delta x} & 0 & 0 & 0 & 0 \\
		0 & 0 & \frac{B^x_{ij}}{2\sqrt{\rho_{i,j}} \Delta x} & 0 & 0 & 0
	\end{bmatrix},
	\]
	\[
	\frac{\partial \mathbf{\Psi}_{ij}}{\partial \mathbf{R}_{i,j-1}} = \begin{bmatrix}
		-\frac{v^y_{i,j-1}}{2 \Delta y} & \frac{v^x_{i,j-1}}{2 \Delta y} & 0 & \frac{B^y_{i,j-1}}{2\sqrt{\rho_{i,j-1}} \Delta y} & -\frac{B^x_{i,j-1}}{2\sqrt{\rho_{i,j-1}} \Delta y} & 0 \\
		0 & 0 & 0 & 0 & 0 & 0 \\
		0 & \frac{v^z_{i,j-1}}{2 \Delta y} & -\frac{v^y_{i,j-1}}{2 \Delta y} & 0 & -\frac{B^z_{i,j-1}}{2\sqrt{\rho_{i,j-1}} \Delta y} & \frac{B^y_{i,j-1}}{2\sqrt{\rho_{i,j-1}} \Delta y} \\
		\frac{B^y_{ij}}{2\sqrt{\rho_{i,j}} \Delta y} & 0 & 0 & 0 & 0 & 0 \\
		-\frac{B^x_{ij}}{2\sqrt{\rho_{i,j}} \Delta y} & 0 & -\frac{B^z_{ij}}{2\sqrt{\rho_{i,j}} \Delta y} & 0 & 0 & 0 \\
		0 & 0 & \frac{B^y_{ij}}{2\sqrt{\rho_{i,j}} \Delta y} & 0 & 0 & 0
	\end{bmatrix},
	\]
	\[
	\frac{\partial \mathbf{\Psi}_{ij}}{\partial \mathbf{R}_{ij}} = \begin{bmatrix}
		0 & 0 & 0 & 0 & 0 & 0 \\
		0 & 0 & 0 & 0 & 0 & 0 \\
		0 & 0 & 0 & 0 & 0 & 0 \\
		0 & \frac{\delta_x B^y_{ij}}{\sqrt{\rho_{i,j}} \Delta x} - \frac{\delta_y B^x_{ij}}{\sqrt{\rho_{i,j}} \Delta y} & \frac{\delta_x B^z_{ij} }{\sqrt{\rho_{i,j}} \Delta x} & 0 & 0 & 0 \\
		-\frac{\delta_x B^y_{ij}}{\sqrt{\rho_{i,j}} \Delta x} + \frac{\delta_y B^x_{ij}}{\sqrt{\rho_{i,j}} \Delta y} & 0 & \frac{\delta_y B^z_{ij} }{\sqrt{\rho_{i,j}} \Delta y} & 0 & 0 & 0 \\
		-\frac{\delta_x B^z_{ij} }{\sqrt{\rho_{i,j}} \Delta x} & -\frac{\delta_y B^z_{ij} }{\sqrt{\rho_{i,j}} \Delta y} & 0 & 0 & 0 & 0
	\end{bmatrix},
	\]
	\[
	\frac{\partial \mathbf{\Psi}_{ij}}{\partial \mathbf{R}_{i+1,j}} = \begin{bmatrix}
		0 & 0 & 0 & 0 & 0 & 0 \\
		-\frac{v^y_{i+1,j}}{2 \Delta x} & \frac{v^x_{i+1,j}}{2 \Delta x} & 0 & \frac{B^y_{i+1,j}}{2 \sqrt{\rho_{i+1,j}} \Delta x } & -\frac{B^x_{i+1,j}}{2 \sqrt{\rho_{i+1,j}} \Delta x} & 0 \\
		-\frac{v^z_{i+1,j}}{2 \Delta x} & 0 & \frac{v^x_{i+1,j}}{2 \Delta x} & \frac{B^z_{i+1,j}}{2 \sqrt{\rho_{i+1,j}} \Delta x} & 0 & -\frac{B^x_{i+1,j}}{2 \sqrt{\rho_{i+1,j}} \Delta x} \\
		0 & \frac{B^y_{ij}}{2 \sqrt{\rho_{i,j}} \Delta x} & \frac{B^z_{ij}}{2 \sqrt{\rho_{i,j}} \Delta x} & 0 & 0 & 0 \\
		0 & -\frac{B^x_{ij}}{2 \sqrt{\rho_{i,j}} \Delta x} & 0 & 0 & 0 & 0 \\
		0 & 0 & -\frac{B^x_{ij}}{2 \sqrt{\rho_{i,j}} \Delta x} & 0 & 0 & 0
	\end{bmatrix},
	\]
	\[
	\frac{\partial \mathbf{\Psi}_{ij}}{\partial \mathbf{R}_{i,j+1}} = \begin{bmatrix}
		\frac{v^y_{i,j+1}}{2 \Delta y} & -\frac{v^x_{i,j+1}}{2 \Delta y} & 0 & -\frac{B^y_{i,j+1}}{2 \sqrt{\rho_{i,j+1}} \Delta y} & \frac{B^x_{i,j+1}}{2 \sqrt{\rho_{i,j+1}} \Delta y} & 0 \\
		0 & 0 & 0 & 0 & 0 & 0 \\
		0 & -\frac{v^z_{i,j+1}}{2 \Delta y} & \frac{v^y_{i,j+1}}{2 \Delta y} & 0 & \frac{B^z_{i,j+1}}{2 \sqrt{\rho_{i,j+1}} \Delta y} & -\frac{B^y_{i,j+1}}{2 \sqrt{\rho_{i,j+1}} \Delta y} \\
		-\frac{B^y_{ij}}{2 \sqrt{\rho_{i,j}} \Delta y} & 0 & 0 & 0 & 0 & 0 \\
		\frac{B^x_{ij}}{2 \sqrt{\rho_{i,j}} \Delta y} & 0 & \frac{B^z_{ij}}{2 \sqrt{\rho_{i,j}} \Delta y} & 0 & 0 & 0 \\
		0 & 0 & -\frac{B^2_{ij}}{2 \sqrt{\rho_{i,j}} \Delta y} & 0 & 0 & 0
	\end{bmatrix}.
	\]
	Denote the element of $\frac{\partial \mathbf{\Psi}_{ij}}{\partial \hat{\mathbf{R}}_{ij}}$ at position $(\ell,m)$ as $a_{ij}^{\ell,m}$, and define $A_{ij}^{(\ell)} = \sum\limits_{m} \left| a_{ij}^{\ell,m} \right|$. Then, we have the following estimates:
	\begin{equation}\label{estimate_a1}
		\begin{aligned}
			A_{ij}^{(1)} &\le	\frac{\left| v^y_{i,j+1} \right| + \left| v^x_{i,j+1} \right| +\left| v^y_{i,j-1} \right|+ \left| v^x_{i,j-1} \right|}{2 \Delta y} + 
			\frac{1}{2 \Delta y } \left( \frac{\left| B^y_{i,j+1} \right|}{\sqrt{\rho_{i,j+1}}} +\frac{\left| B^x_{i,j+1} \right|}{\sqrt{\rho_{i,j+1}}} +\frac{\left| B^y_{i,j-1} \right|}{\sqrt{\rho_{i,j-1}}} +\frac{\left| B^x_{i,j-1} \right|}{\sqrt{\rho_{i,j-1}}} \right) \\
			&\le \frac{1}{\Delta y} \left\|  \left| v^x \right| + \left| v^y \right| + \frac{\left| B^x \right|}{\sqrt{\rho}} + \frac{\left| B^y \right|}{\sqrt{\rho}}  \right\|_{L^{\infty}}, 
		\end{aligned}
	\end{equation}
	\begin{equation}
		\begin{aligned}
			A_{ij}^{(2)} &\le	\frac{\left| v^y_{i+1,j} \right| + \left| v^x_{i+1,j} \right| +\left| v^y_{i-1,j} \right|+ \left| v^x_{i-1,j} \right|}{2 \Delta x} + 
			\frac{1}{2 \Delta x } \left( \frac{\left| B^y_{i+1,j} \right|}{\sqrt{\rho_{i+1,j}}} +\frac{\left| B^x_{i+1,j} \right|}{\sqrt{\rho_{i+1,j}}} +\frac{\left| B^y_{i-1,j} \right|}{\sqrt{\rho_{i-1,j}}} +\frac{\left| B^x_{i-1,j} \right|}{\sqrt{\rho_{i-1,j}}} \right) \\
			&\le \frac{1}{\Delta x} \left\|  \left| v^x \right| + \left| v^y \right| + \frac{\left| B^x \right|}{\sqrt{\rho}} + \frac{\left| B^y \right|}{\sqrt{\rho}}  \right\|_{L^{\infty}}, 
		\end{aligned}
	\end{equation}
	\begin{equation}
		\begin{aligned}
			A_{ij}^{(3)} &\le \frac{1}{2 \Delta y} \left( \left| v^z_{i,j+1} \right|  + \left| v^y_{i,j+1} \right| + \left| v^z_{i,j-1} \right|+ \left| v^y_{i,j-1} \right|  + \frac{\left| B^z_{i,j+1} \right| + \left| B^y_{i,j+1} \right|}{\sqrt{\rho_{i,j+1}}} + \frac{\left| B^z_{i,j-1} \right| + \left| B^y_{i,j-1} \right|}{\sqrt{\rho_{i,j-1}}} \right)	 \\
			&+ \frac{1}{2 \Delta x} \left( \left| v^z_{i+1,j} \right|  + \left| v^x_{i+1,j} \right| + \left| v^z_{i-1,j} \right|+ \left| v^x_{i,j-1} \right|  + \frac{\left| B^z_{i+1,j} \right| + \left| B^x_{i+1,j} \right|}{\sqrt{\rho_{i+1,j}}} + \frac{\left| B^z_{i-1,j} \right| + \left| B^x_{i-1,j} \right|}{\sqrt{\rho_{i-1,j}}} \right) \\
			&\le \frac{1}{\Delta y} \left\|  \left| v^y \right| + \left| v^z \right| + \frac{\left| B^y \right| + \left| B^z \right|}{\sqrt{\rho}}  \right\|_{L^{\infty}} + \frac{1}{\Delta x} \left\|  \left| v^x \right| + \left| v^z \right| + \frac{\left| B^x \right| + \left| B^z \right|}{\sqrt{\rho}}  \right\|_{L^{\infty}}, 
		\end{aligned}
	\end{equation}
	\begin{equation}
		\begin{aligned}
			A_{ij}^{(4)} &\le \frac{\left| \delta_x B^y_{ij} \right|}{\Delta x \sqrt{\rho_{ij}}} + \frac{\left| \delta_y B^x_{ij} \right|}{\Delta y \sqrt{\rho_{ij}}} +  \frac{\left| \delta_x B^z_{ij} \right|}{\Delta x \sqrt{\rho_{ij}}} + \frac{\left|B^y_{ij}\right|}{\Delta y \sqrt{\rho_{ij}}} + \frac{\left|B^y_{ij}\right|}{\Delta x \sqrt{\rho_{ij}}} + \frac{\left|B^z_{ij}\right|}{\Delta x \sqrt{\rho_{ij}}}\\
			&\le \left\| \frac{ \frac{1}{\Delta x} \left|\delta_x B^y\right|+ \frac{1}{\Delta x} \left|\delta_x B^z\right|+ \frac{1}{\Delta y}\left|\delta_y B^x\right|}{\sqrt{\rho}} \right\|_{L^{\infty}} + \left\| \frac{\left| B^y \right| + \left| B^z \right|}{\Delta x \sqrt{\rho}} + \frac{B^y}{\Delta y \sqrt{\rho}} \right\|_{L^\infty},
		\end{aligned}
	\end{equation}
	\begin{equation}
		\begin{aligned}
			A_{ij}^{(5)} &\le  \frac{\left|\delta_y B^x_{ij}\right|}{\Delta y \sqrt{\rho_{ij}}} + \frac{\left|\delta_x B^y_{ij}\right|}{\Delta x \sqrt{\rho_{ij}}} + \frac{\left|\delta_y B^z_{ij}\right|}{\Delta y \sqrt{\rho_{ij}}} + \frac{\left|B^x_{ij}\right|}{\Delta x \sqrt{\rho_{ij}}}  + \frac{\left|B^x_{ij}\right|}{\Delta y \sqrt{\rho_{ij}}}  +\frac{\left|B^z_{ij}\right|}{\Delta y \sqrt{\rho_{ij}}} \\
			&\le \left\| \frac{ \frac{1}{\Delta x} \left|\delta_x B^y\right|+ \frac{1}{\Delta y}\left|\delta_y B^z\right|+\frac{1}{\Delta y}\left|\delta_y B^x\right|}{\sqrt{\rho}} \right\|_{L^{\infty}} + \left\| \frac{\left| B^x \right| + \left| B^z \right|}{\Delta y \sqrt{\rho}} + \frac{B^x}{\Delta x \sqrt{\rho}} \right\|_{L^\infty},		
		\end{aligned}
	\end{equation}
	\begin{equation}\label{estimate_a6}
		\begin{aligned}
			A_{ij}^{(6)} &\le \frac{\left|\delta_x B^z_{ij}\right|}{\Delta x \sqrt{\rho_{ij}}} + \frac{\left|\delta y B^z_{ij}\right|}{ \Delta y \sqrt{\rho_{ij}} } + \frac{\left| B^y_{ij}\right|}{\Delta y\sqrt{\rho_{ij}} } + \frac{\left| B^x_{ij}\right|}{\Delta x\sqrt{\rho_{ij}} } \\
			& \le \left\| \frac{\frac{1}{\Delta x} \left| \delta_x B^z \right| + \frac{1}{\Delta y} \left| \delta_y B^z \right|}{\sqrt{\rho}} \right\|_{L^{\infty}} + \left\| \frac{\left| B^x \right| }{\Delta x \sqrt{\rho}} + \frac{\left| B^y\right|}{\Delta y \sqrt{\rho}} \right\|_{L^{\infty}}.
		\end{aligned}
	\end{equation}
	Combining \eqref{estimate_a1}--\eqref{estimate_a6} gives 
	\begin{equation*}
		\begin{aligned}
			\max_{1\le \ell \le 3} \left\{ A_{ij}^{(\ell)} \right\} &\le \left( \frac{1}{\Delta x} + \frac{1}{\Delta y}  \right) \left\| \left\| \bm{v}\right\|_1 + \frac{\left\| \bm{B}\right\|_1}{\sqrt{\rho}} \right\|_{L^{\infty}},\\
			\max_{4\le \ell \le 6} \left\{ A_{ij}^{(\ell)} \right\} &\le \left( \frac{1}{\Delta x} + \frac{1}{\Delta y}  \right) \left\|  \frac{\left\| \bm{B}\right\|_1}{\sqrt{\rho}} \right\|_{L^{\infty}} +  \frac{1}{\Delta x \sqrt{\rho_{ij}}} \left(\left| \delta_x B^y_{ij}\right|+ \left| \delta_x B^z_{ij}\right|\right) 
			\\
			& \qquad +  \frac{1}{\Delta y \sqrt{\rho_{ij}}} \left(\left| \delta_y B^x_{ij}\right|+ \left| \delta_y B^z_{ij}\right|\right),
		\end{aligned}
	\end{equation*} 
	which yield 
	\begin{equation}\label{JacobianEstiamte}
		\max_{i,j}	\left \|  \frac{\partial \mathbf{\Psi}_{ij}}{\partial \hat{\mathbf{R}}_{ij}} \right \|_{\infty} = \max_{i,j,\ell} \max \left\{ A_{ij}^{(\ell)} \right\} \le  \frac{\alpha_x}{\Delta x} + \frac{\alpha_y}{\Delta y}.
	\end{equation}
	Combining \eqref{JacobianEstiamte} with \eqref{differenceInequality}, we obtain  
	\begin{equation*}
		\begin{aligned}
			\left\| \mathbf{R}^{(k+1)} - \mathbf{R}^{n+1} \right\|_{\infty} &\le \frac{\Delta t}{2} \left( \frac{\alpha_x}{\Delta x} + \frac{\alpha_y}{\Delta y}\right) \left\| \mathbf{R}^{(k)} - \mathbf{R}^{n+1} \right\|_{\infty}\\
			&\le \eta \left\| \mathbf{R}^{(k)} - \mathbf{R}^{n+1} \right\|_{\infty}.
		\end{aligned}
	\end{equation*}
	Since $0<\eta <1$, by the Banach fixed-point theorem, the system \eqref{eq:nonlinears} has a unique solution, and the iterative algorithm \eqref{iter} converges to this solution.
	This concludes the proof.
\end{proof}

	\subsection{Algorithmic Flow of PPCT Scheme}\label{eq:summary}
	
	In this section, we provide a brief summary of the complete PPCT scheme for the full MHD system. Since the midpoint integration formula is second-order accurate, all cell average values can be treated as point values at the cell centers in the sense of second-order accuracy. To simplify the presentation, the scheme does not distinguish between cell center point values and cell average values.
	
	Let $t^{n+1} = t^{n} + \Delta t$. The PPCT scheme consists of three main steps:
	\begin{description}
		\item[Step 1] Solve System-A \eqref{Operator1} with a time step of $\frac{\Delta t}{2}$:
		\begin{equation*}
			\left(\rho^n, {\bm m}^n, E^n\right)^{\top} \xlongrightarrow{(\ref{SSPRK2new})} \left(\rho^{(1)}, {\bm m}^{(1)}, E^{(1)}\right)^{\top}.
		\end{equation*}   	
		\item[Step 2] Solve System-B \eqref{Subsystem} with a time step of $\Delta t$:
		\begin{equation*}
			\left({\bm m}^{(1)}, {\bm B}^n, E^{(1)}\right)^{\top} \xlongrightarrow{(\ref{discretization2}) ~{\rm and}~ (\ref{updatingsource})} \left({\bm m}^{(2)}, {\bm B}^{n+1}, E^{(2)}\right)^{\top}.
		\end{equation*}
		\item[Step 3] Solve System-A \eqref{Operator1} again with a time step of $\frac{\Delta t}{2}$:
		\begin{equation*}
			\left(\rho^{(1)}, {\bm m}^{(2)}, E^{(2)}\right)^{\top} \xlongrightarrow{(\ref{SSPRK2new})} \left(\rho^{n+1}, {\bm m}^{n+1}, E^{n+1}\right)^{\top}.
		\end{equation*} 	
	\end{description}
	
	It can be seen that \textbf{Step 2} enforces the DDF condition for the magnetic field, while the other steps do not involve $\bm{B}$. Therefore, the overall process preserves the DDF condition. Additionally, total energy conservation and positivity are maintained due to the conservation and PP properties of the finite volume method used in System-A.
	
	\begin{corollary}
		If the boundary conditions are periodic or the solution has compact support, the PPCT scheme for solving the MHD system preserves total energy conservation: 
		\begin{equation*}
			\sum_{i,j} \left( E_{ij}^n  + \frac{1}{2} \left| {\bm B}^{n}_{ij} \right|^2 \right) = \sum_{i,j} \left( E^{n+1}_{ij}  + \frac{1}{2} \left| {\bm B}^{n+1}_{ij} \right|^2 \right).
		\end{equation*}
	\end{corollary}
	
	\begin{proof}
		This result follows directly from the fact that each step (\textbf{Step 1}--\textbf{Step 3}) in the PPCT scheme individually conserves total energy.
	\end{proof}
	
	\begin{remark}\label{rem:stability2}
		As discussed in \Cref{rem:stability}, the CT scheme for System-B \eqref{Subsystem} demonstrates unconditional stability in our numerical experiments. Therefore, the time step size $\Delta t$ for both stability and the PP property of the PPCT scheme is restricted only by the scheme $S_A^{\frac{\Delta t}2}$, with $\frac{\Delta t}2$ satisfying the CFL condition \eqref{CFLcondtion}. 
		Specifically, the CFL condition for the entire PPCT scheme \eqref{eq:fullscheme} is given by  
		\begin{equation}\label{CFLcondtion2}
			\Delta t  \left( \frac{\alpha_1}{\Delta x} + \frac{\alpha_2}{\Delta y} \right)  \le \frac{2}{q}, 
		\end{equation}
		where $q>2$ is a fixed constant used in the PP limiter and can be specified arbitrarily in the range $(2, +\infty)$. Here, $\alpha_i$ represents the maximum wave speeds of the Euler equations rather than the MHD equations. If $q$ is chosen to be close to $2$, the maximum allowable CFL number would approach $1$, making this a very mild CFL constraint.
	\end{remark}

	\section{Three-Dimensional PPCT Scheme}\label{S.4}
	In this section, we present our PPCT method in the 3D case. To maintain consistency and clarity, we will use the same notations as in \Cref{S.3}. Where necessary, the definitions will be restated for completeness. 
	
	Assume the computational domain is divided into a set of cubic cells. The cell centered in $\left( x_i, y_j, z_k \right)$ is denoted by $I_{ijk} = \left(x_{i-\frac{1}{2}}, x_{i + \frac{1}{2}}\right) \times \left(y_{j - \frac{1}{2}}, y_{j + \frac{1}{2}}\right) \times \left(z_{k - \frac{1}{2}}, z_{k + \frac{1}{2}}\right)$. Then we  have $x_i = \frac{1}{2} \left(x_{i - \frac{1}{2}} + x_{i + \frac{1}{2}} \right)$, $y_j = \frac{1}{2} \left(y_{j - \frac{1}{2}} + y_{j + \frac{1}{2}} \right)$, and $z_k = \frac{1}{2} \left(z_{k - \frac{1}{2}} + z_{k + \frac{1}{2}} \right)$. 
	Next, we introduce the PP finite volume method $S^{\Delta t}_{A}$ for System-A \eqref{Operator1}. Following that,  we present the finite difference CT method $S^{\Delta t}_{B}$ for System-B \eqref{Operator2}. Based on the second-order Strang splitting, the whole PPCT scheme is given by
	\begin{equation*}
		\bm{U}^{n+1} = S^{\frac{\Delta t}{2}}_{A} \circ S^{\Delta t}_{B} \circ S^{\frac{\Delta t}{2}}_{A} \bm{U}^n.
	\end{equation*}
	
	\subsection{3D Provably PP Finite Volume Method $S_{A}^{{\Delta t}}$ for System-A}
	Let $\overline{\bm{Q}}^n_{ijk} = \left( \overline{\rho}^n_{ijk}, \overline{\bm{m}}^n_{ijk}, \overline{E}^n_{ijk} \right)^\top$ represent the given conservative variables. The corresponding primitive variables are defined as:
	$$\overline{\bm{W}}^n_{ijk} = \left(\overline{\rho}^n_{ijk}, \overline{\bm{v}}^n_{ijk}, \overline{p}^n_{ijk}  \right)^\top= \left( \overline{\rho}^n_{ijk}, \frac{\overline{\bm{m}}^n_{ijk}}{\overline{\rho}^n_{ijk}} , (\gamma - 1) \left( \overline E_{ij}^n - \frac{| \overline{\bm{m}}_{ij}^n|^2}{2 \overline \rho_{ij}^n} \right) \right)^\top.$$
	For second-order accuracy, we reconstruct a piecewise linear function as
	\begin{equation*}\label{reconstruction3D}
		{\bm W}_{ijk}^n(x,y) = \overline{{\bm W}}_{ijk}^n + \delta {\bm W}_{ijk}^x (x - x_i) + \delta {\bm W}_{ijk}^y (y - y_j) + \delta {\bm W}_{ijk}^z (z - z_k), \qquad (x,y, z) \in I_{ijk},
	\end{equation*}
	where $\overline{{\bm W}}_{ijk}^n$ is computed from the cell averages $\overline{\bm{Q}}^n_{ijk}$. The local slopes are determined using a slope limiter. As in the 2D case, we use the van Albada limiter. The definitions of $\delta {\bm W}_{ijk}^x$ and $\delta {\bm W}_{ijk}^y$ are  similar to \eqref{VanAlbada} for the 2D case, while the slope in the $z$-direction is given by
	\begin{equation*}
		\delta {\bm W}_{ijk}^z = \frac{\left(\left[\frac{\overline{ {\bm W}}_{i, j, k + 1}^n- \overline{{\bm W}}_{i j k}^n}{\Delta z}\right]^2+\epsilon_z \right) \odot \frac{\overline{{\bm W}}_{i jk}^n-\overline{{\bm W}}_{i, j, k - 1}^n}{\Delta z}+\left(\left[\frac{\overline{{\bm W}}_{i j k }^n-\overline{{\bm W}}_{i, j, k-1}^n}{\Delta z}\right]^2+\epsilon_z\right) \odot \frac{\overline{{\bm W}}_{i, j, k+1}^n-\overline{{\bm W}}_{i jk}^n}{\Delta z}}{\left[\frac{\overline{{\bm W}}_{i jk}^n-\overline{{\bm W}}_{i, j,k-1}^n}{\Delta z}\right]^2+\left[\frac{\overline{{\bm W}}_{i, j, k +1}^n-\overline{{\bm W}}_{i j k}^n}{\Delta z}\right]^2+2\epsilon_z},
	\end{equation*}
	where $\epsilon_z = 3 \Delta z$ and ``$\odot$" denotes the Hadamard product. Define
	\begin{equation*}
		\begin{aligned}
			\Delta \bm{W}_{ijk}^x &= \left( \Delta \rho_{ijk}^x, \Delta \bm{v}_{ijk}^x, \Delta p_{ijk}^x \right)^\top = \frac{\Delta x}{2} \delta \bm{W}_{ijk}^x, \quad \Delta \bm{W}_{ijk}^y = \left( \Delta \rho_{ijk}^y, \Delta \bm{v}_{ijk}^y, \Delta p_{ijk}^y \right)^\top = \frac{\Delta y}{2} \delta \bm{W}_{ijk}^y, \\
			\Delta \bm{W}_{ijk}^z &= \left( \Delta \rho_{ijk}^z, \Delta \bm{v}_{ijk}^z, \Delta p_{ijk}^z \right)^\top = \frac{\Delta z}{2} \delta \bm{W}_{ijk}^z.
		\end{aligned}
	\end{equation*}
	The limiting values can be written as 
	\begin{equation*}
		\begin{aligned}
			\bm{W}^{n,-}_{i+\frac{1}{2},j, k} &= \overline{\bm{W}}_{ijk}^n + \Delta \bm{W}_{ijk}^x, \qquad 
			\bm{W}^{n,+}_{i-\frac{1}{2},j, k} = \overline{\bm{W}}_{ijk}^n - \Delta \bm{W}_{ijk}^x, \\
			\bm{W}^{n,-}_{i,j+\frac{1}{2},k} &= \overline{\bm{W}}_{ijk}^n + \Delta \bm{W}_{ijk}^y, \qquad 
			\bm{W}^{n,+}_{i,j-\frac{1}{2},k} = \overline{\bm{W}}_{ijk}^n - \Delta \bm{W}_{ijk}^y, \\
			\bm{W}^{n,-}_{i,j,k+\frac{1}{2}} &= \overline{\bm{W}}_{ijk}^n + \Delta \bm{W}_{ijk}^z, \qquad 
			\bm{W}^{n,+}_{i,j,k-\frac{1}{2}} = \overline{\bm{W}}_{ijk}^n - \Delta \bm{W}_{ijk}^z.
		\end{aligned}
	\end{equation*}
	Now we impose a PP limiter on these limiting values. For a fixed value $q > 2$, we calculate the PP limiting parameters by the following three steps:
	
	\noindent
	{\bf Step 1:} Modify the density slopes and ensure density positivity:
		\begin{equation}\label{rho_3D}
			\begin{aligned}
				&\alpha^x  =\begin{cases}
					\min\left(\frac{\overline{\rho}_{ijk}^n}{\left|  \Delta \rho_{ijk}^x \right|(1+\epsilon)}, 1\right),  \quad &  \mbox{if} ~~~  \Delta \rho_{ijk}^x  \neq 0, \\
					1, & \mbox{otherwise},
				\end{cases} \qquad
				\alpha^y  =\begin{cases}
					\min\left(\frac{\overline{\rho}_{ijk}^n}{\left| \Delta \rho_{ijk}^y \right|(1+\epsilon)}, 1\right), \quad  &  \mbox{if} ~~~  \Delta \rho_{ijk}^y  \neq 0, \\
					1, & \mbox{otherwise},
				\end{cases} \\
				&\alpha^z  =\begin{cases}
					\min\left(\frac{\overline{\rho}_{ijk}^n}{\left| \Delta \rho_{ijk}^z \right|(1+\epsilon)}, 1\right),  \quad &  \mbox{if} ~~~  \Delta \rho_{ijk}^z  \neq 0, \\
					1, & \mbox{otherwise},
				\end{cases} 
			\end{aligned}
		\end{equation} 
		where $\epsilon$ is a small positive number and can be taken as $10^{-14}$.

	\noindent
{\bf Step 2:} Modify the pressure slopes and ensure pressure positivity:
		\begin{equation}\label{p_3D}
			\begin{aligned}
				&\kappa^x  =\begin{cases}
					\min\left(\frac{\overline{p}_{ijk}^n}{\left| \Delta p_{ijk}^x \right|(1+\epsilon)}, 1\right),  \quad &  \mbox{if} ~~~  \Delta \rho_{ijk}^x  \neq 0, \\
					1, & \mbox{otherwise},
				\end{cases} \qquad
				\kappa^y  =\begin{cases}
					\min\left(\frac{\overline{p}_{ijk}^n}{\left| \Delta p_{ijk}^y \right|(1+\epsilon)}, 1\right), \quad  &  \mbox{if} ~~~  \Delta  p_{ijk}^y \neq 0, \\
					1, & \mbox{otherwise},
				\end{cases} \\
				&\kappa^z =\begin{cases}
					\min\left(\frac{\overline{p}_{ijk}^n}{\left| \Delta p_{ijk}^z \right|(1+\epsilon)}, 1\right),  \quad &  \mbox{if} ~~~  \Delta p_{ijk}^z   \neq 0, \\
					1, & \mbox{otherwise}.
				\end{cases} 
			\end{aligned}
		\end{equation} 
		
	\noindent
{\bf Step 3:} Modify the velocity slopes:
		\begin{equation}\label{u_3D}
			\begin{aligned}
				\beta &= \begin{cases}
					\min\left( \sqrt{\frac{(q - 2)^2  \overline{\rho}_{ijk}  \overline{p}_{ijk}}{(\gamma - 1) \left(  2 A_1   + (q - 2) \overline{\rho}_{ijk}^2 A_2 \right)}}, 1\right), \quad & \mbox{if} ~~~ \left| \Delta {\bm v}_{ijk}^x \right| + \left| \Delta {\bm v}_{ijk}^y \right| + \left| \Delta {\bm v}_{ijk}^z \right| \neq 0, \\
					1, & \mbox{otherwise},
				\end{cases} \\
				A_ 1 &= \left| C_y \alpha^x \Delta \rho_{ijk}^x \Delta {\bm v}^x_{ijk} + C_y \alpha^y \Delta \rho_{ijk}^y \Delta {\bm v}_{ijk}^y +C_z \alpha^z \Delta \rho_{ijk}^z \Delta {\bm v}_{ijk}^z \right|^2,\\
				A_2 &=  C_x \left|\Delta {\bm v}_{ijk}^x \right|^2 + C_y \left|\Delta {\bm v}_{ijk}^y \right|^2 + C_z \left|\Delta {\bm v}_{ijk}^z \right|^2, 
			\end{aligned}
		\end{equation}
		where the constants 
		$C_y = \frac{\alpha_1 \Delta y \Delta z}{ \alpha_1 \Delta y \Delta z + \alpha_2 \Delta x \Delta z + \alphaLF_3 \Delta x \Delta y}, C_y = \frac{\alpha_2 \Delta x \Delta z}{ \alpha_1 \Delta y \Delta z + \alpha_2 \Delta x \Delta z + \alphaLF_3 \Delta x \Delta y}, C_z = \frac{\alphaLF_3 \Delta x \Delta y}{ \alpha_1 \Delta y \Delta z + \alpha_2 \Delta x \Delta z + \alphaLF_3 \Delta x \Delta y}$, and $\alphaLF_3$ is the numerical viscosity parameter in the Lax--Friedrichs flux in the $z$-direction.

	The limiting values of the primitive variables, after applying the PP limiter, are defined as follows:
	\begin{equation}\label{3D_Euler_limiter}
		\begin{aligned}
			\widetilde{{\bm W}}^{n,+}_{i-\frac{1}{2},j,k} = \begin{pmatrix}
				\tilde{\rho}_{i-\frac{1}{2}, j, k}^{n,+}	 \\
				\tilde{{\bm v}}_{i-\frac{1}{2}, j, k}^{n,+}	 \\
				\tilde{p}_{i-\frac{1}{2}, j, k}^{n,+}	
			\end{pmatrix} = 
			\begin{pmatrix}
				\overline{\rho}_{ijk}^n - \alpha^x \Delta \rho_{ijk}^{x}	 \\
				\overline{{\bm v}}_{ijk}^n - \beta \Delta {\bm v}_{ijk}^{x}	 \\
				\overline{p}_{ijk}^n - \kappa^x \Delta p_{ijk}^{x}	
			\end{pmatrix}, \quad 
			\widetilde{{\bm W}}^{n,-}_{i+\frac{1}{2},j,k} = \begin{pmatrix}
				\tilde{\rho}_{i+\frac{1}{2}, j, k}^{n,-}	 \\
				\tilde{{\bm v}}_{i+\frac{1}{2}, j, k}^{n,-}	 \\
				\tilde{p}_{i+\frac{1}{2}, j, k}^{n,-}	
			\end{pmatrix} = 
			\begin{pmatrix}
				\overline{\rho}_{ijk}^n + \alpha^x \Delta \rho_{ijk}^{x}	 \\
				\overline{{\bm v}}_{ijk}^n + \beta \Delta {\bm v}_{ijk}^{x}	 \\
				\overline{p}_{ijk}^n + \kappa^x \Delta p_{ijk}^{x}	
			\end{pmatrix}, \\
			\widetilde{{\bm W}}^{n,+}_{i,j-\frac{1}{2},k} = \begin{pmatrix}
				\tilde{\rho}_{i, j-\frac{1}{2},k}^{n,+}	 \\
				\tilde{{\bm v}}_{i, j-\frac{1}{2},k}^{n,+}	 \\
				\tilde{p}_{i, j-\frac{1}{2},k}^{n,+}	
			\end{pmatrix} = 
			\begin{pmatrix}
				\overline{\rho}_{ijk}^n - \alpha^y \Delta \rho_{ijk}^{y}	 \\
				\overline{{\bm v}}_{ijk}^n - \beta \Delta {\bm v}_{ijk}^{y}	 \\
				\overline{p}_{ijk}^n - \kappa^y \Delta p_{ijk}^{y}	
			\end{pmatrix}, \quad
			\widetilde{{\bm W}}^{n,-}_{i,j+\frac{1}{2},k} = \begin{pmatrix}
				\tilde{\rho}_{i, j+\frac{1}{2},k}^{n,-}	 \\
				\tilde{{\bm v}}_{i, j+\frac{1}{2},k}^{n,-}	 \\
				\tilde{p}_{i, j+\frac{1}{2},k}^{n,-}	
			\end{pmatrix} = 
			\begin{pmatrix}
				\overline{\rho}_{ijk}^n + \alpha^y \Delta \rho_{ijk}^{y}	 \\
				\overline{{\bm v}}_{ijk}^n + \beta \Delta {\bm v}_{ijk}^{y}	 \\
				\overline{p}_{ijk}^n + \kappa^y \Delta p_{ijk}^{y}	
			\end{pmatrix},\\
			\widetilde{{\bm W}}^{n,+}_{i,j,k-\frac{1}{2}} = \begin{pmatrix}
				\tilde{\rho}_{i, j, k-\frac{1}{2}}^{n,+}	 \\
				\tilde{{\bm v}}_{i, j, k-\frac{1}{2}}^{n,+}	 \\
				\tilde{p}_{i, j, k-\frac{1}{2}}^{n,+}	
			\end{pmatrix} = 
			\begin{pmatrix}
				\overline{\rho}_{ijk}^n - \alpha^z \Delta \rho_{ijk}^{z}	 \\
				\overline{{\bm v}}_{ijk}^n - \beta \Delta {\bm v}_{ijk}^{z}	 \\
				\overline{p}_{ijk}^n - \kappa^z \Delta p_{ijk}^{z}	
			\end{pmatrix}, \quad
			\widetilde{{\bm W}}^{n,-}_{i,j, k+\frac{1}{2}} = \begin{pmatrix}
				\tilde{\rho}_{i, j,k+\frac{1}{2}}^{n,-}	 \\
				\tilde{{\bm v}}_{i, j, k+\frac{1}{2}}^{n,-}	 \\
				\tilde{p}_{i, j, k+\frac{1}{2}}^{n,-}	
			\end{pmatrix} = 
			\begin{pmatrix}
				\overline{\rho}_{ijk}^n + \alpha^z \Delta \rho_{ijk}^{z}	 \\
				\overline{{\bm v}}_{ijk}^n + \beta \Delta {\bm v}_{ijk}^{z}	 \\
				\overline{p}_{ijk}^n + \kappa^z \Delta p_{ijk}^{z}	
			\end{pmatrix}.
		\end{aligned}
	\end{equation}
	Then, the PP finite volume method in the 3D case, with the forward Euler time discretization as example, is given by
	\begin{equation}\label{FVM3D}
		\begin{aligned}
			\overline{{\bm Q}}_{ijk}^{n+1} = & \overline{{\bm Q}}_{ijk}^{n} -  \lambda_1 \left( \mathbf{\hat{G}}_{1} \left( \widetilde{{\bm Q}}^{n,-}_{i+\frac{1}{2}, j,k}, \widetilde{{\bm Q}}^{n,+}_{i+\frac{1}{2}, j, k}\right)  - \mathbf{\hat{G}}_{1}\left(\widetilde{{\bm Q}}^{n,-}_{i-\frac{1}{2}, j, k}, \widetilde{{\bm Q}}^{n,+}_{i-\frac{1}{2}, j, k}\right)  \right) \\ &- \lambda_2 \left(\mathbf{\hat{G}}_{2}\left( \widetilde{{\bm Q}}^{n,-}_{i, j+\frac{1}{2},k}, \widetilde{{\bm Q}}^{n,+}_{i, j+\frac{1}{2},k} \right)  - \mathbf{\hat{G}}_{2}\left( \widetilde{{\bm Q}}^{n,-}_{i, j-\frac{1}{2}, k}, \widetilde{{\bm Q}}^{n,+}_{i, j-\frac{1}{2}, k} \right)  \right) \\
			&- \lambda_3 \left(\mathbf{\hat{G}}_{3}\left( \widetilde{{\bm Q}}^{n,-}_{i, j,k+\frac{1}{2}}, \widetilde{{\bm Q}}^{n,+}_{i, j,k+\frac{1}{2}} \right)  - \mathbf{\hat{G}}_{3}\left( \widetilde{{\bm Q}}^{n,-}_{i, j, k-\frac{1}{2}}, \widetilde{{\bm Q}}^{n,+}_{i, j, k-\frac{1}{2}} \right)  \right),
		\end{aligned} 
	\end{equation} 
	where $\lambda_3 = \frac{\Delta t}{\Delta z}$, the limiting values $\widetilde{{\bm Q}}$ are derived from the PP modified values $\widetilde{{\bm W}}$, 
	and $\mathbf{\hat{G}}_{3}$ denotes the Lax--Friedrichs flux in the $z$-direction  similar to  \eqref{LFflux}. This scheme can be proven PP, following similar steps as in \Cref{PPTheorem}. For brevity, the details are omitted.

	\begin{theorem}
		If $\overline{\bm{Q}}^n_{ijk} \in \mathcal{G}$ for all $i$, $j$ and $k$, then $\overline{\bm{Q}}^{n+1}_{ijk}$ obtained by \eqref{FVM3D} always belongs to $\mathcal{G}$ under the CFL condition:
		\begin{equation}\label{CFL_3d}
			\Delta t \left( \frac{\alpha_1}{\Delta x} + \frac{\alpha_2}{\Delta y}  +\frac{\alphaLF_3}{\Delta z} \right)  \le \frac{1}{q},
		\end{equation} 
		where $q > 2$ is a fixed constant used in the PP limiter and can be arbitrarily specified in $(2, +\infty)$. 
	\end{theorem}
	
	\subsection{3D Finite Difference CT Method $S^{\Delta t}_{B}$ for System-B}
	We now detail the CT discretization of \eqref{FDequations} in the 3D case:
	\begin{equation}\label{discretization2_3D}
		\left\{
		\begin{aligned}
			B^{x, n+1}_{ijk} &= B^{x,n}_{ijk} - \Delta t \frac{ \Omega^{z, n + \frac{1}{2}}_{i, j +1, k} - \Omega^{z, n + \frac{1}{2}}_{i, j -1, k} }{2 \Delta y} + \Delta t \frac{ \Omega^{y, n + \frac{1}{2}}_{i, j, k+1} - \Omega^{y, n + \frac{1}{2}}_{i, j, k-1} }{2 \Delta z}, \\
			B^{y, n+1}_{ijk} &= B^{y,n}_{ijk} + \Delta t \frac{ \Omega^{z, n + \frac{1}{2}}_{i + 1, j,k} - \Omega^{z, n + \frac{1}{2}}_{i - 1, j, k} }{2 \Delta x} - \Delta t \frac{ \Omega^{x, n + \frac{1}{2}}_{i, j, k+1} - \Omega^{x, n + \frac{1}{2}}_{i, j, k-1} }{2 \Delta z} , \\
			B^{z, n + 1}_{ijk} &= B^{z,n}_{ijk} - \Delta t \frac{ \Omega^{y, n + \frac{1}{2}}_{i+1, j,k} - \Omega^{y, n + \frac{1}{2}}_{i-1, j, k } }{2 \Delta x} + \Delta t \frac{ \Omega^{x, n + \frac{1}{2}}_{i, j +1, k} - \Omega^{x, n + \frac{1}{2}}_{i, j -1, k} }{2 \Delta y}, \\
			\rho^n_{ijk} {\bm v}^{n+1}_{ijk} &= \rho^n_{ijk} {\bm v}^n_{ijk} - \Delta t {\bm B}^{n+\frac{1}{2}}_{ijk} \times \operatorname{curl} {\bm B}^{n + \frac{1}{2}}_{ijk},
		\end{aligned}
		\right.
	\end{equation}
	where the electric filed $\bm{\Omega}$ is defined in \eqref{def:Omega} and the discrete curl $\operatorname{curl} {\bm B}^{n + \frac{1}{2}}_{ijk}$ is given by
	\begin{equation*}
		\operatorname{curl} {\bm B}^{n + \frac{1}{2}}_{ijk} := \begin{pmatrix}
			\displaystyle 
			\frac{B^{z, n + \frac{1}{2}}_{i,j+1,k} - B^{z, n + \frac{1}{2}}_{i,j-1,k}}{2 \Delta y } - \frac{B^{y, n + \frac{1}{2}}_{i,j,k+1} - B^{y, n + \frac{1}{2}}_{i,j,k-1}}{2 \Delta z } \\
			\displaystyle 
			\frac{B^{x, n + \frac{1}{2}}_{i,j,k+1} - B^{x, n + \frac{1}{2}}_{i,j,k-1}}{2 \Delta z } -\frac{B^{z, n + \frac{1}{2}}_{i+1,j,k} - B^{z, n + \frac{1}{2}}_{i-1,j,k}}{2 \Delta x } \\
			\displaystyle 
			\frac{B^{y, n + \frac{1}{2}}_{i+1,j,k} - B^{y, n + \frac{1}{2}}_{i-1,j,k}}{2 \Delta x } - \frac{B^{x, n + \frac{1}{2}}_{i,j+1,k} - B^{x, n + \frac{1}{2}}_{i,j-1,k}}{2 \Delta y } 
		\end{pmatrix}.
	\end{equation*}
	After solving the nonlinear system \eqref{discretization2_3D}, we use the invariance of internal energy to update the mechanical energy $E$:
	\begin{equation}\label{updatingsource3D}
		E^{n+1}_{ijk} = E^{n}_{ijk} - \frac{1}{2} \rho^n_{ijk} \left|{\bm v}^n_{ijk}\right|^2 + \frac{1}{2} \rho^{n+1}_{ijk} \left|{\bm v}^{n+1}_{ijk}\right|^2.
	\end{equation}
	Note that the density remains unchanged in System-B, that is, $\rho^{n+1}_{ijk} = \rho^n_{ijk}$.\par 
	We define the discrete divergence of $\bm{B}$ at time $t_n$ at point $(x_i, y_j, z_k)$ as
	\begin{equation*}\label{definition:DDF3D}
		\left( \nabla \cdot {\bm B} \right)_{ijk}^{n} = \frac{B^{x,n}_{i+1, j, k} - B^{x,n}_{i-1,j,k}}{2 \Delta x} + \frac{B^{y,n}_{i, j+1, k} - B^{y,n}_{i,j-1,k}}{2 \Delta y}
		+ \frac{B^{z,n}_{i, j, k+1} - B^{z,n}_{i,j,k-1}}{2 \Delta z}.
	\end{equation*}
	The DDF condition of the magnetic field is exactly preserved by the CT method \eqref{discretization2_3D}. Additionally, we can rigorously prove that energy conservation is maintained under periodic boundary conditions or when the solution has compact support. The proofs are similar with those in \Cref{theorem:DDF2D,theorem:energyconservation}. 
	\begin{theorem}[DDF Property in 3D]\label{therorem:DDF3D}
		Assume $\bm{B}^n_{ijk}$ satisfies the DDF condition
		for all $i,j$ and $k$ then $\bm{B}^{n+1}_{ijk}$ also satisfies the DDF condition.
	\end{theorem}
	\begin{theorem}[Energy Identity in 3D]
		If the boundary condition is periodic or the solution has compact support then the CT scheme (\ref{discretization2_3D}) satisfies the following energy conservative identity:
		\begin{equation*}
			\sum_{i,j,k} \frac{1}{2} \rho^{n}_{ijk} \left| {\bm v}^{n}_{ijk} \right|^2  + \frac{1}{2} \left| {\bm B}^{n}_{ijk} \right|^2 = \sum_{i,j,k} \frac{1}{2} \rho^{n+1}_{ijk} \left| {\bm v}^{n+1}_{ijk} \right|^2  + \frac{1}{2} \left| {\bm B}^{n+1}_{ijk} \right|^2.
		\end{equation*}
	\end{theorem}
	\begin{corollary}[Total Energy Conservation in 3D]\label{thm:totalenergy3D}
		If the boundary condition is periodic or the solution has compact support, then the scheme \eqref{discretization2_3D} with \eqref{updatingsource3D} conserves the total energy:
		\begin{equation*}
			\sum_{i,j,k} \left( E_{ijk}^n + \frac{1}{2} \left| {\bm B}^{n}_{ijk} \right|^2 \right) = \sum_{i,j,k} \left( E^{n+1}_{ijk} + \frac{1}{2} \left| {\bm B}^{n+1}_{ijk} \right|^2 \right).
		\end{equation*}
	\end{corollary}

	\section{Numerical Experiments}\label{S.5}
In this section, we demonstrate the performance of the proposed PPCT method on several benchmark or challenging examples in the 2D case. Unless otherwise specified, the error tolerance $\varepsilon_{tol}$ in \eqref{eq:stopping} is set to $10^{-10}$, the adiabatic index $\gamma = \frac{5}{3}$, the value of $q$ is chosen as $3$ in the PP limiter, and the CFL number $C_{\tt CFL} = \frac2{q}=\frac{2}{3}$. 
	The SSP second-order Runge--Kutta method is used for time discretization of System-A. As discussed in \Cref{rem:stability2}, the time step-size $\Delta t$ for stability and the PP property of the PPCT scheme is only restricted by the CFL condition \eqref{CFLcondtion} for System-A. Therefore, we take  
	\begin{equation*}\label{CFL}
		\Delta t = \frac{C_{\tt CFL}}{\frac{\alpha_1 }{\Delta x} + \frac{ \alpha_2}{\Delta y}}.
	\end{equation*}

	\begin{expl}[Smooth Isentropic Vortex]\label{Ex:vortex} \rm 
		We start with the smooth isentropic vortex to test the accuracy. The initial conditions are defined as:
		\begin{equation*}\label{initial:vortex}
			\left( \rho, {\bm v}, {\bm B}, p \right) = \left( 1, 1 + \delta v_1, 1 + \delta v_2 , 0, \delta B_1, \delta B_2, 0, 1 + \delta p \right),
		\end{equation*}
		with $\left( \delta v_1, \delta v_2 \right) = \frac{\mu}{\sqrt{2} \pi}e^{0.5(1 - r^2)}\left(-y, x\right)$, $\left(\delta B_1, \delta B_2\right)=\frac{\mu}{2 \pi} e^{0.5\left(1-r^2\right)}(-y, x)$, and  $\delta p=-\frac{\mu^2\left(1+r^2\right)}{8 \pi^2} e^{1-r^2}$, where $r^2 = x^2 + y^2$. 
		First, we set the vortex strength to $\mu = 1$, which results in a mild vortex problem. By selecting a value of $q$ close to $2$ in the PP limiter, we allow the theoretical CFL number to approach $1$. As shown in \Cref{tab:Ex-Vortex-0.5}, second-order accuracy is observed under these conditions. 
		Next, we increase the vortex strength to $\mu = 5.389489439$, which significantly lowers the pressure at the vortex center to $5.3 \times 10^{-12}$. The errors and convergence rates of velocity and magnetic field are presented in \Cref{tab:Ex-Vortex-0.3} computed on different meshes of $N \times N$ cells. Here, we observe optimal convergence rates for $l^1$ and $l^2$ error, although the $l^{\infty}$ error reduces noticeably. This reduction does not imply that the PP limiter destroys accuracy. It is well known that when using a MUSCL type scheme with slope limiting, the order of $l^{\infty}$ error often reduces to first-order near extremum points. One effective method to mitigate this issue is to relax the slope limitation by selecting a larger value for $q$.  
		When we set $q = 5$ (resulting in $C_{\tt CFL} = \frac{2}{q} = 0.4$), the order of the $l^{\infty}$ error recovers, as shown in \Cref{tab:Ex-Vortex-0.2}. The last column of the table records the average iteration count per time step for the iterative algorithm \eqref{iter} used to solve system (\ref{eq:nonlinears}). \Cref{fig:iteration error} illustrates how the iteration error ${\mathcal E}_{k}$ decreases with the number of iterations in a typical single time step, where we set $\varepsilon_{tol}=10^{-14}$ to show the convergence behavior. We observe exponential convergence. In general, the average iteration count is about $5$ to $9$, with ${\mathcal E}_k$ decreasing below $10^{-14}$ (near machine precision). Moreover, in all our following experiments, the iteration count typically stays below $10$ and never exceeds $20$, demonstrating the efficiency of the iterative process.

		\begin{table}[!htb] 
			\centering
			\belowrulesep=0pt
			\aboverulesep=0pt
			\caption{Errors in ${\bm B}_h$ and ${\bm v}_h$ at $t=0.05$ for the mild vortex with $\mu = 1.0$, computed by PPCT method on a $N\times N$ mesh with $q=2.01$ and $C_{\tt CFL} = \frac{2}{q} \approx 1$. 
			}
			\label{tab:Ex-Vortex-0.5}
			\setlength{\tabcolsep}{3mm}{
				\begin{tabular}{c|ccccccccc}
					\toprule[1.5pt]
					\multirow{2}{*}{ } &
					\multirow{2}{*}{$ N$} &
					\multicolumn{2}{c}{$ l^{1} $ norm} &
					\multicolumn{2}{c}{$ l^{2} $ norm} &
					\multicolumn{2}{c}{$ l^{\infty} $ norm} &
					\multirow{2}{*}{ite \#} \\
					\cmidrule(r){3-4} \cmidrule(r){5-6} \cmidrule(l){7-8}
					& & error & order &  error & order &  error & order \\	
					\midrule[1.5pt]
					\multirow{5}{*}{${\bm B}_h$} & $64$ & 3.35e-05 & {-} &  1.27e-04 & {-} &  1.23e-03 & {-} &   9.0  \\
					& $128$ &   5.56e-06& 2.56&  2.11e-05& 2.59&  2.11e-04& 2.55&     7.0\\
					& $256$ &  1.10e-06& 2.34&  4.08e-06& 2.37&  4.17e-05& 2.34  &  5.5\\
					& $512$ &  2.49e-07& 2.14&  9.20e-07& 2.15&  1.02e-05& 2.03  &  4.9\\
					& $1024$ &  6.03e-08& 2.05&  2.23e-07& 2.05&  2.56e-06& 2.00 &  4.9 \\
					\midrule
					\multirow{5}{*}{${\bm v}_h$} & $64$ &  2.05e-05& {-}&  7.16e-05& {-}&  6.11e-04& {-}  & 9.0  \\
					& $128$ &  5.30e-06& 1.96&  1.86e-05& 1.95&  1.64e-04& 1.90 &7.0\\
					& $256$ &  1.34e-06& 1.98&  4.71e-06& 1.98&  4.19e-05& 1.97  &5.5 \\
					& $512$ &  3.36e-07& 2.00&  1.18e-06& 2.00&  1.05e-05& 2.00  &4.9 \\
					& $1024$ &  8.41e-08& 2.00&  2.96e-07& 2.00&  2.64e-06& 2.00 &4.9 \\
					\bottomrule[1.5pt]
				\end{tabular}
			}
		\end{table}
		\begin{table}[!htb] 
			\centering
			\belowrulesep=0pt
			\aboverulesep=0pt
			\caption{Errors in ${\bm B}_h$ and ${\bm v}_h$ at $t=0.05$ for the extreme vortex with $\mu = 5.389489439$, computed by PPCT method on a $N\times N$ mesh with $q=3$ and $C_{\tt CFL} =\frac{2}{3}$. 
			}
			\label{tab:Ex-Vortex-0.3}
			\setlength{\tabcolsep}{3mm}{
				\begin{tabular}{c|ccccccccc}
					\toprule[1.5pt]
					\multirow{2}{*}{ } &
					\multirow{2}{*}{$ N$} &
					\multicolumn{2}{c}{$ l^{1} $ norm} &
					\multicolumn{2}{c}{$ l^{2} $ norm} &
					\multicolumn{2}{c}{$ l^{\infty} $ norm} &
					\multirow{2}{*}{ite \#} \\
					\cmidrule(r){3-4} \cmidrule(r){5-6} \cmidrule(l){7-8}
					& & error & order &  error & order &  error & order \\	
					\midrule[1.5pt]
					\multirow{5}{*}{${\bm B}_h$}& $64$ &  1.10e-04& {-}&  3.82e-04& {-}&  3.49e-03& {-}& 8.5\\
					& $128$ &  2.82e-05& 1.96&  9.92e-05& 1.94&  1.06e-03& 1.72 &    8.0\\
					& $256$ &  7.10e-06& 1.99&  2.51e-05& 1.98&  3.64e-04& 1.54 &   7.0\\
					& $512$ &  1.78e-06& 1.99&  6.31e-06& 1.99&  1.18e-04& 1.62 &  5.9 \\
					& $1024$ &  4.47e-07& 2.00&  1.58e-06& 2.00&  4.07e-05& 1.54 & 5.7\\
					\midrule
					\multirow{5}{*}{${\bm v}_h$}& $64$&  3.04e-04& {-}&  1.41e-03& {-}&  2.02e-02& {-} & 8.5\\
					& $128$ &  6.75e-05& 2.17&  4.33e-04& 1.71&  1.12e-02& 0.85 &    8.0 \\
					& $256$ &  1.51e-05& 2.16&  1.33e-04& 1.71&  5.78e-03& 0.95  &   7.0\\
					& $512$&  3.28e-06& 2.20&  3.61e-05& 1.88&  2.36e-03& 1.29 & 5.9\\
					& $1024$&  7.20e-07& 2.19&  9.42e-06& 1.94&  9.18e-04& 1.37 & 5.7\\
					\bottomrule[1.5pt]
				\end{tabular}
			}
		\end{table}
		\begin{table}[!htb] 
			\centering
			\belowrulesep=0pt
			\aboverulesep=0pt
			\caption{Errors in ${\bm B}_h$ and ${\bm v}_h$ at $t=0.05$ for the extreme vortex with $\mu = 5.389489439$, computed by PPCT method on a $N\times N$ mesh with $q=5$ and $C_{\tt CFL} =\frac{2}{5}$. 
			}
			\label{tab:Ex-Vortex-0.2}
			\setlength{\tabcolsep}{2.66mm}{
				\begin{tabular}{c|ccccccccc}
					\toprule[1.5pt]
					\multirow{2}{*}{ } &
					\multirow{2}{*}{$N$} &
					\multicolumn{2}{c}{$ l^{1} $ norm} &
					\multicolumn{2}{c}{$ l^{2} $ norm} &
					\multicolumn{2}{c}{$ l^{\infty} $ norm} &
					\multirow{2}{*}{ite \#} \\
					\cmidrule(r){3-4} \cmidrule(r){5-6} \cmidrule(l){7-8}
					& & error & order &  error & order &  error & order \\	
					\midrule[1.5pt]
					\multirow{5}{*}{${\bm B}_h$} & $64$  & 1.09e-04 & {-} &  3.80e-04 & {-} & 3.36e-03 & {-} & 7.3\\
					& $128$  & 2.81e-05 & 1.96 & 9.84e-05 & 1.95 & 9.18e-04 & 1.87 &  6.8\\
					& $256$ &  7.07e-06& 1.99&  2.48e-05& 1.99&  2.35e-04& 1.97 & 5.9\\
					& $512$ &  1.77e-06& 1.99&  6.23e-06& 1.99&  5.94e-05& 1.98 & 5.0\\
					& $1024$&  4.44e-07& 2.00&  1.56e-06& 2.00&  1.50e-05& 1.99 & 4.0\\
					\midrule
					\multirow{5}{*}{${\bm v}_h$} 
					& $64$ &  2.64e-04& {-}&  1.14e-03& {-}&  1.53e-02& {-} & 7.3\\
					& $128$&  5.27e-05& 2.36&  2.59e-04& 2.14&  4.35e-03& 1.81 & 6.8\\
					& $256$&  1.11e-05& 2.24&  5.76e-05& 2.17&  1.16e-03& 1.91 & 5.9\\
					& $512$ &  2.47e-06& 2.17&  1.28e-05& 2.17&  3.00e-04& 1.95 & 5.0\\
					& $1024$ &  5.68e-07& 2.12&  2.87e-06& 2.15&  7.61e-05& 1.98 & 4.0\\
					\bottomrule[1.5pt]
				\end{tabular}
			}
		\end{table}
		\begin{figure}[htbp]
			\centering
			\includegraphics[width=0.618\textwidth]{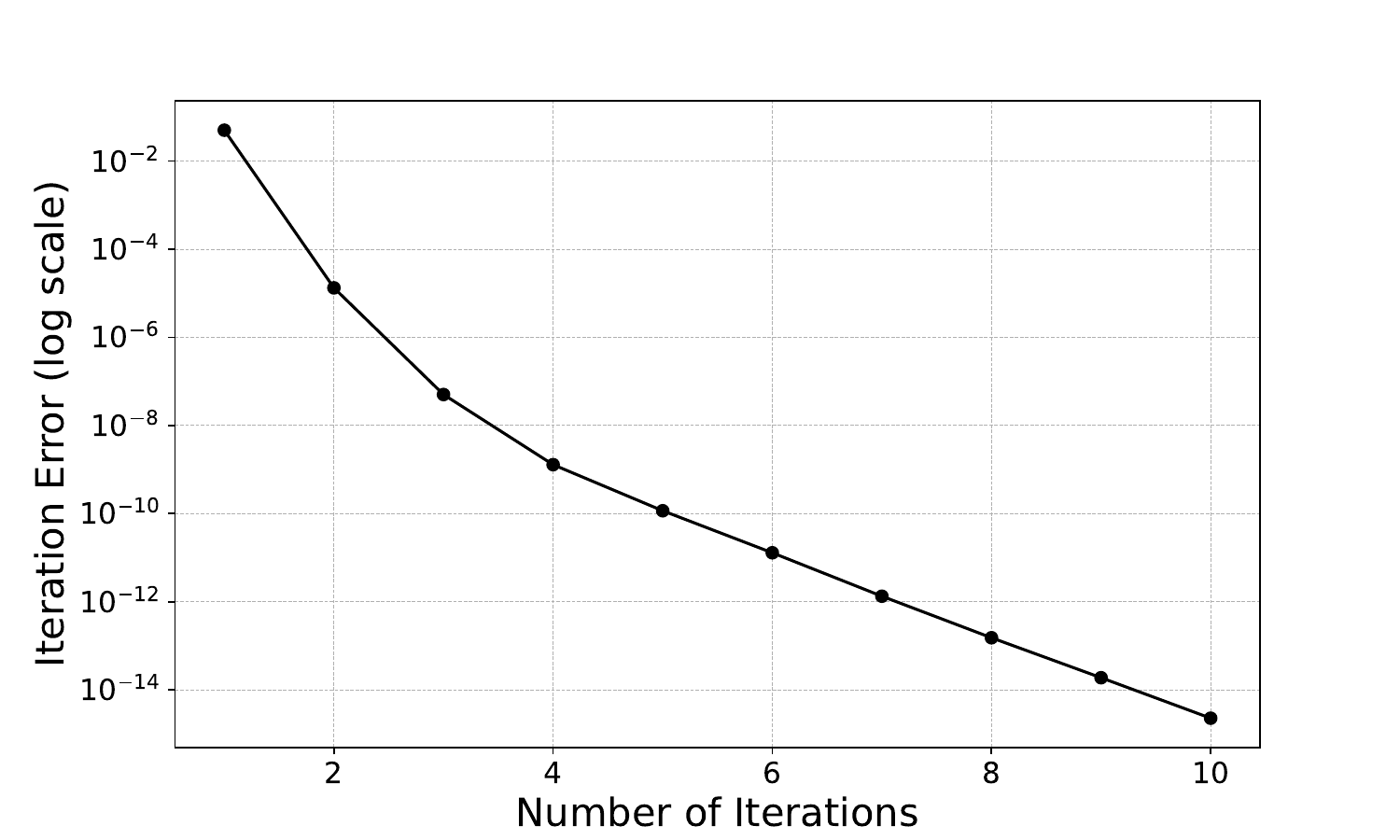} 
			\caption{\Cref{Ex:vortex} : Iteration error varies with the number of iterations in the case of $C_{\tt CFL} = \frac{2}{3}$ on the $1024 \times 1024$ mesh.} 
			\label{fig:iteration error}
		\end{figure} 
	\end{expl}

	\begin{expl}[Orszag--Tang Problem]\label{Ex:OT} \rm  
		The Orszag--Tang problem is a well-known benchmark for testing the shock-capturing capability of numerical methods for MHD. Although the initial conditions are smooth, as time evolves, multiple shocks form and interact with each other, creating a complex flow structure. The initial conditions are set as follows:
		\begin{equation*}
			(\rho, {\bm v}, {\bm B}, p) = \left(\gamma^2, -\sin y, \sin x, 0, -\sin y, \sin (2x), 0, \gamma \right).
		\end{equation*}
		The computational domain $\left[0, 2\pi \right]^2$ is divided into $400 \times 400$ uniform rectangular cells, and periodic boundary conditions are applied. \Cref{fig:Ex-OT-density} shows the density contours at times $t = 2$ and $t = 4$. Our PPCT scheme demonstrates good performance in capturing both shocks and smooth flow regions. The results are consistent with those found in \cite{DingWu2024SISCMHD, Li2011, Guillet2019}. As the solution progresses from $t=2$ to $t=4$, the flow becomes increasingly intricate. 
		It has been reported in \cite{jiang1999high} that negative numerical pressure could easily occur around $t \approx 3.9$ in this test. However, our scheme avoids this issue due to its provably PP property. \Cref{fig:Ex-OT-cut_off} presents the density and thermal pressure profiles along the line $y = 0.625\pi$ at $t = 3$. The captured  discontinuities are in agreement with findings reported in \cite{LiuWu2024OEDGforMHD}, where high-resolution simulations revealed that shock profiles are typically sharper in these regions.

		\begin{figure}[!htb]
			\centering
			\begin{subfigure}{0.48\textwidth}
				\includegraphics[width=\textwidth]{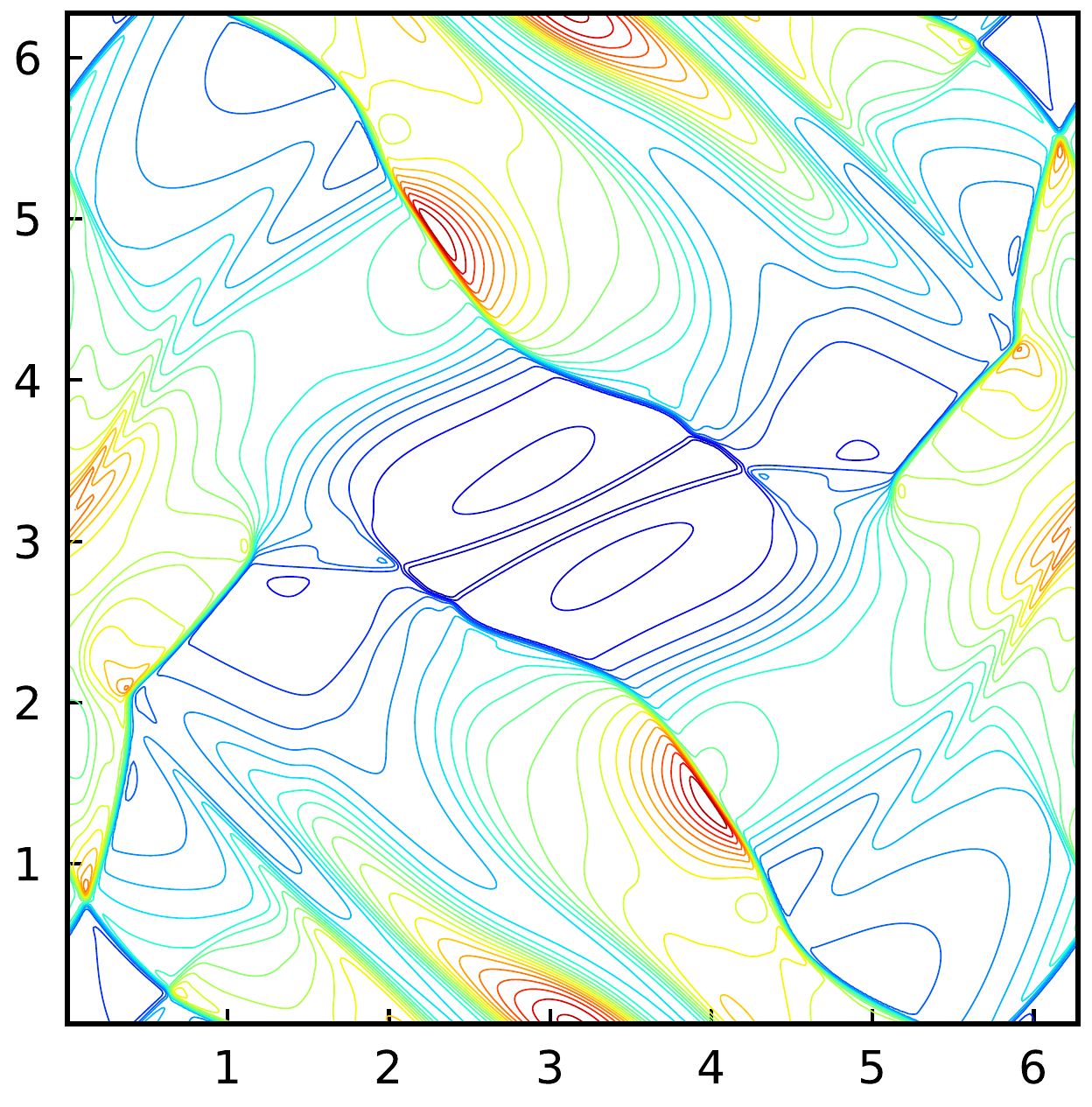}
			\end{subfigure}
			\hfill
			\begin{subfigure}{0.48\textwidth}
				\includegraphics[width=\textwidth]{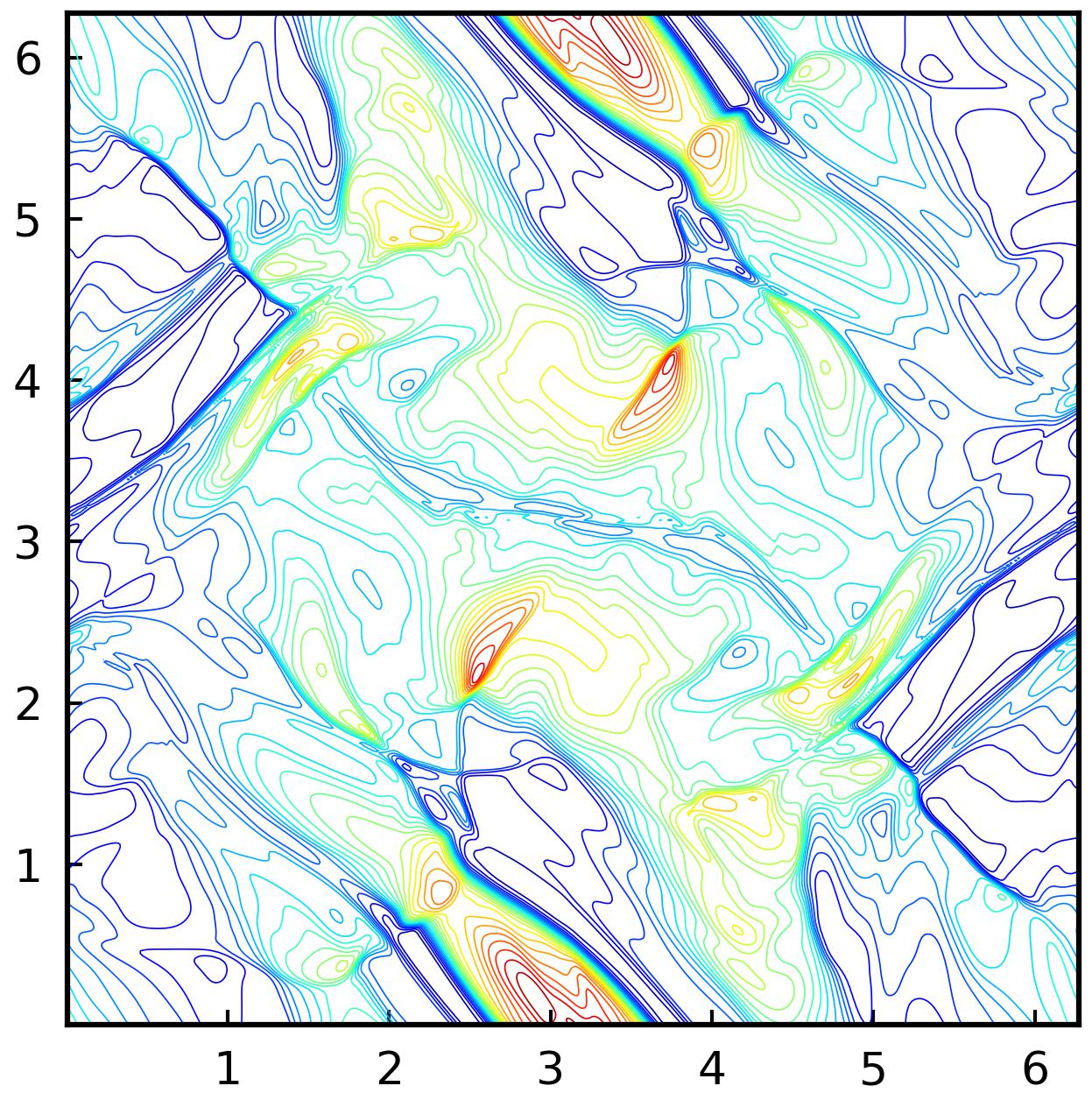}
			\end{subfigure}
			\caption{\Cref{Ex:OT}: Density of Orszag-Tang turbulence problem at $t = 2$ (left) and $t = 4$(right). Twenty-four contour lines are displayed. 
			}
			\label{fig:Ex-OT-density}
		\end{figure}
		\begin{figure}[!htb]
			\centering
			\begin{subfigure}{0.48\textwidth}
				\includegraphics[width=\textwidth]{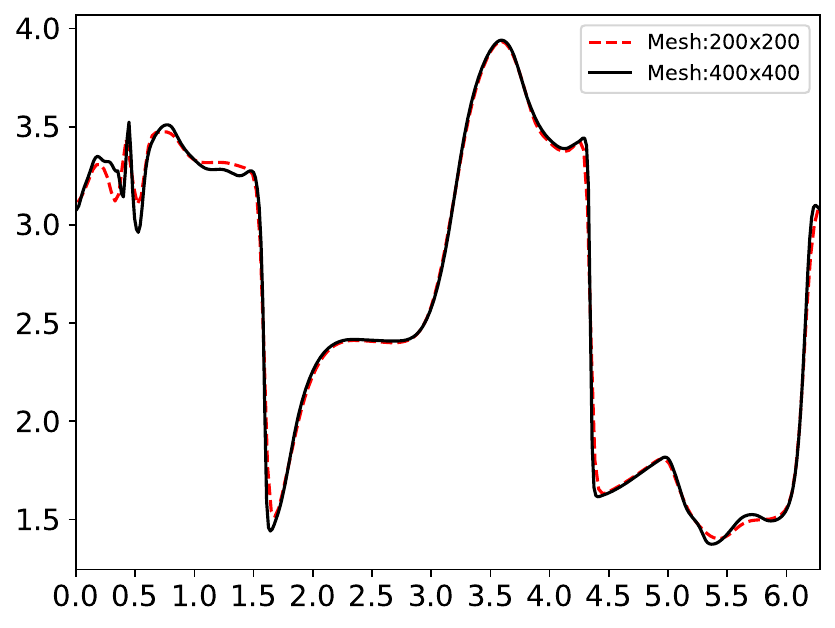}
			\end{subfigure}
			\hfill
			\begin{subfigure}{0.48\textwidth}
				\includegraphics[width=\textwidth]{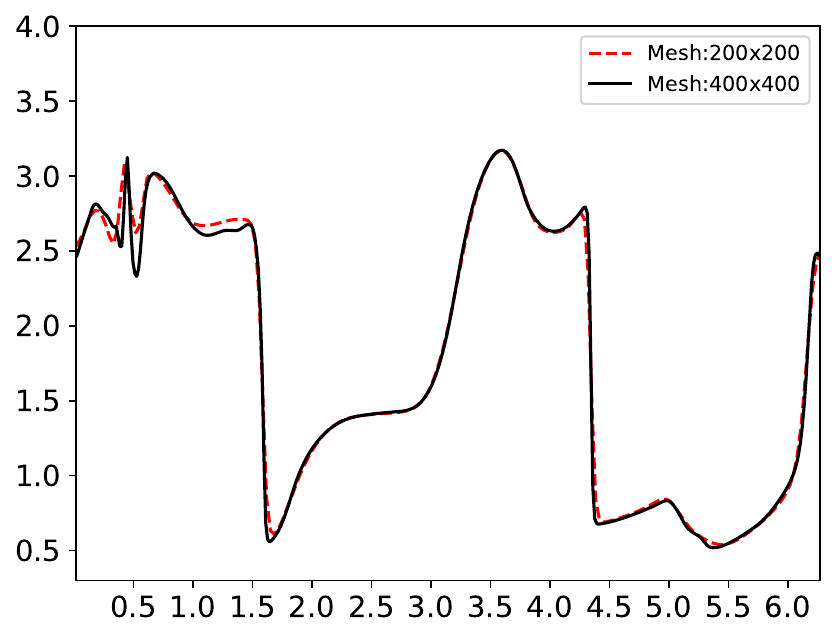}
			\end{subfigure}
			\caption{\Cref{Ex:OT}:Profiles of density (left) and thermal pressure (right) along line $y=0.625 \pi$ at $t = 3$. 
			}
			\label{fig:Ex-OT-cut_off}
		\end{figure} 
	\end{expl}

	\begin{expl}[Rotor Problem]\label{Ex:Rotor}\rm
		Next, we consider the rotor problem, which serves as another benchmark for the MHD system. This test simulates a rapidly rotating disk embedded in a uniform-density plasma. The initial conditions are given by
		\begin{equation*}
			(\rho, {\bm v}, {\bm B}, p)= \begin{cases}
				\left(10,-(y-0.5) / r_1,(x-0.5) / r_1, 0,2.5 / \sqrt{4 \pi}, 0,0,0.5\right), & r \leq r_1, \\
				(1+9 \phi,-\phi(y-0.5) / r, \phi(x-0.5) / r, 0,2.5 / \sqrt{4 \pi}, 0,0,0.5), & r_1<r \leq r_2, \\
				(1,0,0,0,2.5 / \sqrt{4 \pi}, 0,0,0.5), & r_2<r,
			\end{cases}
		\end{equation*}
		where $\phi=(r_2-r)/(r_2-r_1)$, $r=\sqrt{(x - 0.5)^2+(y-0.5)^2}$, $r_1=0.1$, and $r_2=0.115$.
		
		The computational domain is $[0, 1]^2$, discretized into a $400 \times 400$ uniform rectangular mesh with outflow boundary conditions. The simulation is run until $t=0.295$. 
		\Cref{fig:Ex-Rotor1} presents the contour plots of density $\rho$, thermal pressure $p$, magnetic pressure $\frac{\left| {\bm B} \right|^2}{2}$, and the Mach number $\frac{{\bm v}}{c}$. Our numerical results effectively preserve the circular rotation pattern, a feature that T\'oth identified as challenging for certain MHD numerical schemes \cite{Toth2000}. Typically, large divergence errors in the magnetic field can cause distortions in the numerical solution, particularly around the central region of the Mach number, as noted in \cite{Li2005}. However, our simulation results exhibit no such distortions in the central almond-shaped disk region, indicating that our PPCT schemes effectively control divergence errors. 
		\Cref{fig:Ex-Rotor2} shows a zoomed-in view of the Mach number for a comparison of reconstructions using conservative variables versus primitive variables. The reconstruction using conservative variables reveals some nonphysical structures, which are alleviated by reconstructing the primitive variables. 
		A further comparison on a refined mesh is shown in \Cref{fig:Ex-Rotor33}, once again highlighting the better performance of reconstruction using  primitive variables. 
		Due to the rapid changes in the solution, spurious oscillations and negative pressure can easily form in the central region. However, our results exhibit high resolution, and no negative pressure is observed. 
		For clarity, we also present in \Cref{fig:Ex-Rotor-cut_off} the profiles of the Mach number and $B_2$, sliced at $t = 0.295$ along the horizontal axes of the simulation domain. The scheme using reconstruction with conservative variables exhibits significant overshoots and spurious oscillations. In contrast, the PPCT scheme with reconstruction using primitive variables captures all the waves accurately. Our results align well with those reported in \cite{DingWu2024SISCMHD,LiuWu2024OEDGforMHD}.

		\begin{figure}[!htb]
			\centering		
			\begin{subfigure}{0.48\textwidth}
				\includegraphics[width=\textwidth]{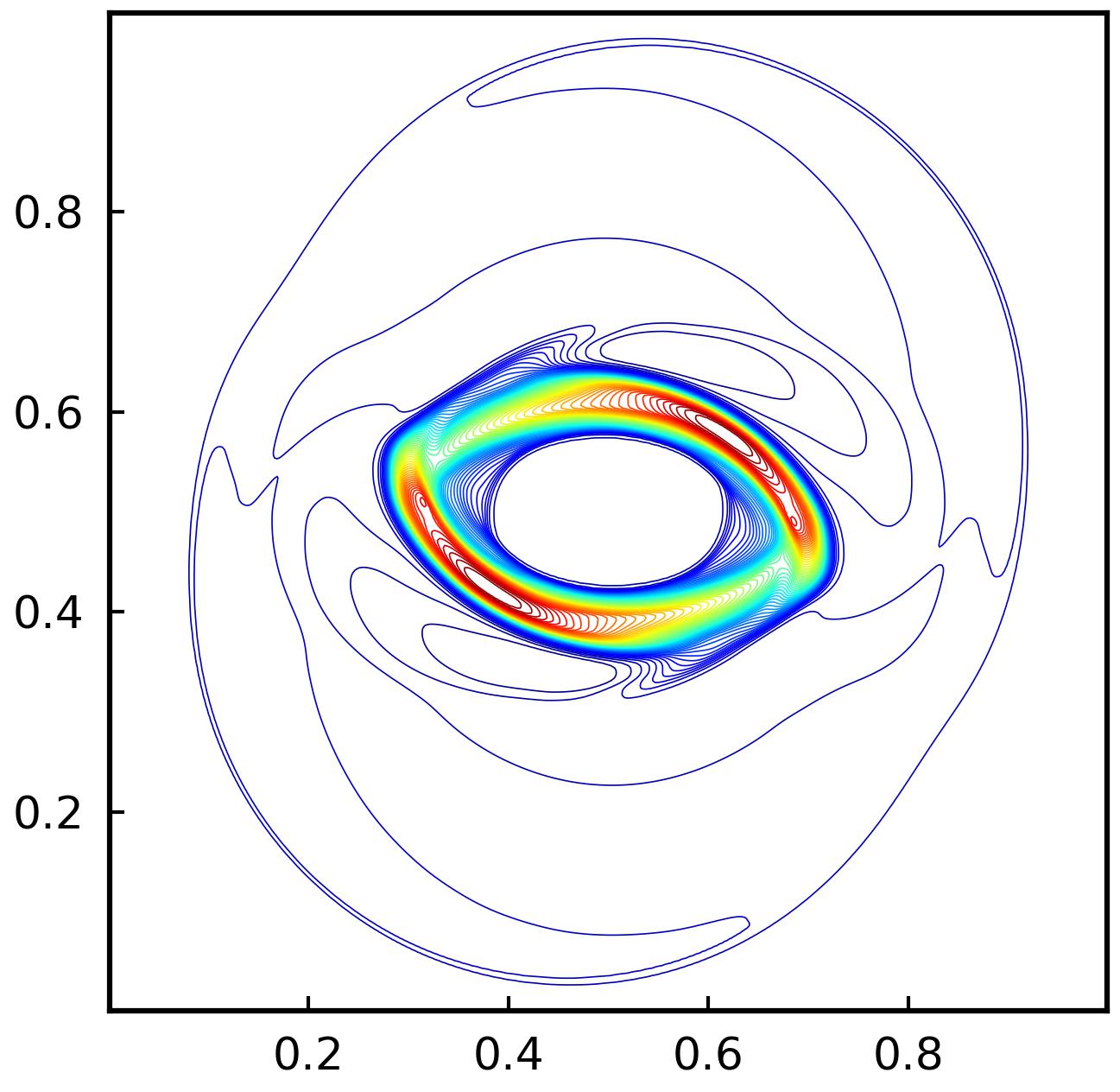}
			\end{subfigure}
			\hfill
			\begin{subfigure}{0.48\textwidth}
				\includegraphics[width=\textwidth]{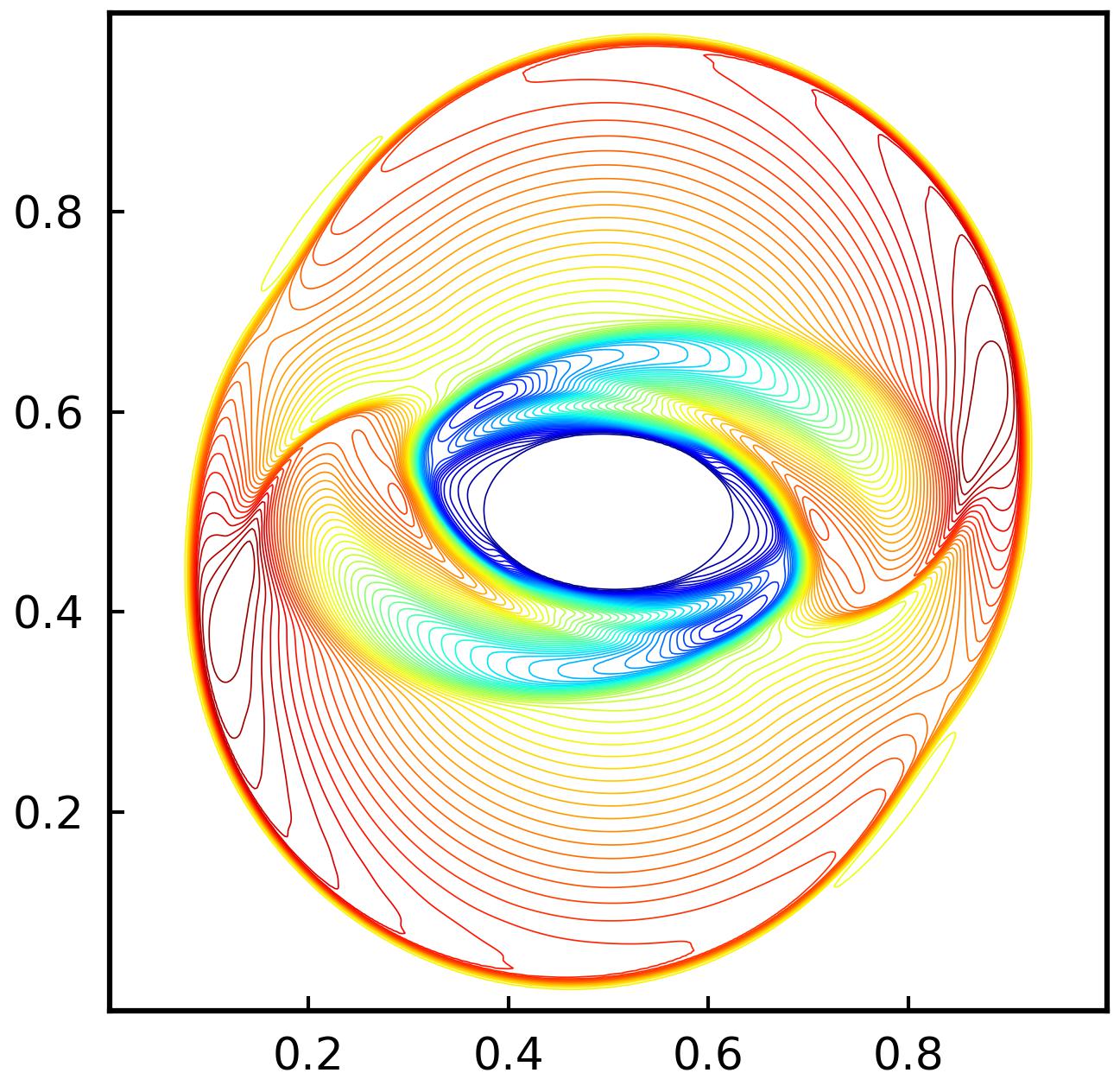}
			\end{subfigure}
			
			\begin{subfigure}{0.48\textwidth}
				\includegraphics[width=\textwidth]{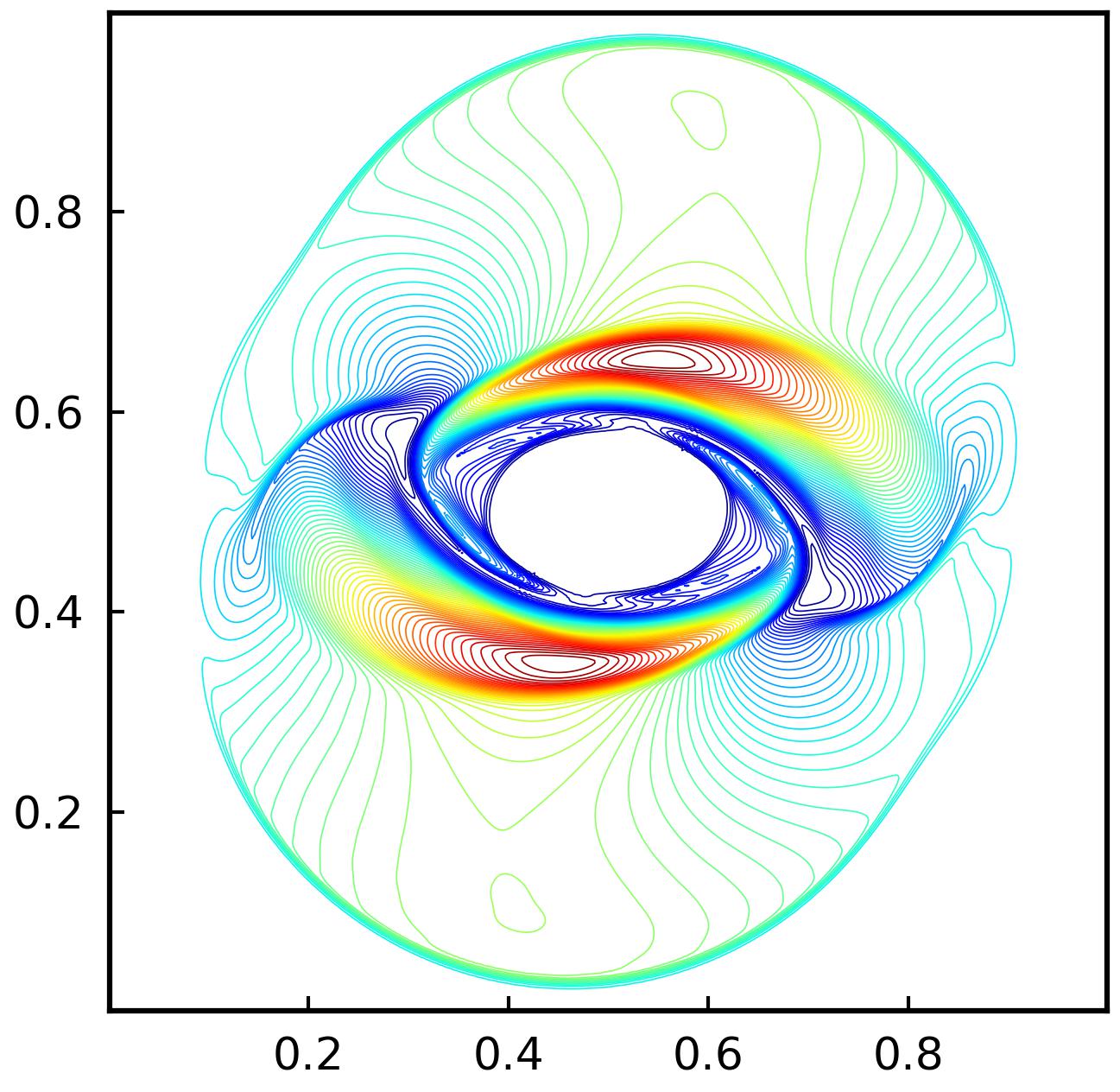}
			\end{subfigure}
			\hfill
			\begin{subfigure}{0.48\textwidth}
				\includegraphics[width=\textwidth]{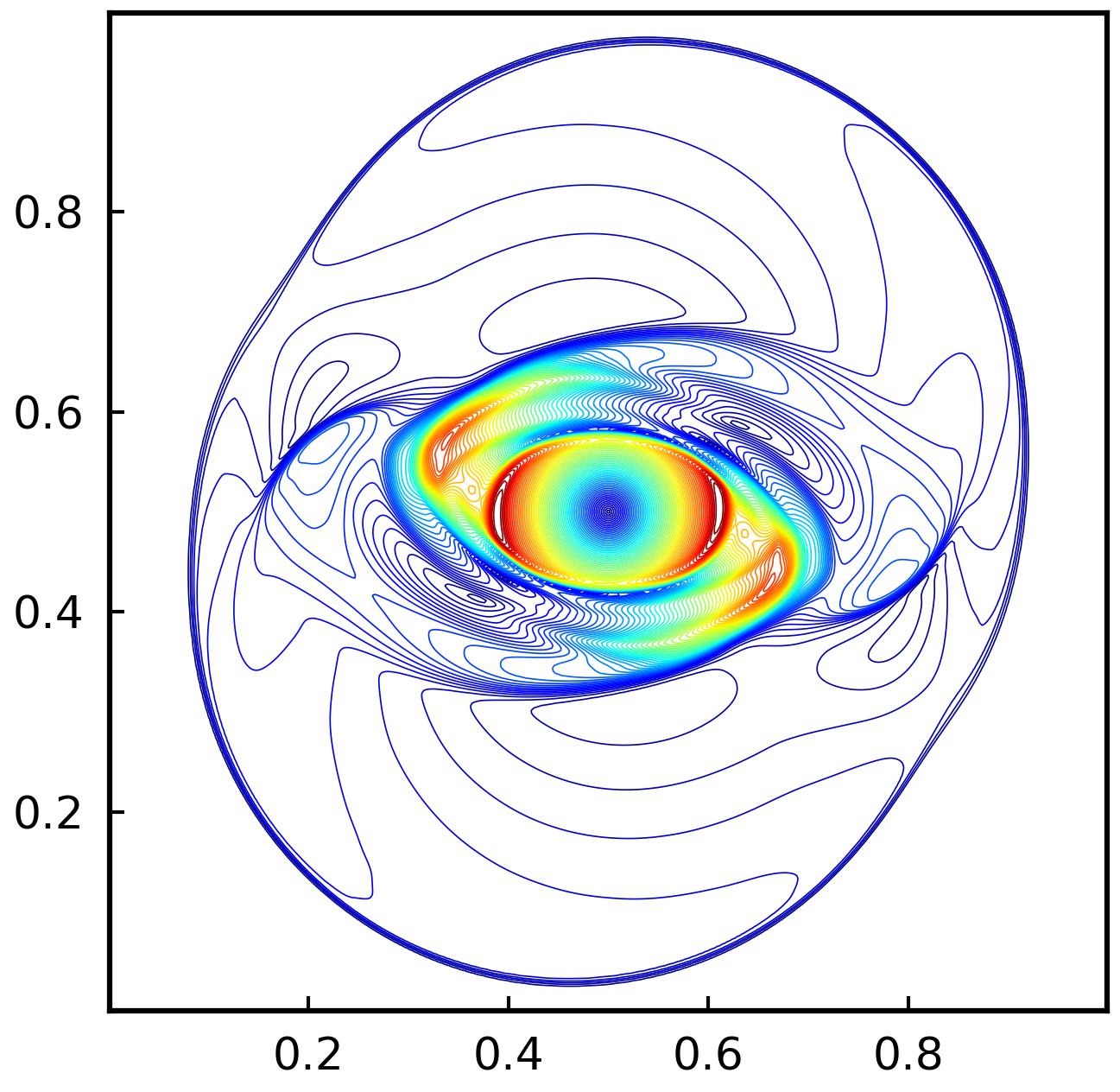}
			\end{subfigure}
			\caption{\Cref{Ex:Rotor}: The density (top-left), thermal pressure (top-right), magnetic pressure (bottom-left), and Mach number (bottom-right) for the rotor problem at $t = 0.295$. 
			}
			\label{fig:Ex-Rotor1}
		\end{figure}
		
		\begin{figure}[!htb]
			\centering
			\begin{subfigure}{0.48\textwidth}
				\includegraphics[width=\textwidth]{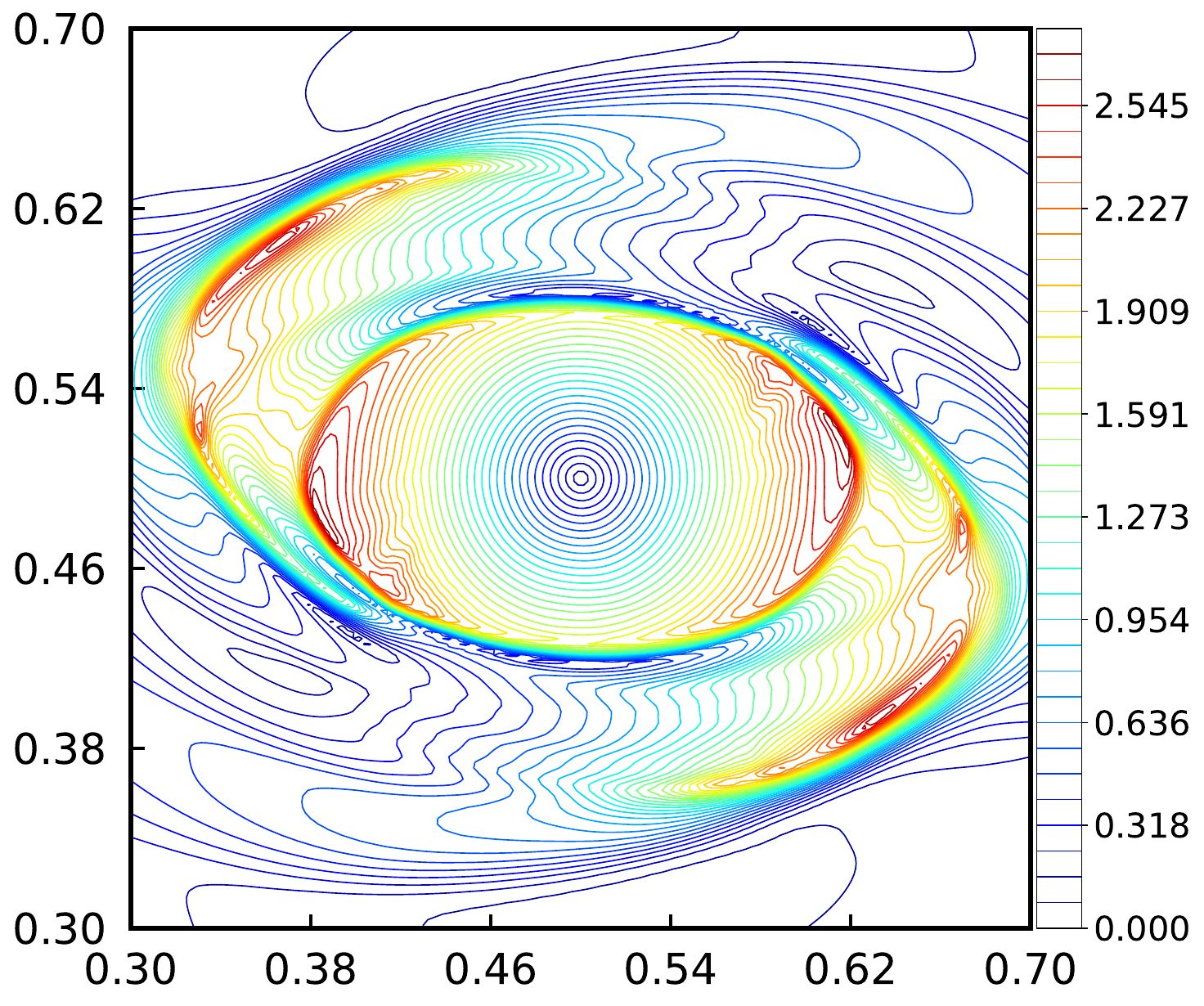}
			\end{subfigure}
			\hfill
			\begin{subfigure}{0.48\textwidth}
				\includegraphics[width=\textwidth]{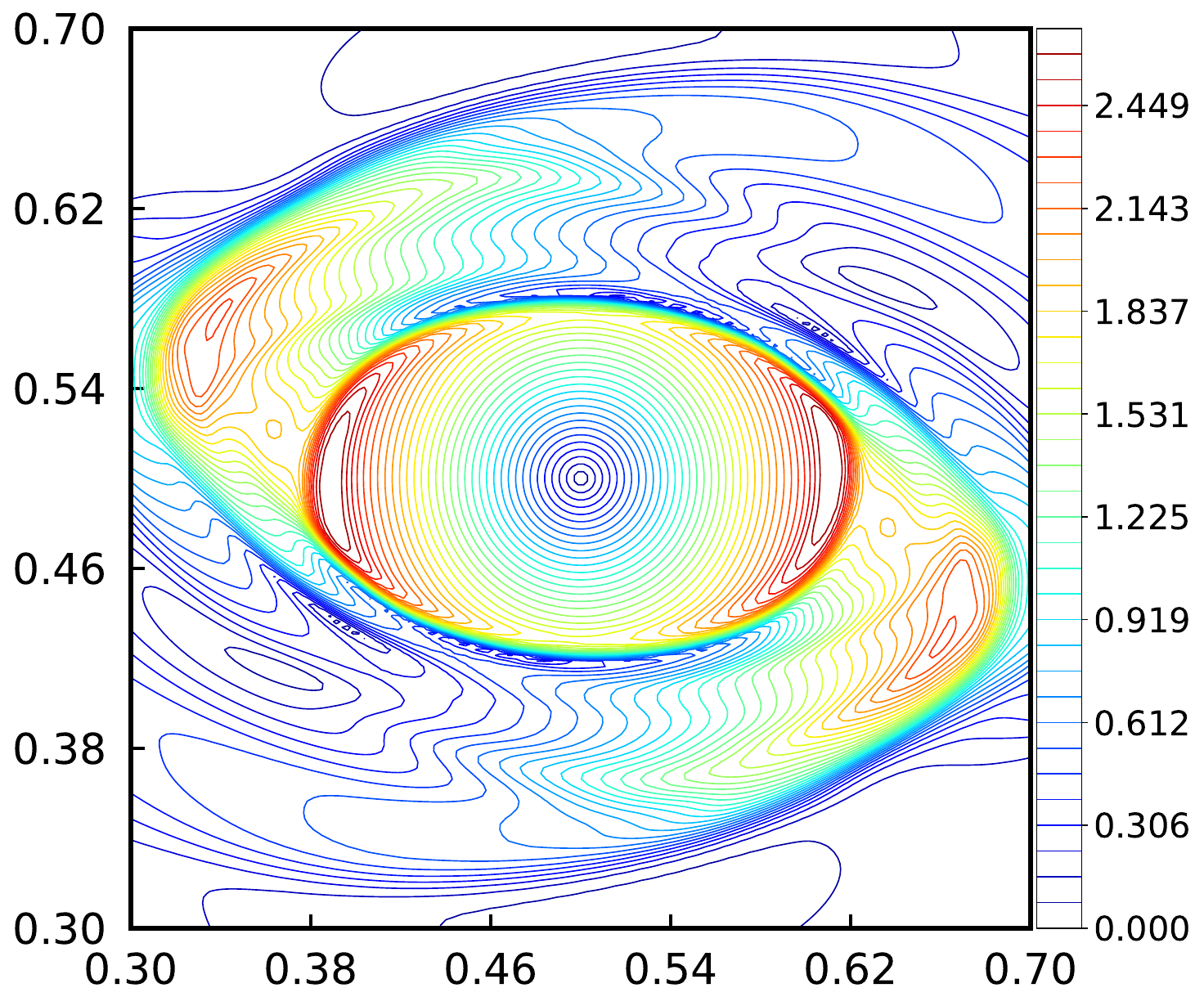}
			\end{subfigure}
			\caption{\Cref{Ex:Rotor}: Comparison of reconstructions using conservative variables (left) vs.~primitive variables (right) for the Mach number computed by the PPCT scheme with $400 \times 400$ cells. 
			}
			\label{fig:Ex-Rotor2}
		\end{figure}

		\begin{figure}[!htb]
			\centering
			\begin{subfigure}{0.48\textwidth}
				\includegraphics[width=\textwidth]{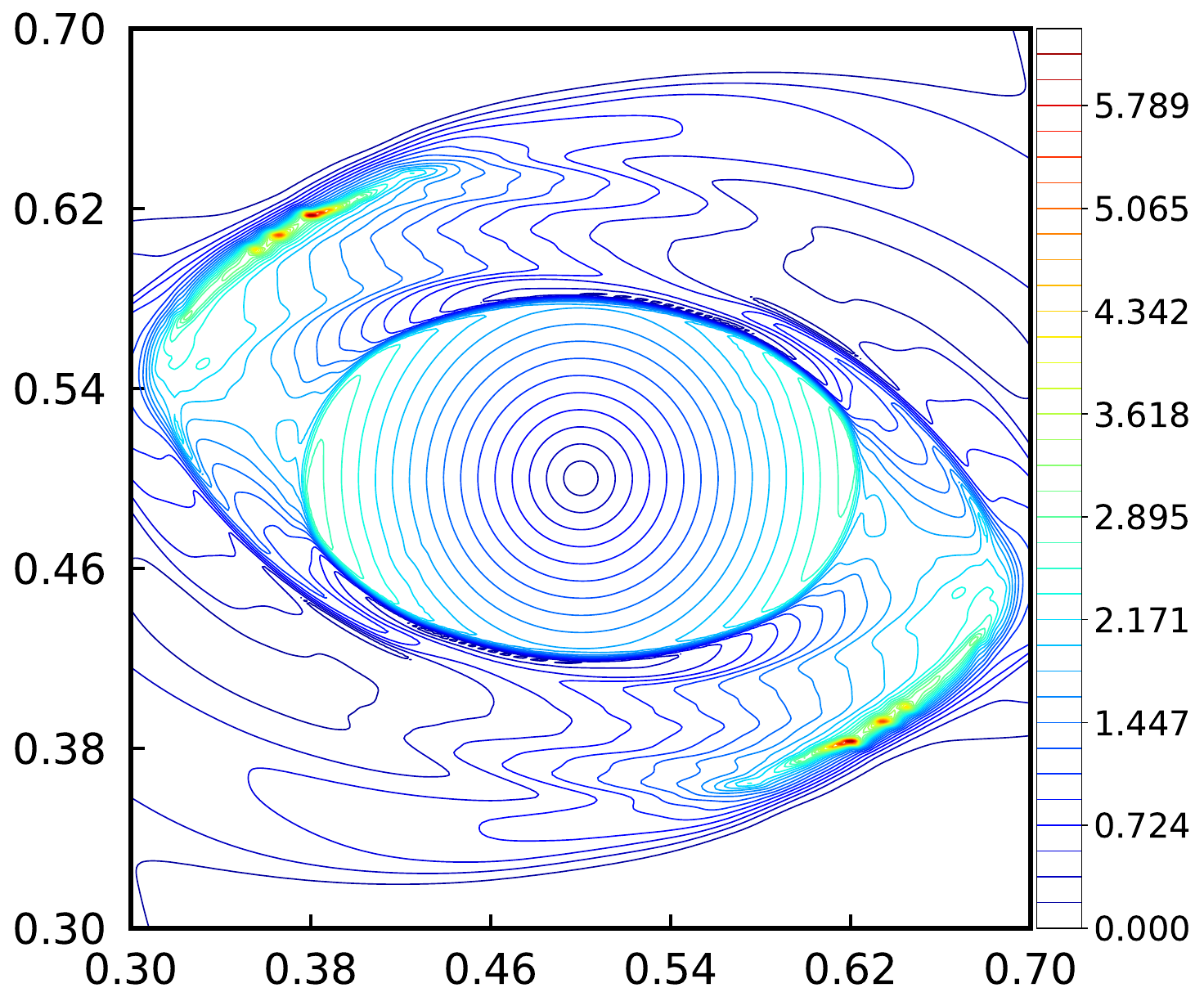}
			\end{subfigure}
			\hfill
			\begin{subfigure}{0.48\textwidth}
				\includegraphics[width=\textwidth]{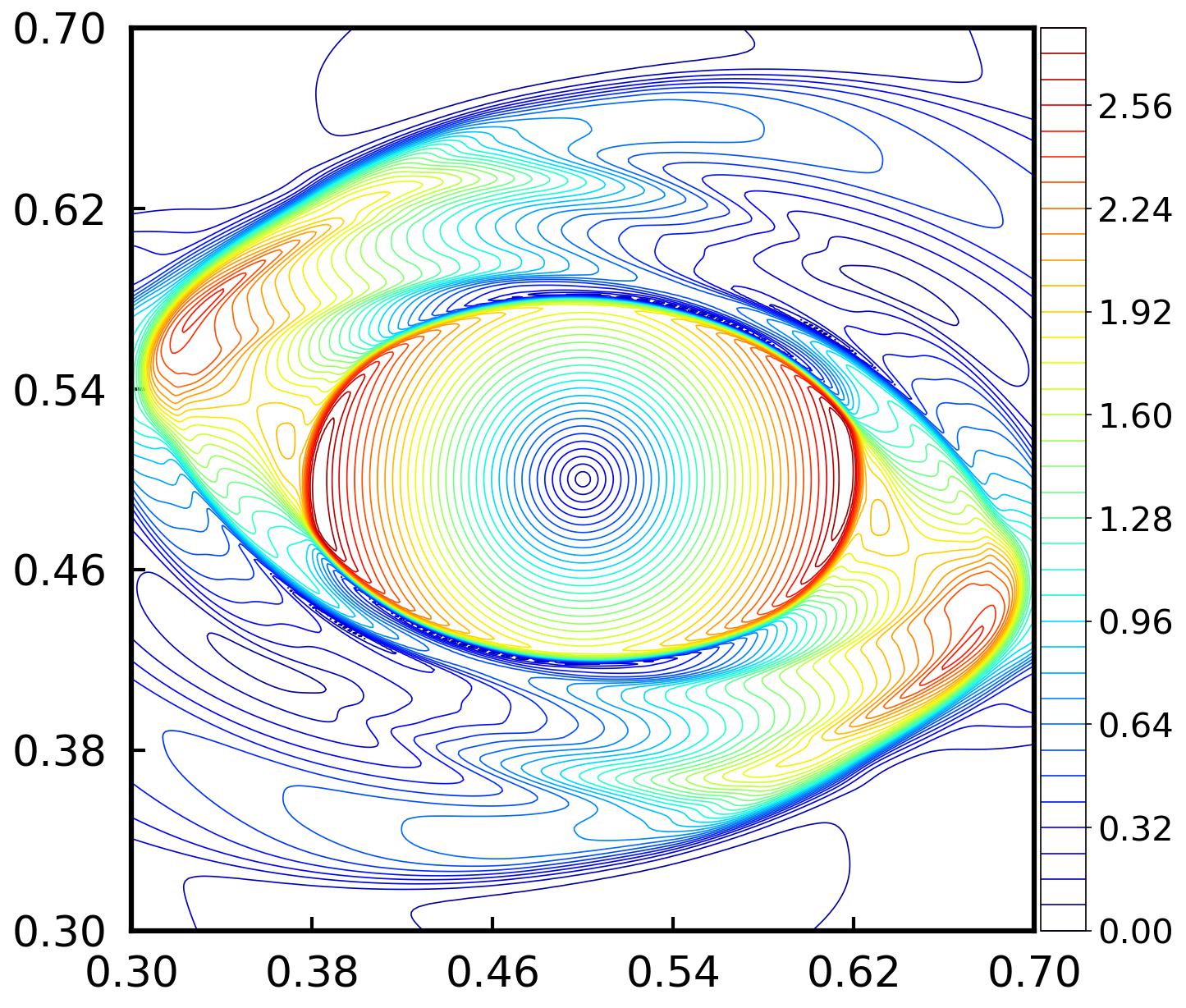}
			\end{subfigure}
			\caption{\Cref{Ex:Rotor}: Comparison of reconstructions using conservative variables (left) vs.~primitive variables (right) for the Mach number computed by the PPCT scheme with $800 \times 800$ cells. 
			}
			\label{fig:Ex-Rotor33}
		\end{figure}
		
		\begin{figure}[!htb]
			\centering
			\begin{subfigure}{0.48\textwidth}
				\includegraphics[width=\textwidth]{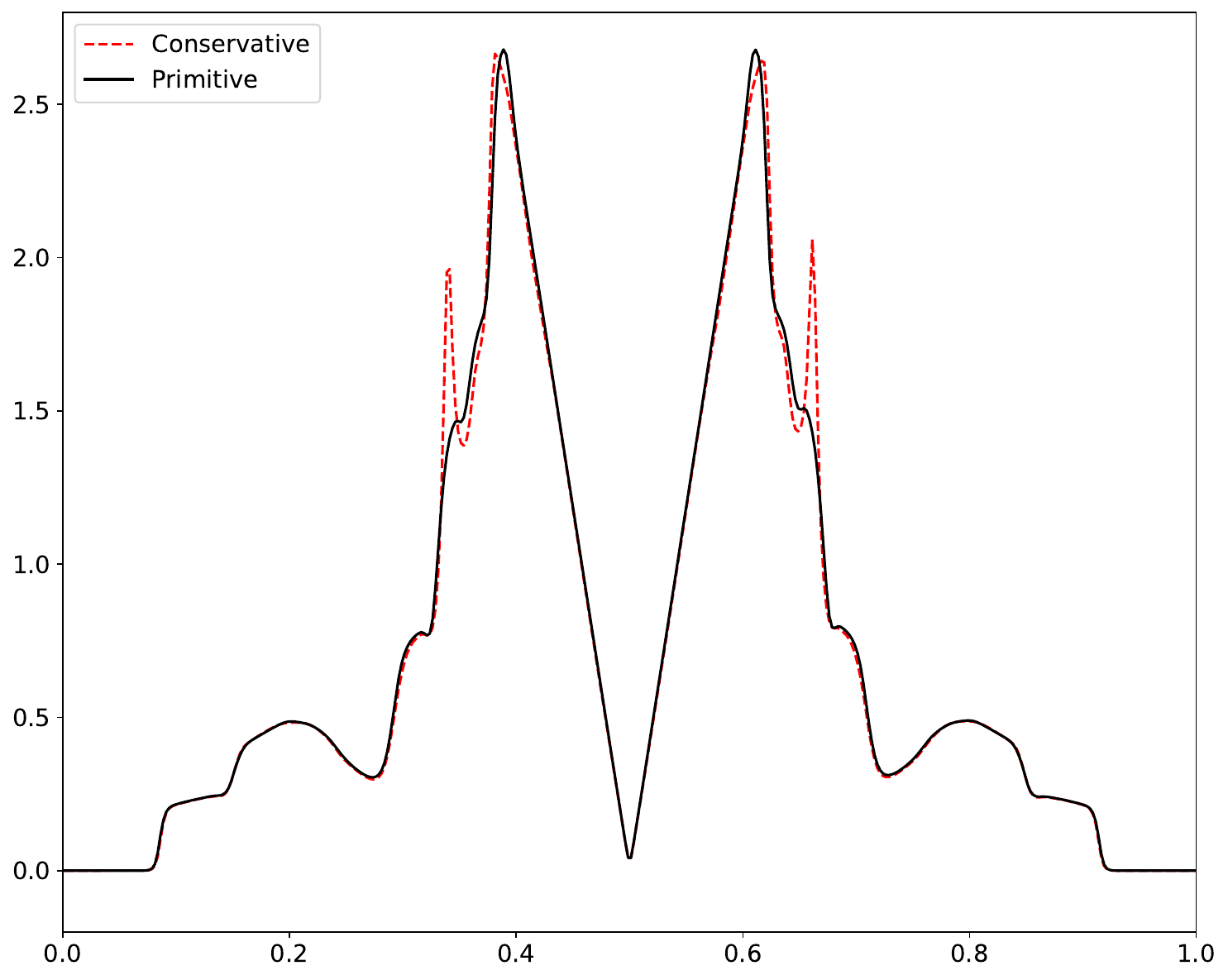}
			\end{subfigure}
			\hfill
			\begin{subfigure}{0.48\textwidth}
				\includegraphics[width=\textwidth]{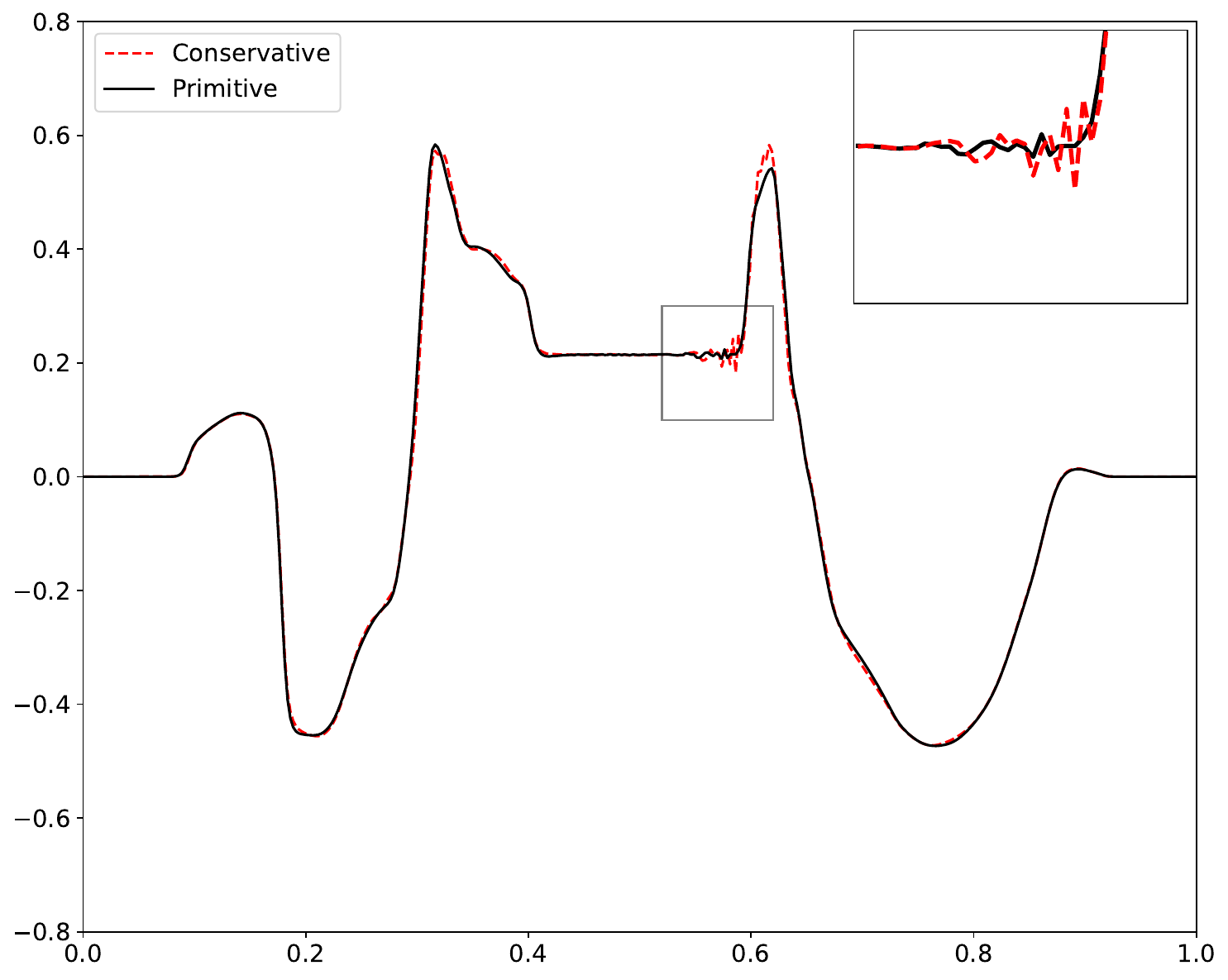}
			\end{subfigure}
			\caption{\Cref{Ex:Rotor}: Mach number slice along the line $y = 0.5$ (left) and $B_2$ slice along the line $y=0.55$ (right) at $ t= 0.295$, computed by the PPCT scheme with $400 \times 400$ cells. Comparison of reconstructions using conservative variables (dotted lines) vs.~primitive variables (solid lines).}
			\label{fig:Ex-Rotor-cut_off}
		\end{figure}

	\end{expl}

	\begin{expl}[Blast Problem]\label{Ex:Blast}\rm
		A challenging MHD blast problem was first introduced by Balsara and Spicer in \cite{BalsaraSpicer1999}. As highlighted in references \cite{BalsaraSpicer1999,Li2011,Li2012,WuShu2018,WuShu2019}, solving this problem is particularly difficult. Many numerical methods for MHD are prone to producing negative pressures near shock fronts, especially in regions with sharp pressure increases, where numerical oscillations may push pressures below zero. This example is often used to test the robustness of MHD schemes. 
		The computational domain is set as $[-0.5, 0.5]^2$, with a mesh size of $400 \times 400$. Outflow boundary conditions are applied on all four boundaries.
		
		For the classical MHD blast problem \cite{BalsaraSpicer1999}, the initial conditions are
		\begin{equation*}
			(\rho, {\bm v}, {\bm B}, p) = \begin{cases} 
				\left( 1, 0, 0, 0, B_0, 0, 0, p_0 \right), & \text{if } \sqrt{x^2 + y^2} \leq 0.1, \\ 
				\left( 1, 0, 0, 0, B_0, 0, 0, 0.1 \right), & \text{otherwise.} 
			\end{cases}
		\end{equation*}
		Here, $\gamma = 1.4$, $p_0 = 10^3$, and $B_0 = \frac{100}{\sqrt{4\pi}}$ (plasma-beta $\beta \approx 2.51 \times 10^{-4}$). 
		\Cref{fig:Ex-Blast} shows the contour plots of the solutions at $t = 0.01$. The results agree well with those in \cite{BalsaraSpicer1999, Christlieb2015PP, Li2011, WuShu2018, WuShu2019}. The visualizations in \Cref{fig:Ex-Blast} illustrate the outward movement of a circular blast wave and the inward trajectory of a rarefaction wave. The shocks are well captured, and no negative numerical pressure is observed, indicating the robustness and reliability of our PPCT method in dealing with highly magnetized shock configurations.

		\begin{figure}[!htb]
			\centering		
			\begin{subfigure}{0.48\textwidth}
				\includegraphics[width=\textwidth]{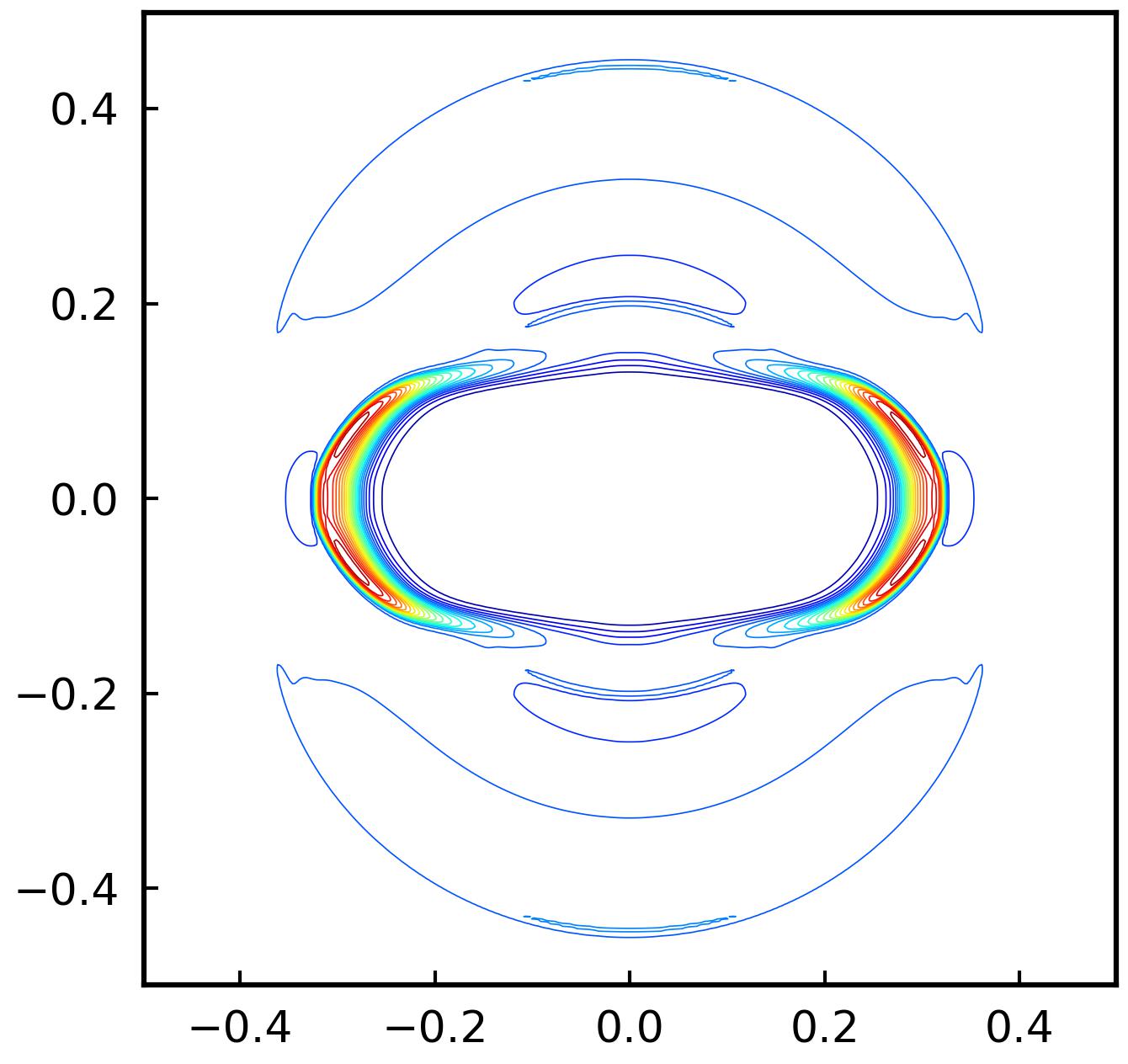}
			\end{subfigure}
			\hfill
			\begin{subfigure}{0.48\textwidth}
				\includegraphics[width=\textwidth]{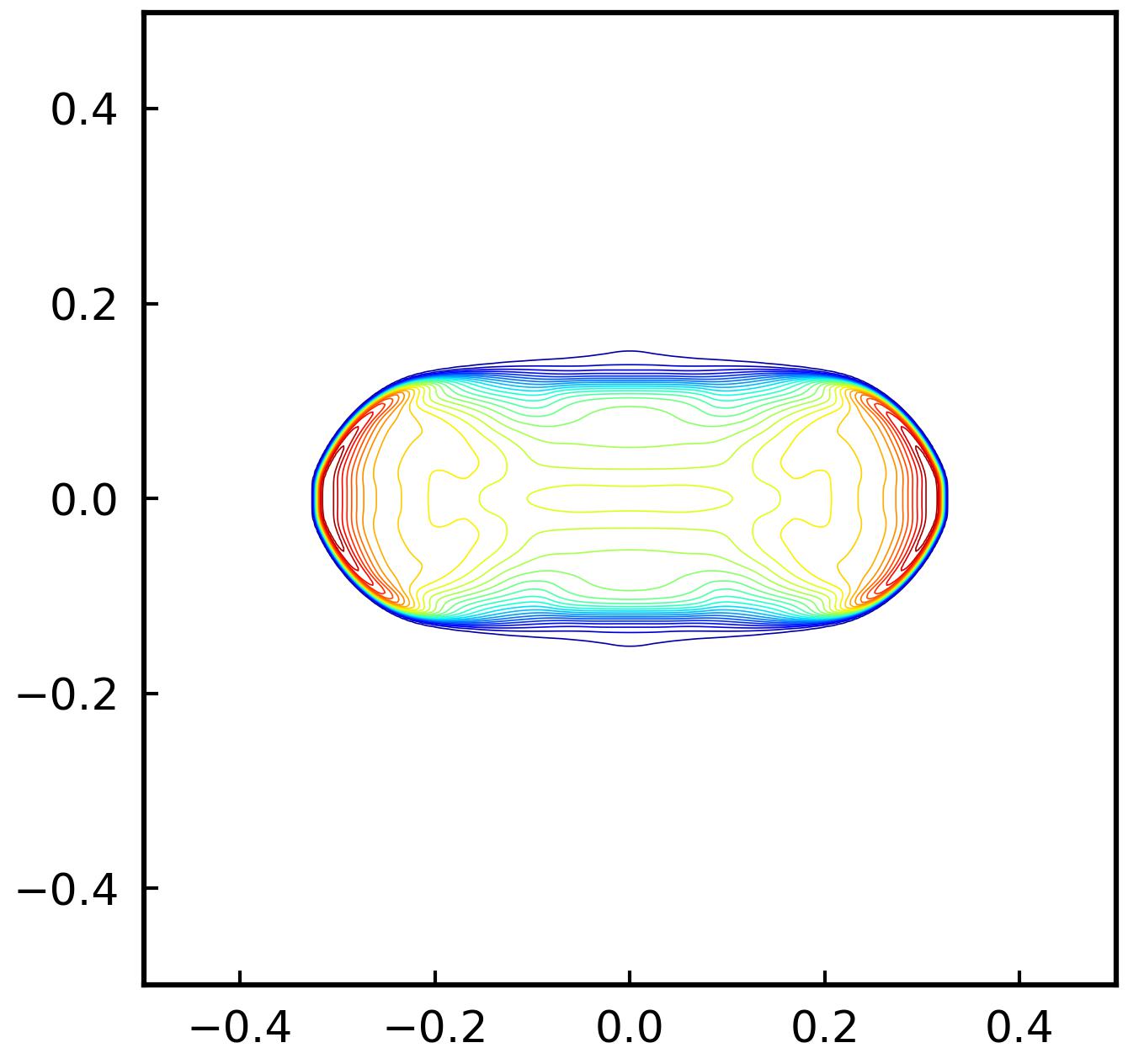}
			\end{subfigure}
			
			\begin{subfigure}{0.48\textwidth}
				\includegraphics[width=\textwidth]{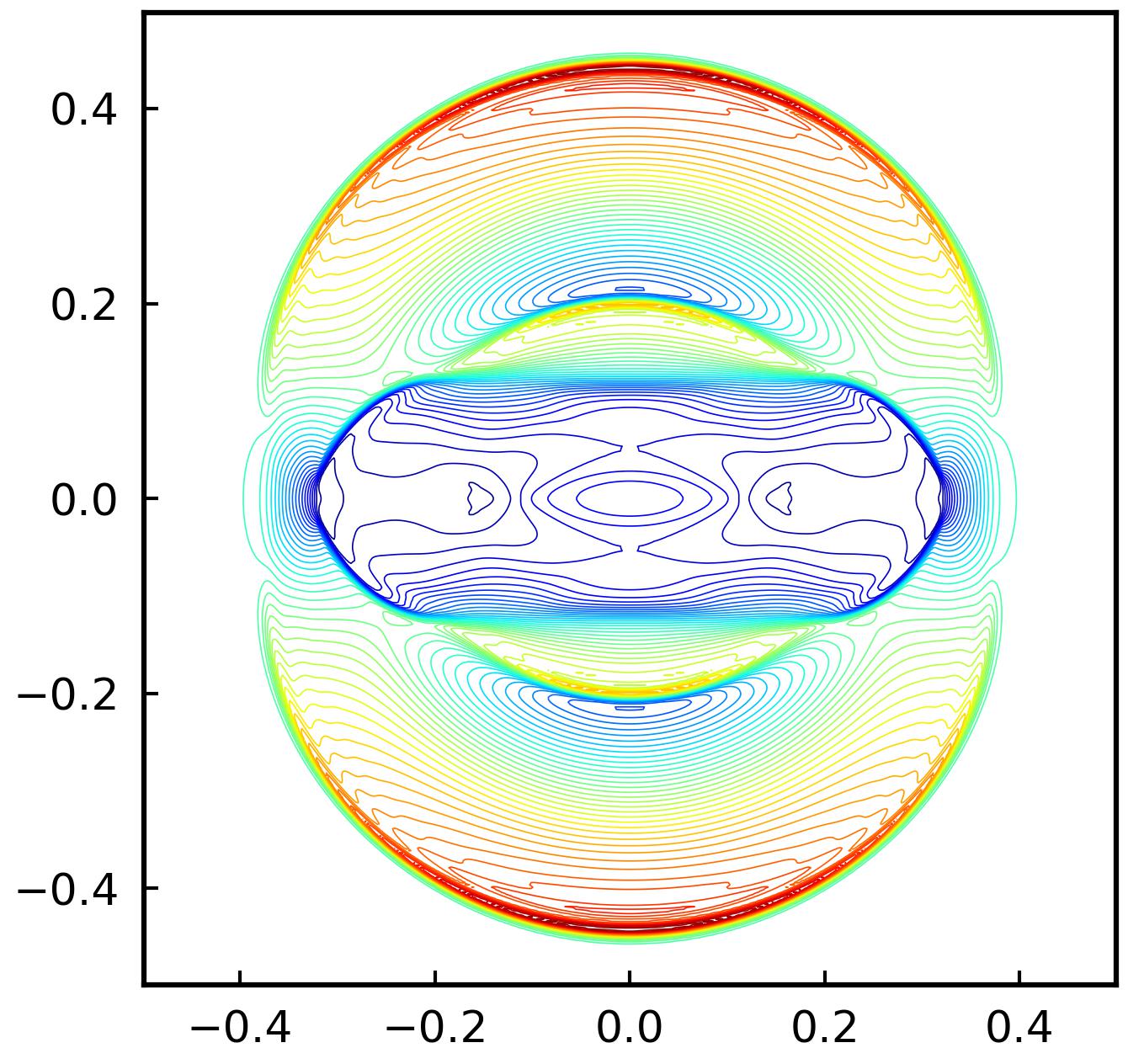}
			\end{subfigure}
			\hfill
			\begin{subfigure}{0.48\textwidth}
				\includegraphics[width=\textwidth]{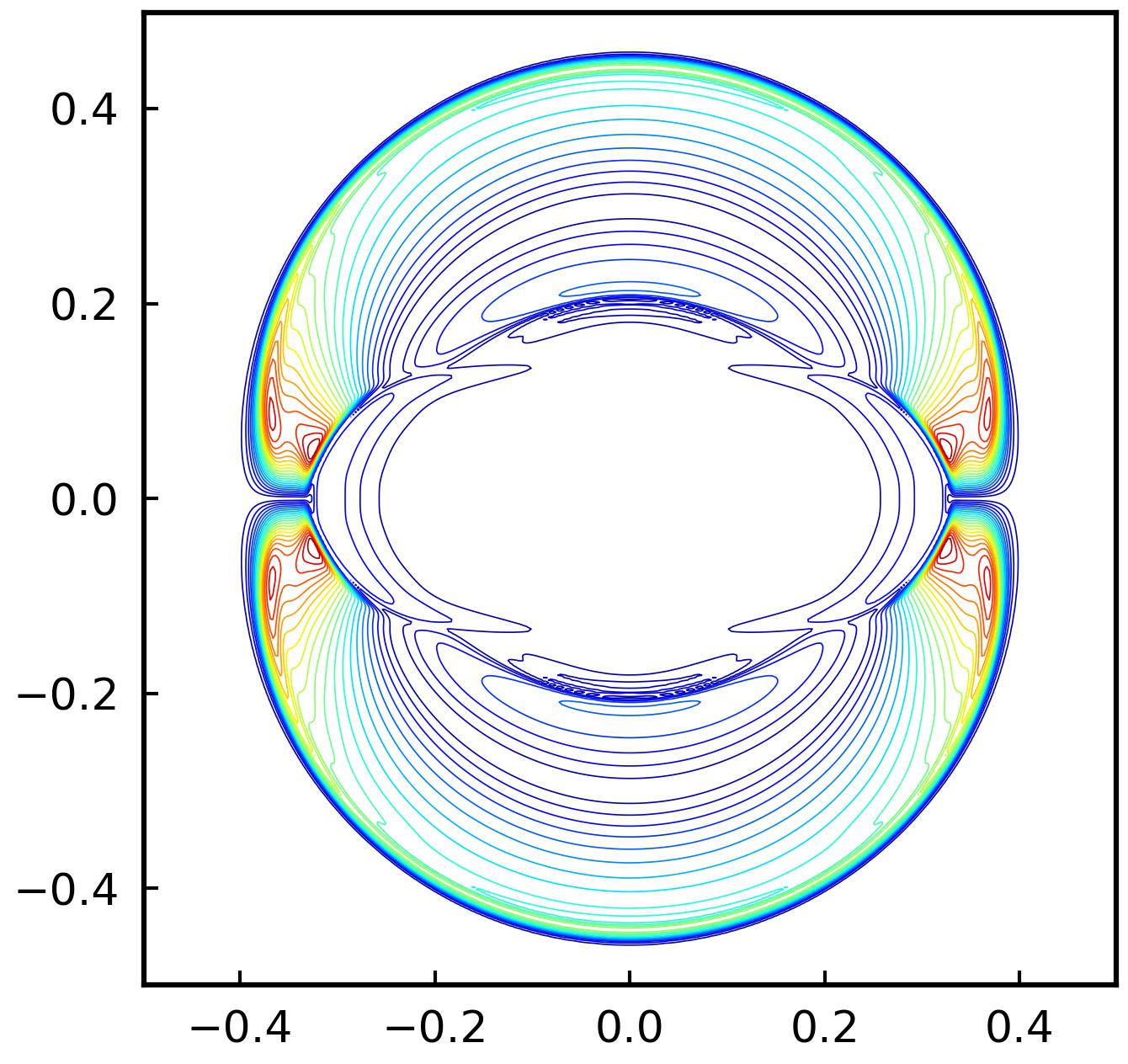}
			\end{subfigure}
			\caption{\Cref{Ex:Blast}: Contour plots of density (top-left), thermal pressure (top-right), magnetic pressure (bottom-left), and Mach number (bottom-right) for the blast problem at $t = 0.01$.}
			\label{fig:Ex-Blast}
		\end{figure}	
	\end{expl}

	\begin{expl}[Shock-Cloud Interaction Problem]\label{Ex:ShockCloud}\rm
		
		The shock-cloud interaction problem describes how a strong shock wave disrupts a high-density cloud, leading to discontinuities and small-scale flow instabilities. We follow the same setup as in \cite{Dai1998,jiang1998nonoscillatory,WuShu2018,WuShu2019}. The initial conditions are given by
		\begin{equation*}
			(\rho, {\bm v}, {\bm B}, p) = \begin{cases}
				(3.86859, 0, 0, 0, 0, 2.1826182, -2.1826182, 167.345), & x < 0.6, \\
				(1, -11.2536, 0, 0, 0, 0.56418958, 0.56418958, 1), & x > 0.6,
			\end{cases}
		\end{equation*}
		separated by a discontinuity at $x = 0.6$. On the right side of the discontinuity, there is a circular region defined by $(x - 0.8)^2 + (y - 0.5)^2 < 0.15^2$, where the cloud has a higher density of $10$. The computational domain is $[0, 1]^2$, partitioned into a uniform $400 \times 400$ rectangular mesh. Except for the right boundary, which is specified as an inflow condition, all other boundaries are treated as outflow boundary conditions. 
		The numerical results are shown in \Cref{fig:Ex-Shockcloud}, which display the plots for density, thermal pressure,  magnetic pressure, and velocity magnitude at $t = 0.06$. The results successfully capture the complex flow structures and interactions with high resolution, aligning well with those reported in the literature, e.g., \cite{Toth2000, jiang1998nonoscillatory, Balbas2006,WuShu2019}. Thanks to the PP property of our scheme, no negative density or pressure is observed throughout the entire simulation.

		\begin{figure}[!htb]
			\centering		
			\begin{subfigure}{0.48\textwidth}
				\includegraphics[width=\textwidth]{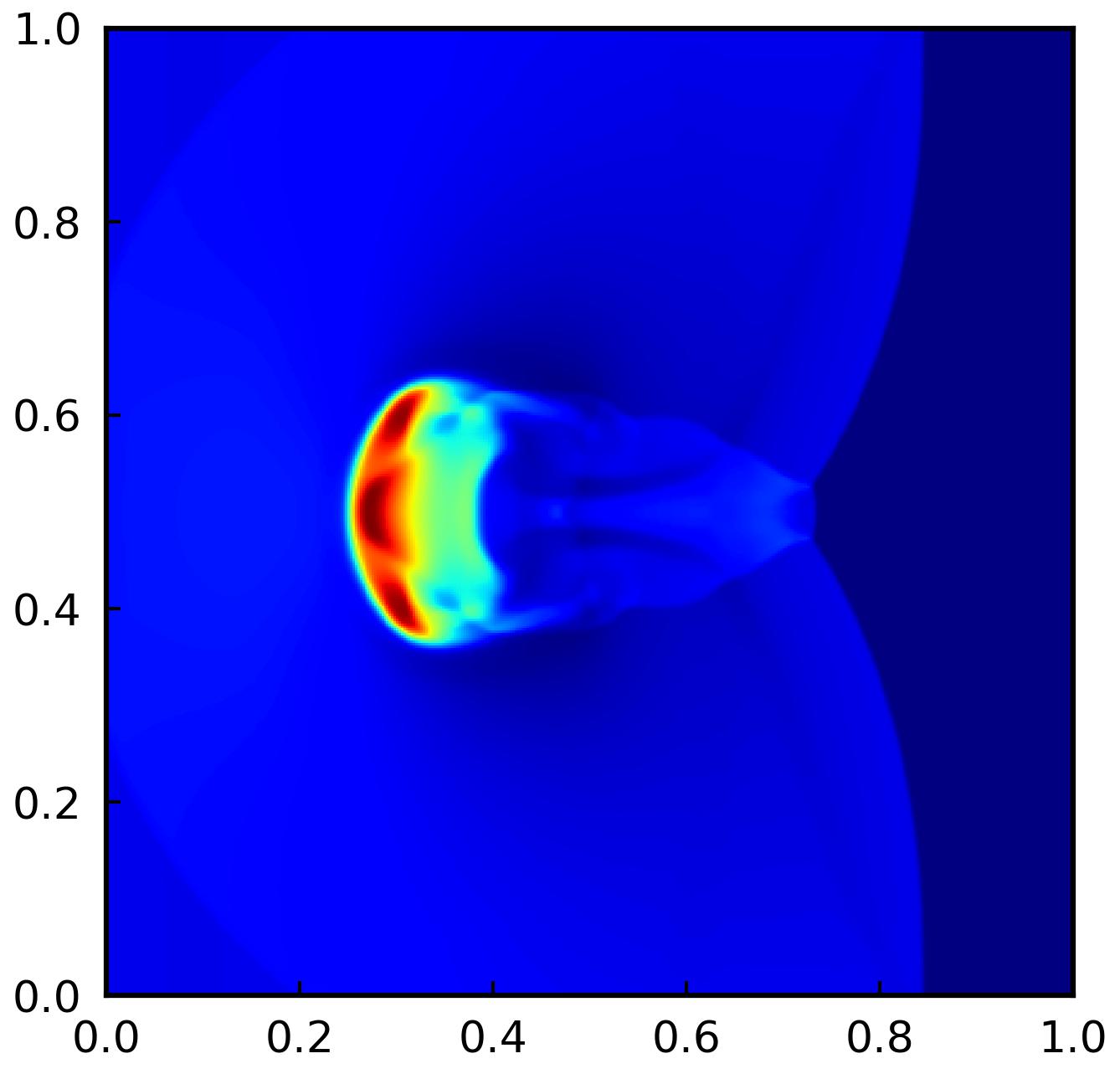}
			\end{subfigure}
			\hfill
			\begin{subfigure}{0.48\textwidth}
				\includegraphics[width=\textwidth]{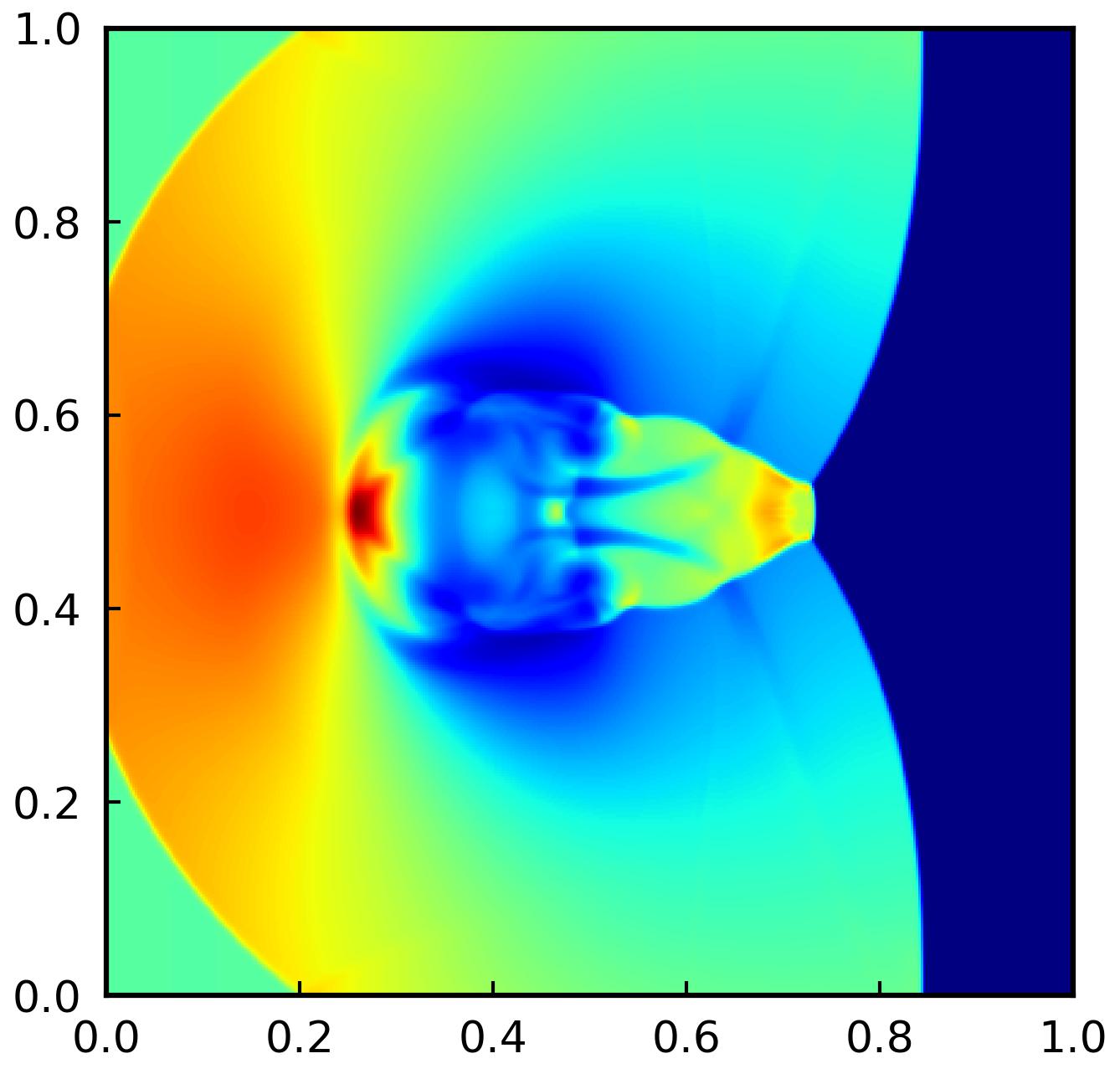}
			\end{subfigure}
			
			\begin{subfigure}{0.48\textwidth}
				\includegraphics[width=\textwidth]{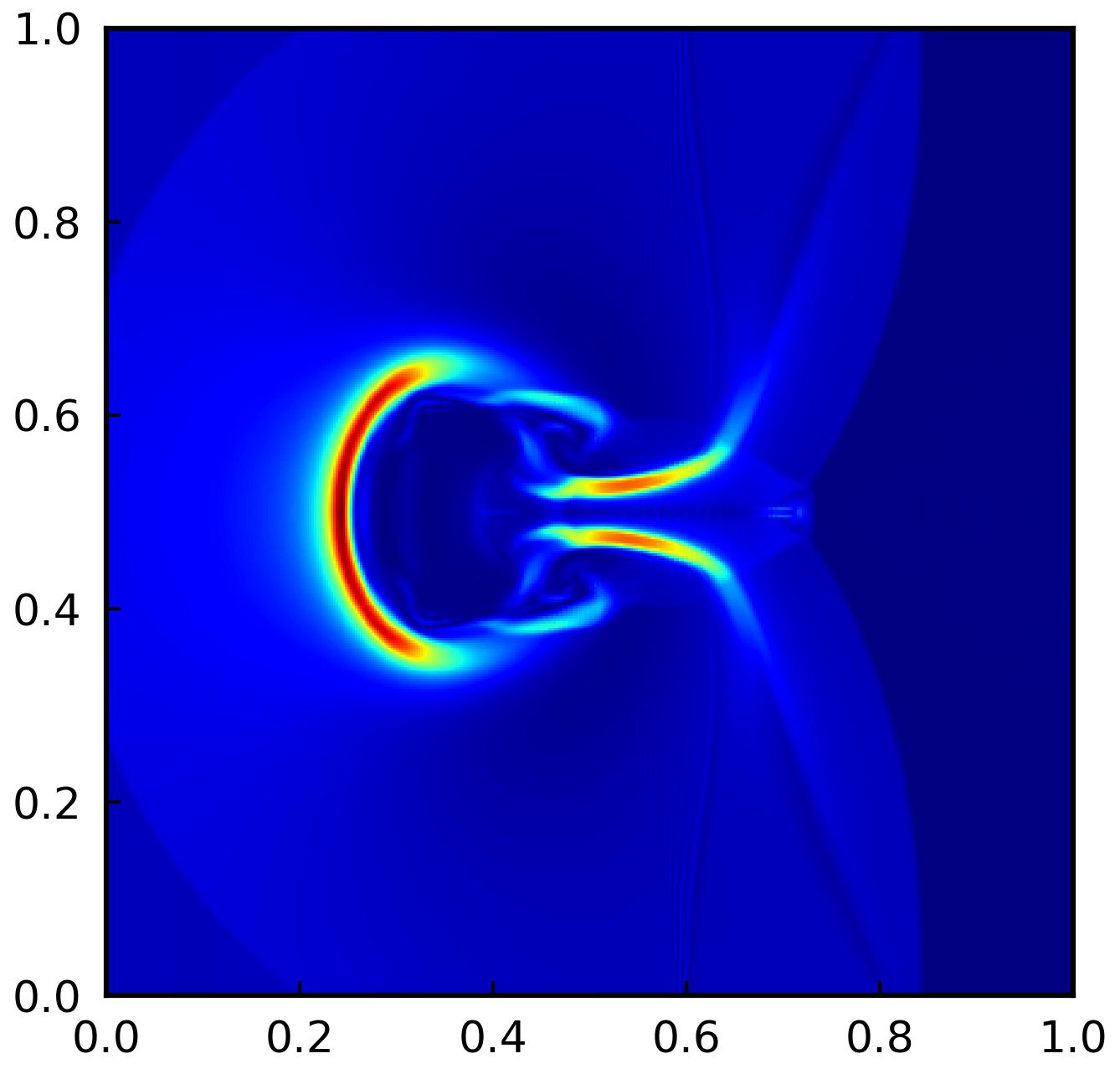}
			\end{subfigure}
			\hfill
			\begin{subfigure}{0.48\textwidth}
				\includegraphics[width=\textwidth]{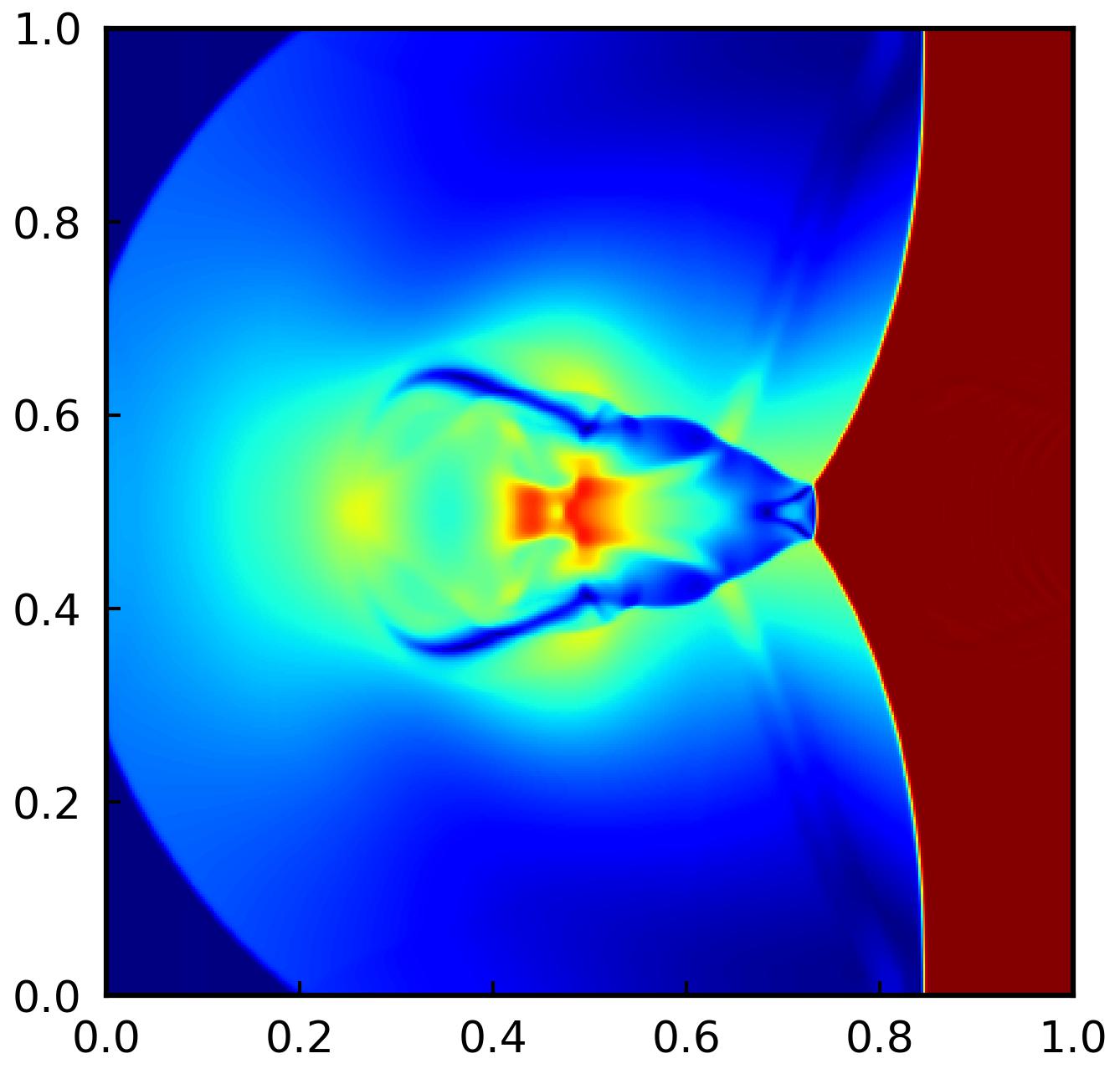}
			\end{subfigure}
			\caption{\Cref{Ex:ShockCloud}: Contour plots of density (top-left), thermal pressure (top-right), magnetic pressure (bottom-left), and velocity magnitude (bottom-right) for the shock-cloud problem at $t = 0.06$.}
			\label{fig:Ex-Shockcloud}
		\end{figure} 
	\end{expl}
	
	\begin{expl}[MHD Sedov Problem]\label{Ex:Sedov}\rm
		In this example, we extend the classical Sedov blast wave problem \cite{Zhang2010b} from the compressible Euler equations to the MHD context. It simulates an expanding blast wave from a central explosion, illustrating typical low-density conditions in shock dynamics. The initial setup models the blast wave originating at the point $(0, 0)$. The initial conditions are specified as
		\begin{equation*}
			(\rho, {\bm v}, {\bm B})= (1, 0, 0, 0, 1, 1, 0),
		\end{equation*}
		where $\gamma = 1.4$. The total mechanical energy is $2.5 \times 10^{-5}$ everywhere, except in the cell containing the origin, where it is set to $\frac{0.244816}{\Delta x \Delta y}$. The simulation is conducted in the domain $[-1, 1]^2$ with outflow boundary conditions on all sides, using a uniform mesh of $400 \times 400$ cells. The simulation runs up to $t = 0.4$. 
		\Cref{fig:Ex-Sedov} displays contour plots of density, thermal pressure, magnetic pressure, and velocity magnitude. The simulation successfully captures the blast wave and its associated shock dynamics. Moreover, we observed that turning off the PP limiter would cause the simulation to crash due to negative pressure, underscoring the critical role of the PP limiter in ensuring robust simulations.
		\begin{figure}[!htb]
			\centering		
			\begin{subfigure}{0.48\textwidth}
				\includegraphics[width=\textwidth]{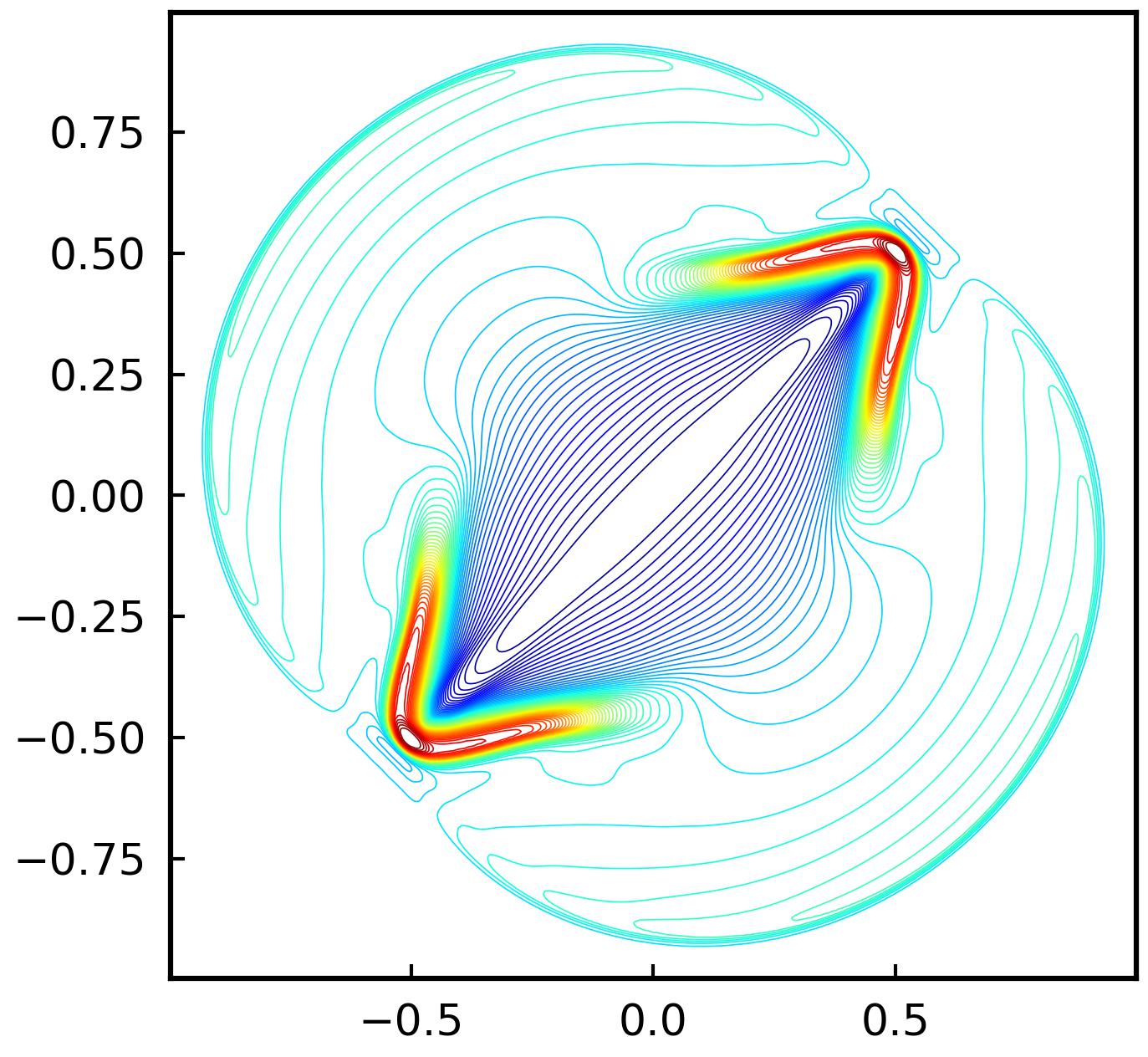}
			\end{subfigure}
			\hfill
			\begin{subfigure}{0.48\textwidth}
				\includegraphics[width=\textwidth]{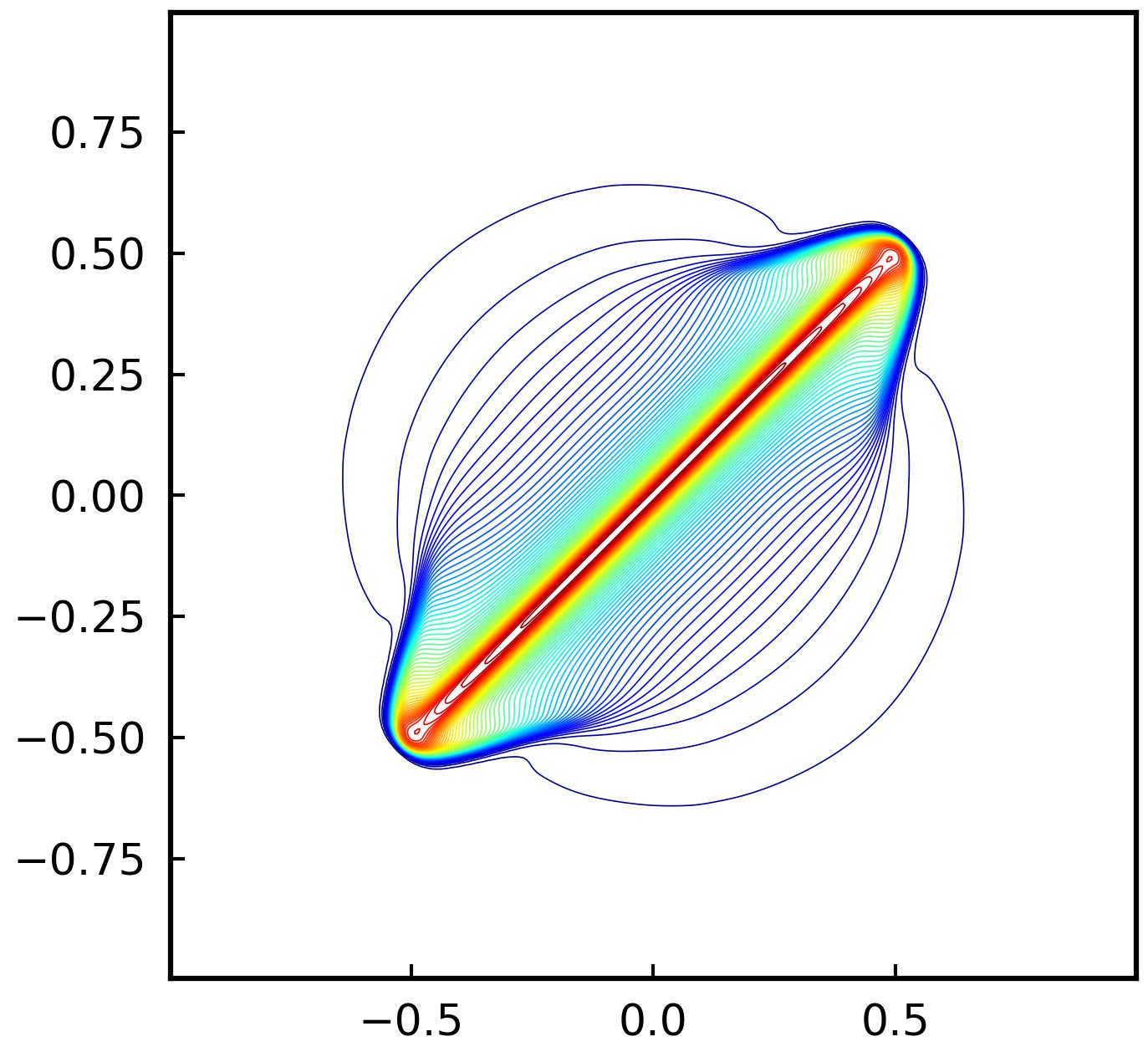}
			\end{subfigure}
			
			\begin{subfigure}{0.48\textwidth}
				\includegraphics[width=\textwidth]{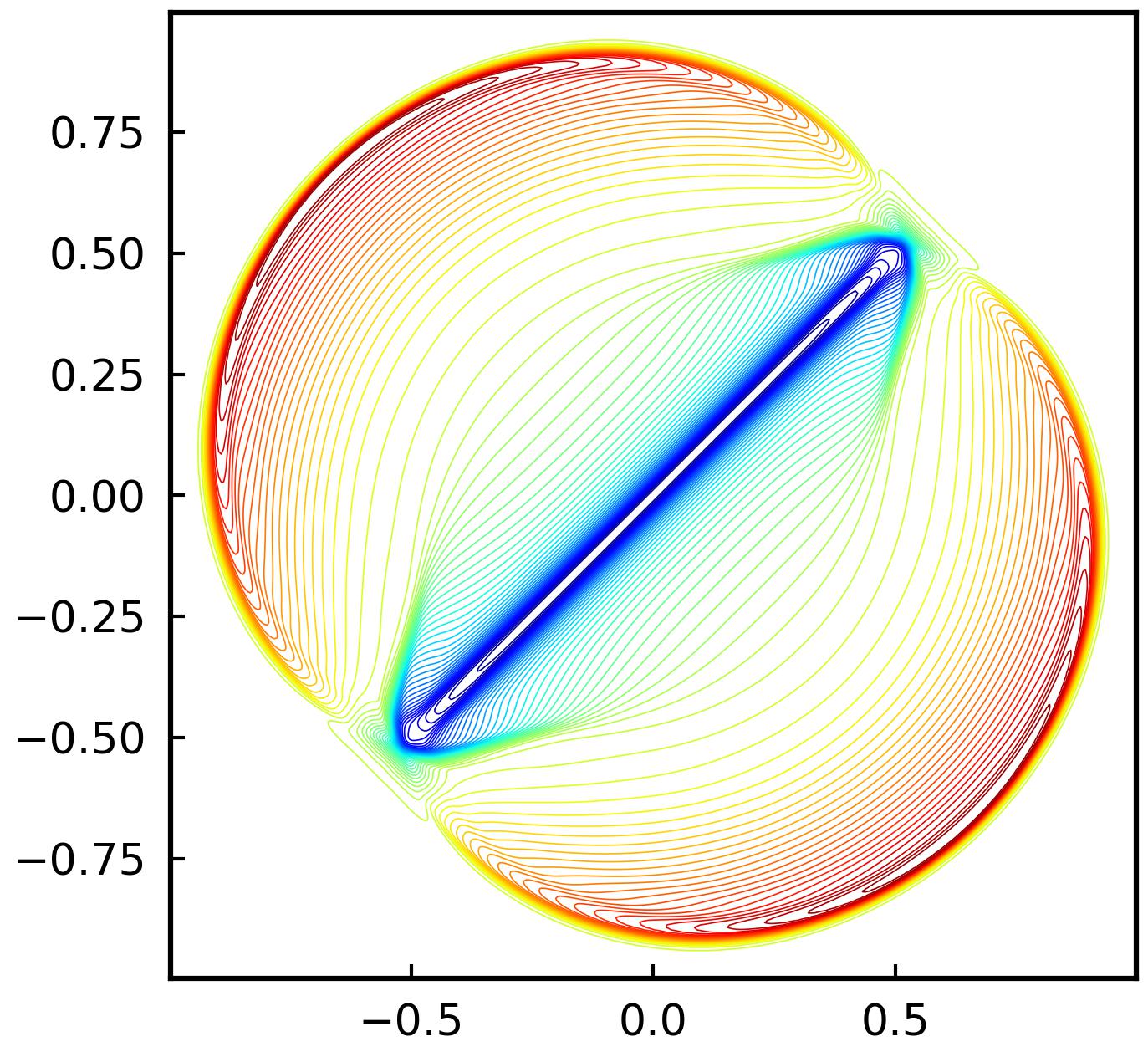}
			\end{subfigure}
			\hfill
			\begin{subfigure}{0.48\textwidth}
				\includegraphics[width=\textwidth]{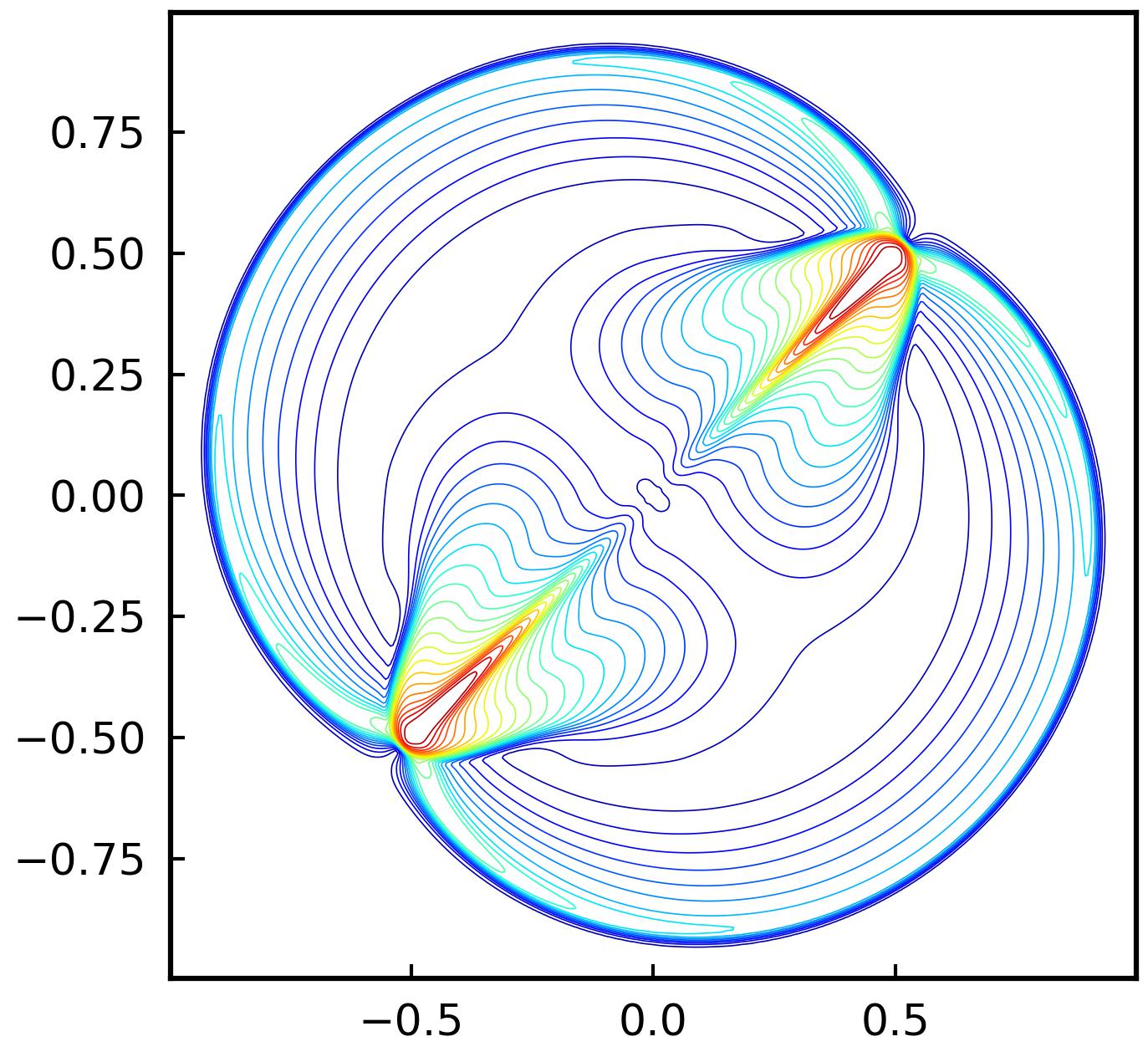}
			\end{subfigure}
			\caption{\Cref{Ex:Sedov}: Contour plots of density (top-left), thermal pressure (top-right), magnetic pressure (bottom-left), and velocity magnitude (bottom-right) for the MHD Sedov problem at $t = 0.4$.}
			\label{fig:Ex-Sedov}
		\end{figure} 
	\end{expl}

	\begin{figure}[!thb]
		\centering
		\begin{subfigure}{0.32\textwidth}
			\includegraphics[scale=0.32]{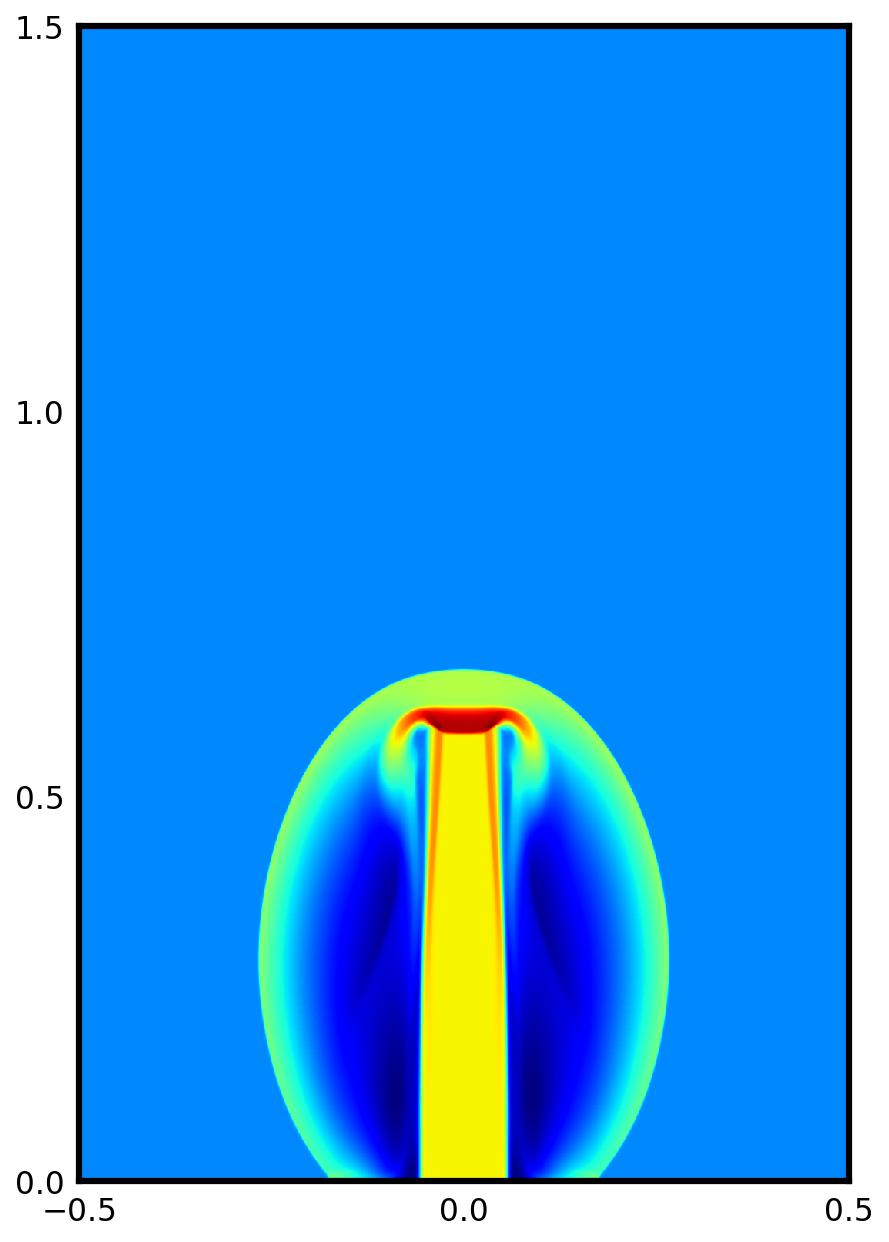}
		\end{subfigure}
		\hfill
		\begin{subfigure}{0.32\textwidth}
			\includegraphics[scale=0.32]{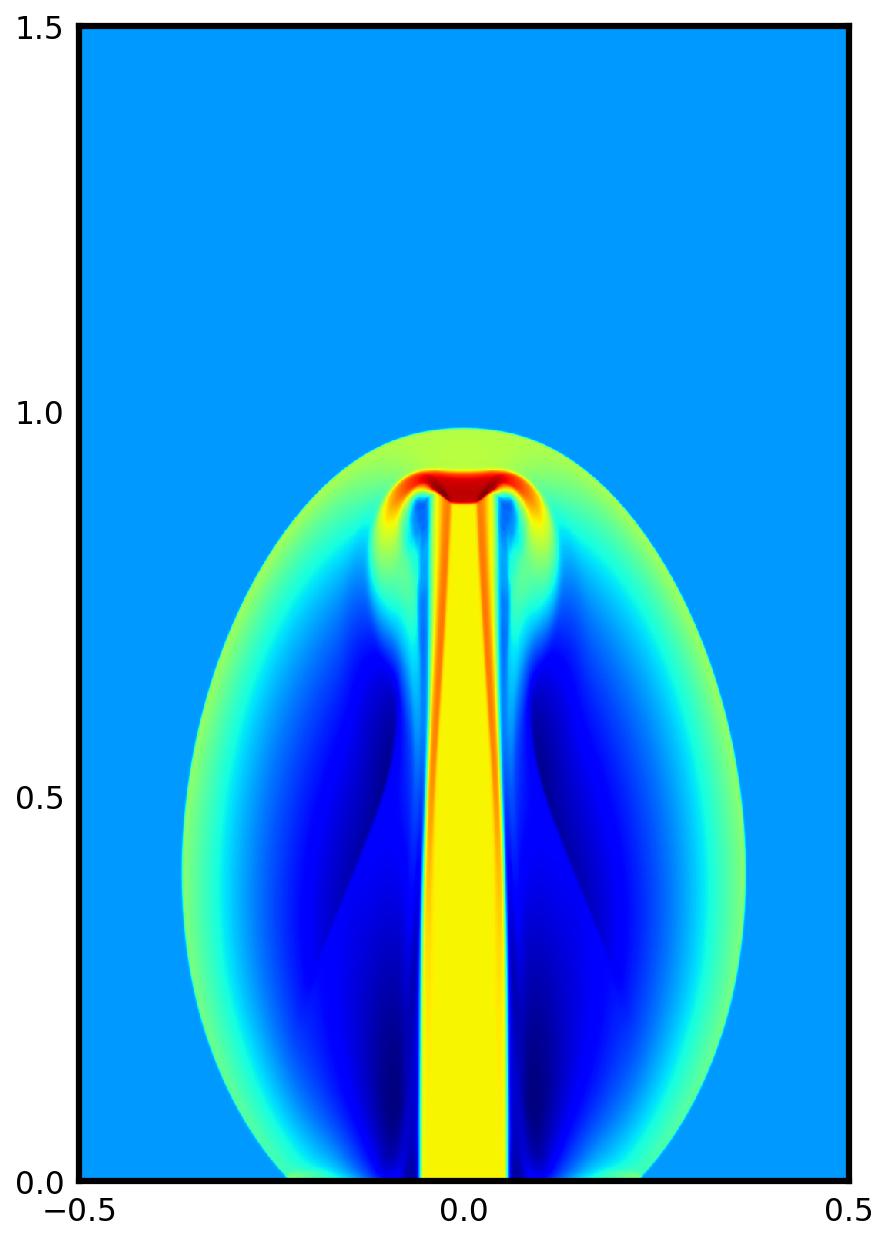}
		\end{subfigure}
		\hfill
		\begin{subfigure}{0.32\textwidth}
			\includegraphics[scale=0.32]{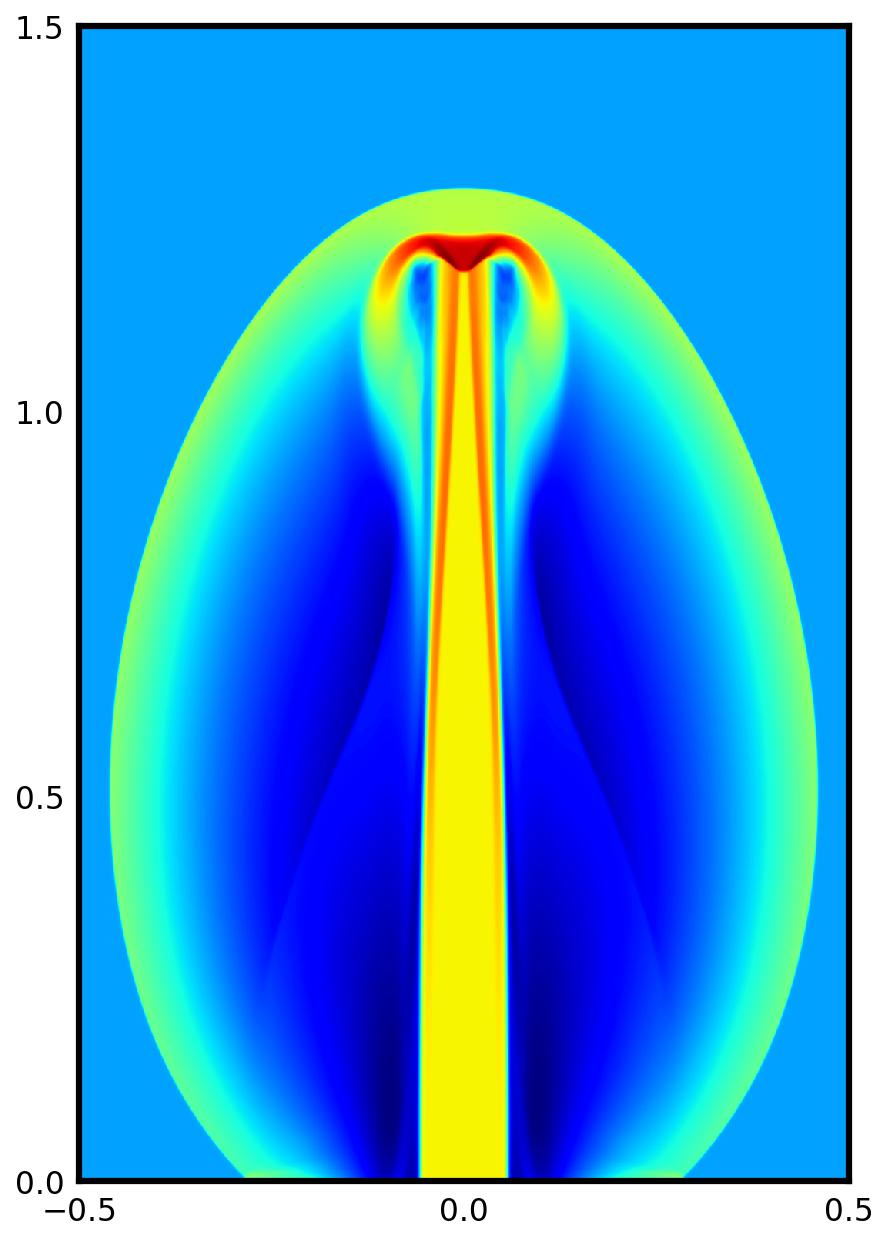}
		\end{subfigure}
		\caption{Mach 800 jet problem with $B_0 = \sqrt{200}$: Density logarithm at $t = 0.001, 0.0015$, and $0.002$(from left to right).
		}
		\label{fig:Ex-Jet_800_200}
	\end{figure} 
	
	\begin{figure}[!htb]
		\centering	
		\begin{subfigure}{0.32\textwidth}
			\includegraphics[scale=0.32]{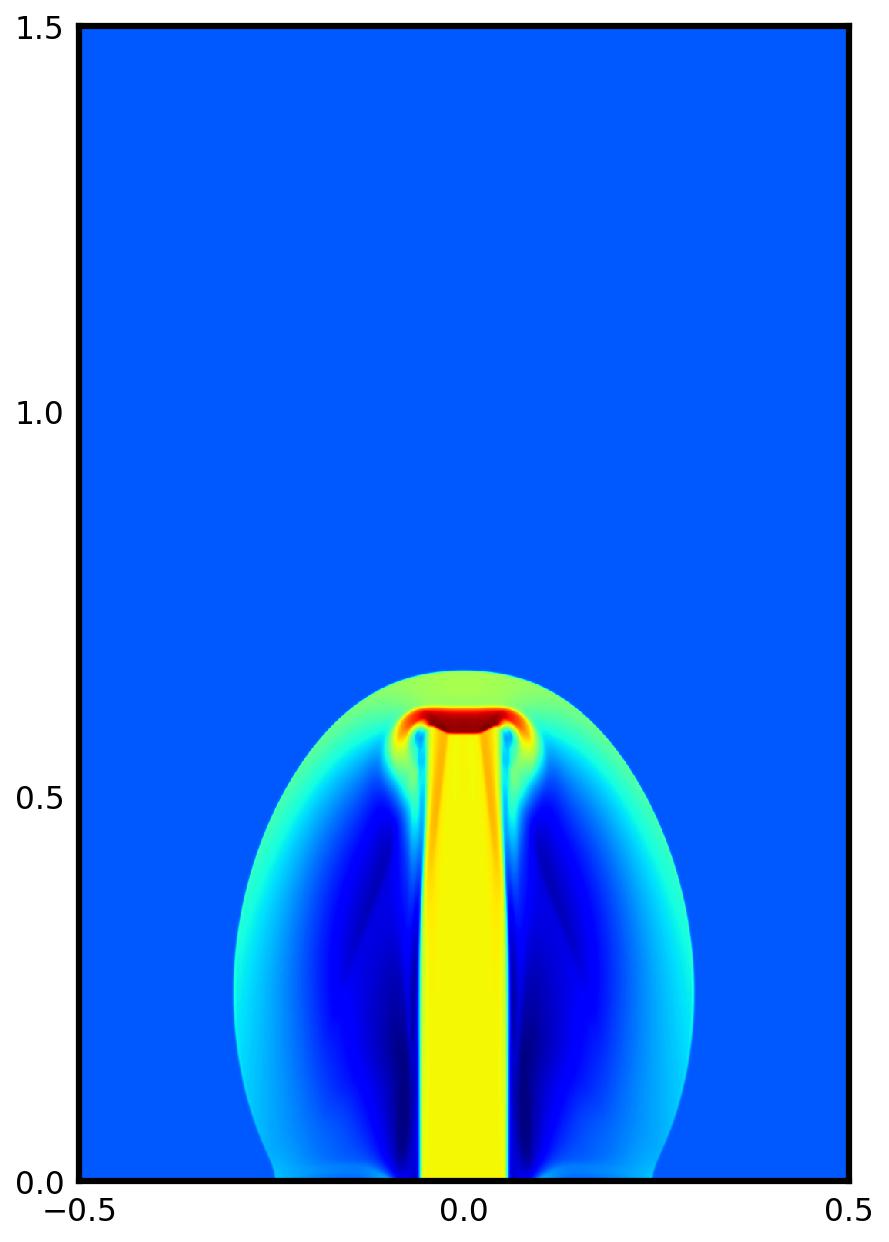}
		\end{subfigure}
		\hfill
		\begin{subfigure}{0.32\textwidth}
			\includegraphics[scale=0.32]{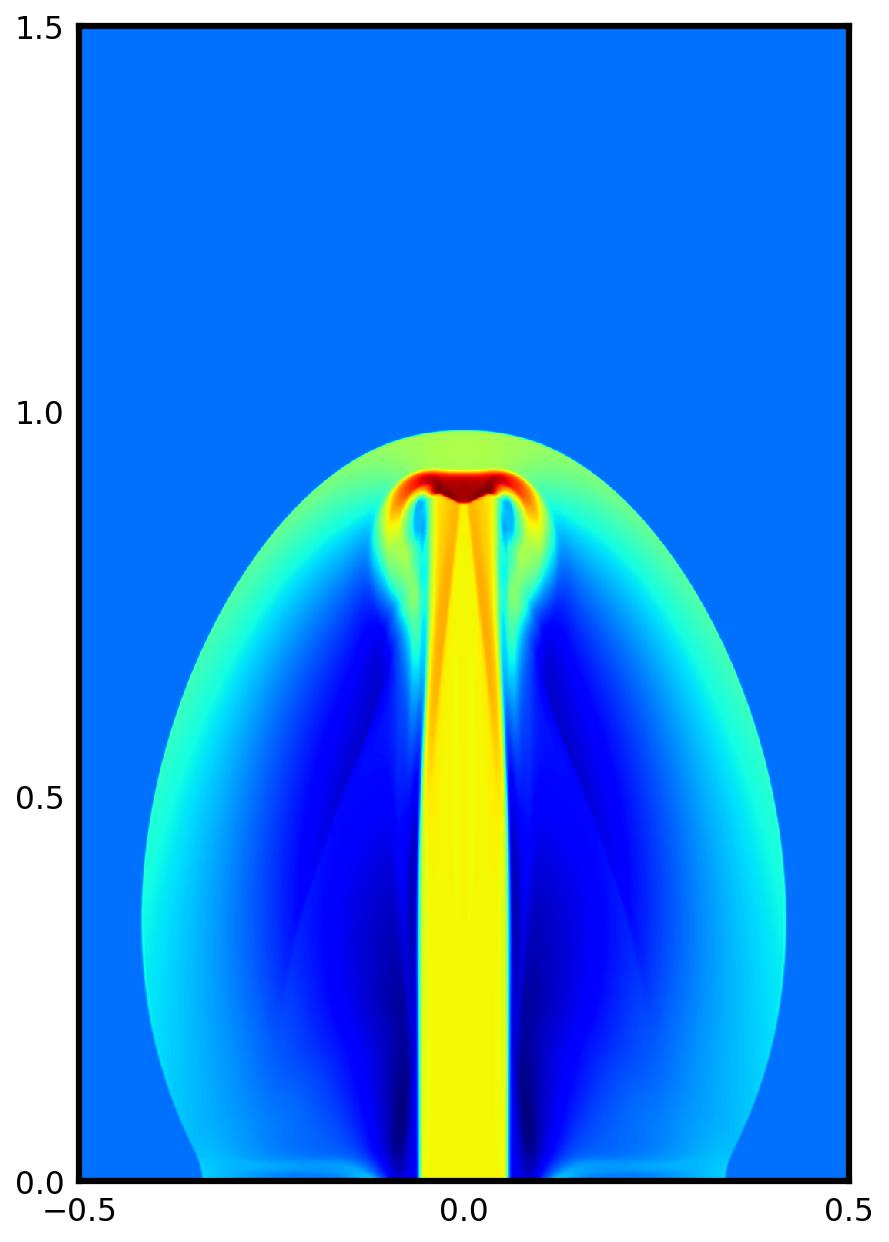}
		\end{subfigure}
		\hfill
		\begin{subfigure}{0.32\textwidth}
			\includegraphics[scale=0.32]{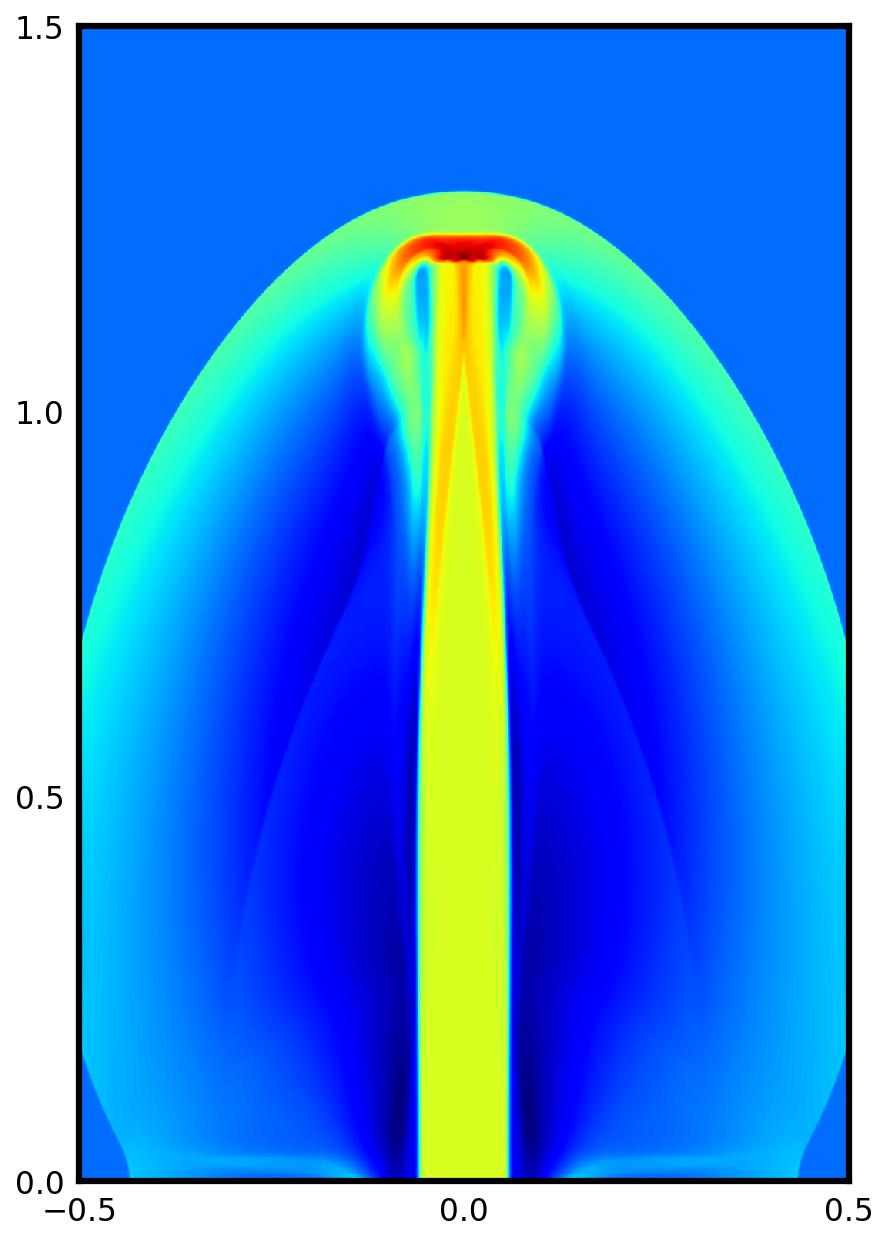}
		\end{subfigure}
		
		\begin{subfigure}{0.32\textwidth}
			\includegraphics[scale=0.32]{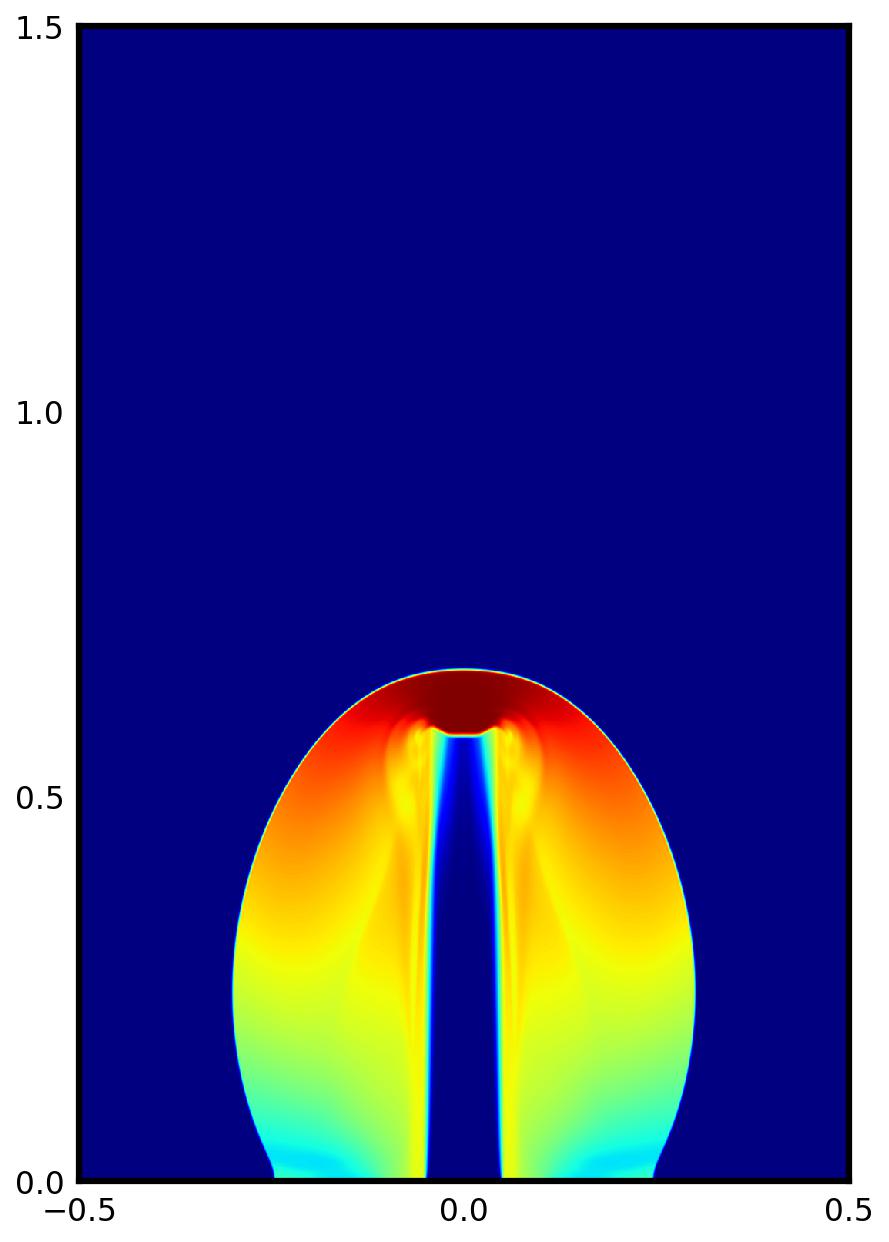}
		\end{subfigure}
		\hfill
		\begin{subfigure}{0.32\textwidth}
			\includegraphics[scale=0.32]{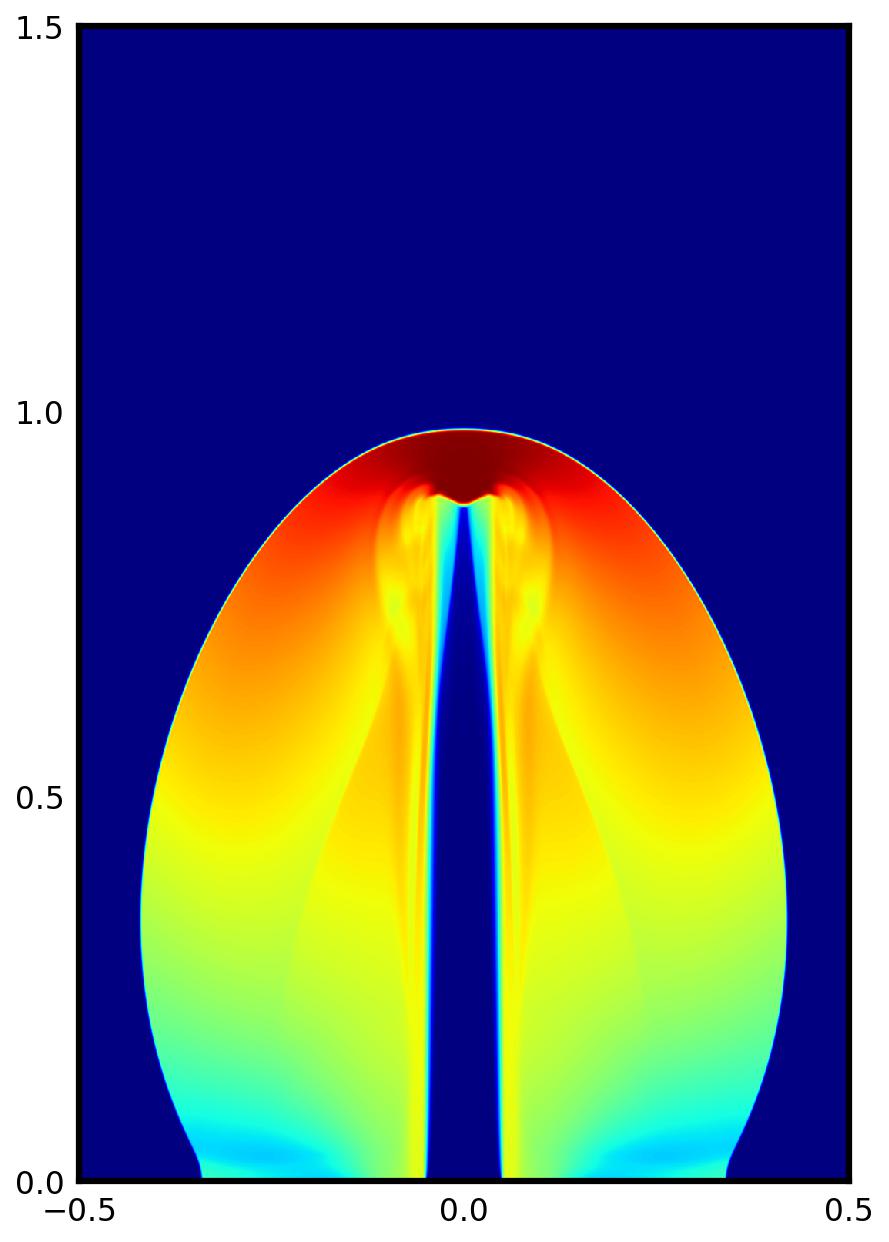}
		\end{subfigure}
		\hfill
		\begin{subfigure}{0.32\textwidth}
			\includegraphics[scale=0.32]{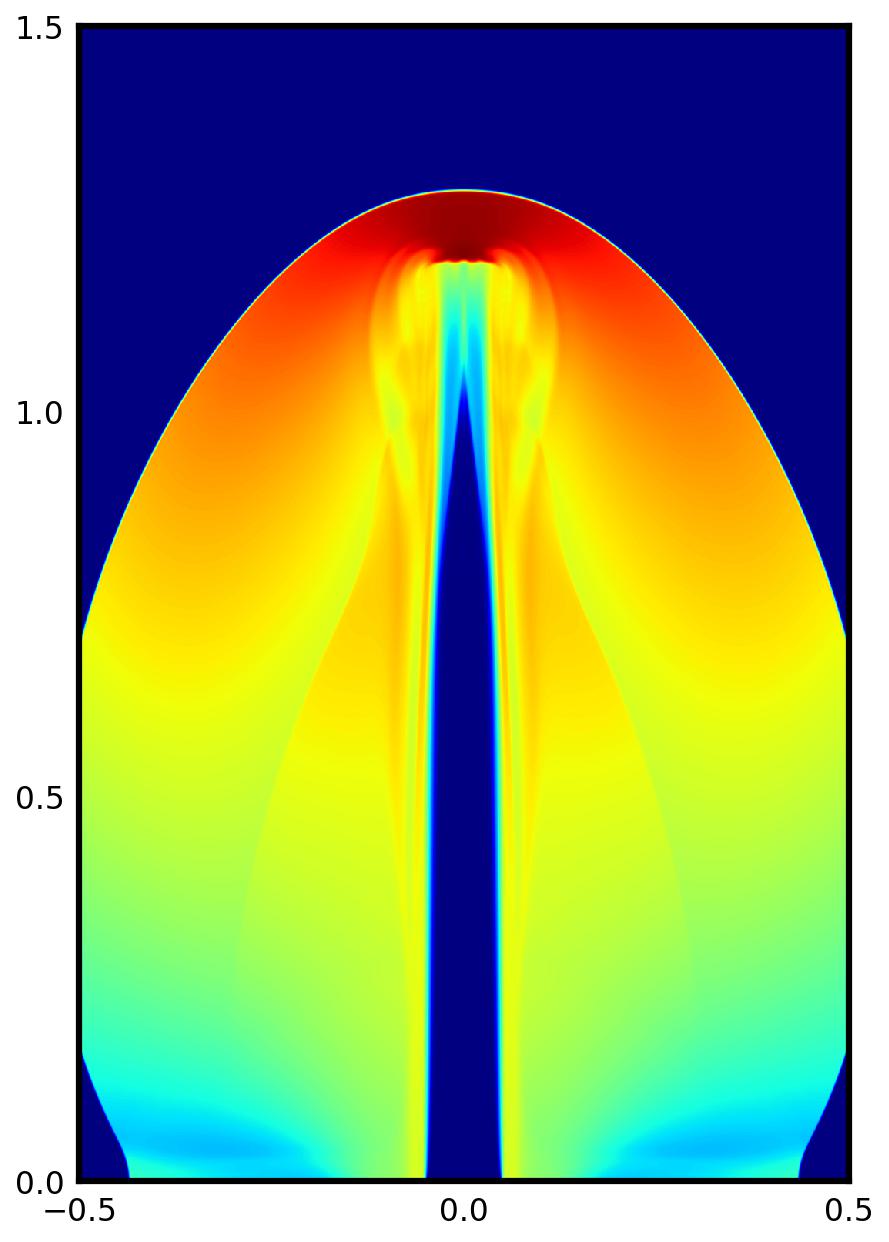}
		\end{subfigure}
		
		\caption{Mach 800 jet with $B_0 = \sqrt{2000}$: Density logarithm (top) and gas pressure logarithm (bottom) at $t = 0.001, 0.0015$, and $0.002$ (from left to right).
		}
		\label{fig:Ex-Jet_800_2000}
	\end{figure}

	\begin{expl}[Astrophysical Jets]\label{Ex:Jet}\rm
		Our final set of test cases involves several high Mach number MHD jet problems proposed in \cite{WuShu2018,WuShu2019}. These problems are particularly challenging due to the presence of high-speed jet flows, strong magnetic fields, and intense kinetic energy. The interactions between strong shock waves, shear flows, and interface instabilities in high Mach number jets often result in numerical difficulties such as negative pressure.
		
		We first consider the Mach 800 problem, set in the domain $[-0.5, 0.5] \times [0, 1.5]$ with the initial conditions defined by
		$$
		(\rho, {\bm v}, {\bm B}, p) = \left(0.1 \gamma, 0, 0, 0, 0, B_0, 0, 1\right),
		$$
		where $\gamma = 1.4$. The computational domain is restricted to $[0, 0.5] \times [0, 1.5]$, divided into $500 \times 1500$ uniform cells. A reflecting boundary condition is applied to the left boundary, while a section of the bottom boundary $(|x| < 0.05)$ is set as an inflow with conditions $(\rho, {\bm v}, {\bm B}, p) = \left(\gamma, 0, 800, 0, 0, B_0, 0, 1\right)$. The remaining boundaries are outflow conditions.
		
		Following \cite{WuShu2018,WuShu2019}, we test the numerical schemes under different values of $B_0$:
		(i) $B_0 = \sqrt{200}$, corresponding to a plasma-beta $\beta_0 = 10^{-2}$;
		(ii) $B_0 = \sqrt{2000}$, with $\beta_0 = 10^{-3}$;
		(iii) $B_0 = \sqrt{20000}$, with $\beta_0 = 10^{-4}$. 
		As $B_0$ increases, the problem becomes increasingly difficult. The schlieren plots at $t = 0.001$, $0.0015$, and $0.002$ are shown in \Cref{fig:Ex-Jet_800_200}, \Cref{fig:Ex-Jet_800_2000}, and \Cref{fig:Ex-Jet_800_20000}, respectively. The results indicate that the Mach shock waves are accurately captured, with no notable nonphysical oscillations. These tests demonstrate the robustness and high resolution of the proposed PPCT schemes, and no negative values for density or pressure are observed.
		
		To further highlight the robustness of the PPCT method, we consider higher Mach numbers. \Cref{fig:Ex-Jet_2000_20000} illustrates the Mach 2000 jet case with $B_0 = \sqrt{20000}$, while \Cref{fig:Ex-Jet_10000_20000} shows the solution for the Mach 10,000 jet with the same $B_0$. In the Mach 10,000 jet case, the solution appears narrower and taller compared to the Mach 2000 jet. The time evolution of the jets is clearly depicted, revealing distinct flow structures for different magnetization levels. Our PPCT method captures the Mach shock waves at the jet head and other discontinuities with high accuracy. Throughout all simulations, no negative density or pressure is encountered, emphasizing the robustness of our schemes under extreme conditions. These results are consistent with those reported in \cite{WuShu2018,WuShu2019} based on provably PP discontinuous Galerkin methods.

		\begin{figure}[!htb]
			\centering	
			\begin{subfigure}{0.32\textwidth}
				\includegraphics[scale=0.32]{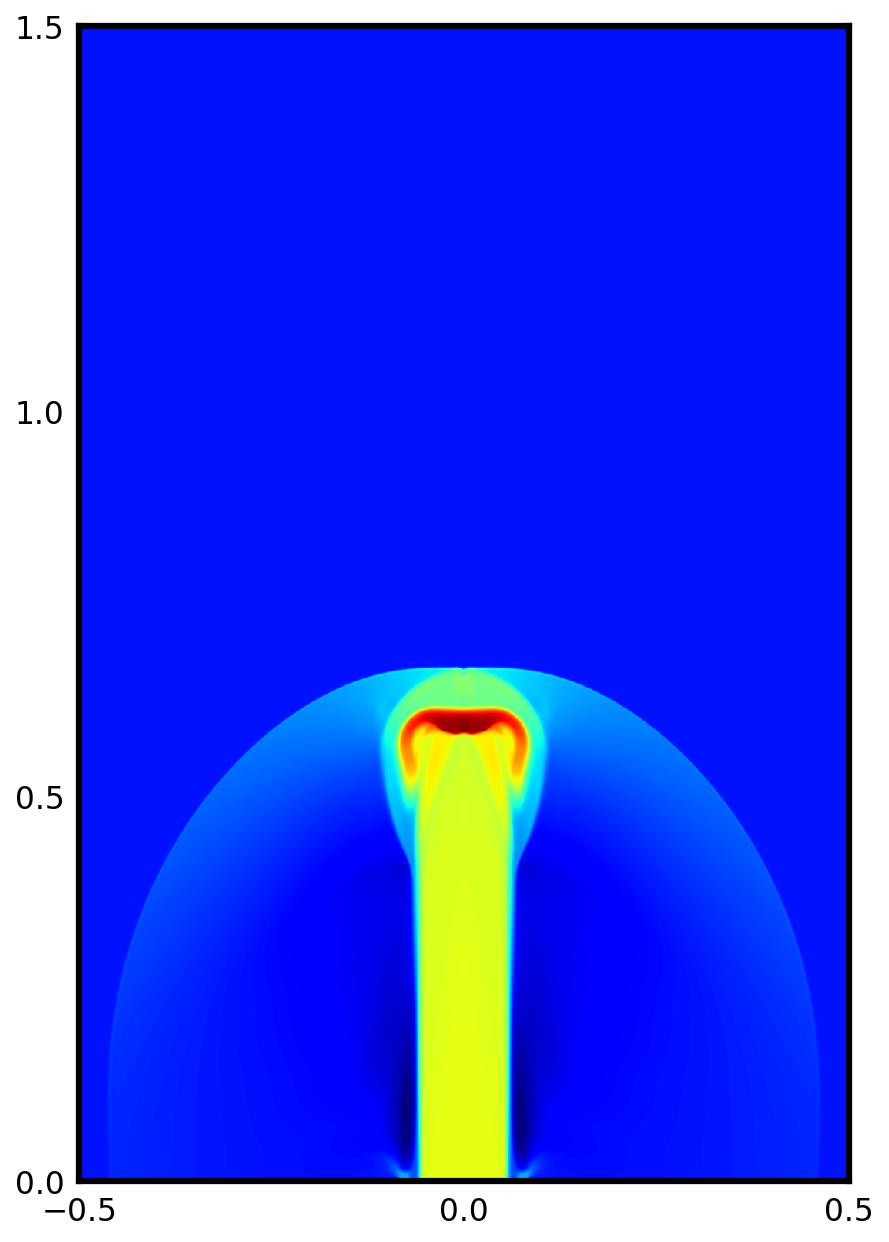}
			\end{subfigure}
			\hfill
			\begin{subfigure}{0.32\textwidth}
				\includegraphics[scale=0.32]{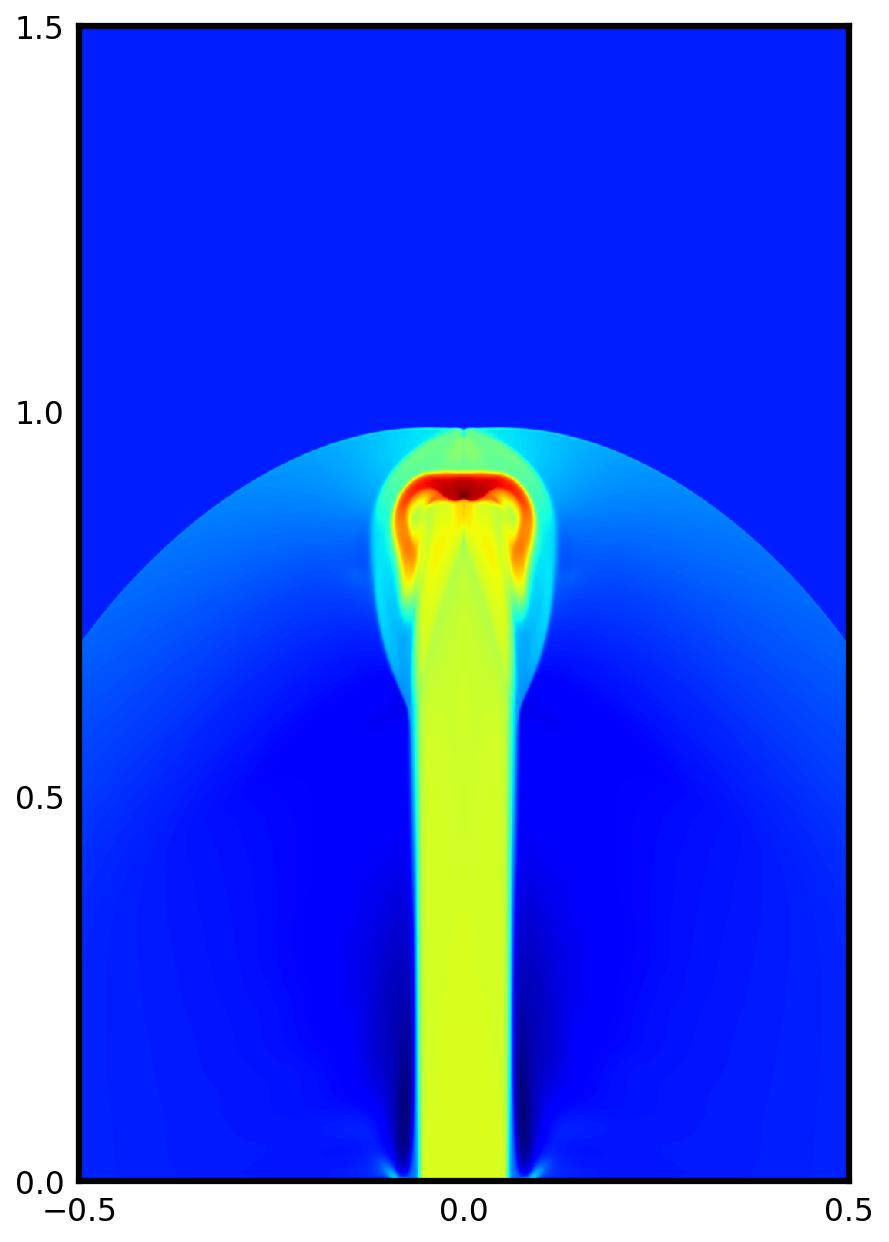}
			\end{subfigure}
			\hfill
			\begin{subfigure}{0.32\textwidth}
				\includegraphics[scale=0.32]{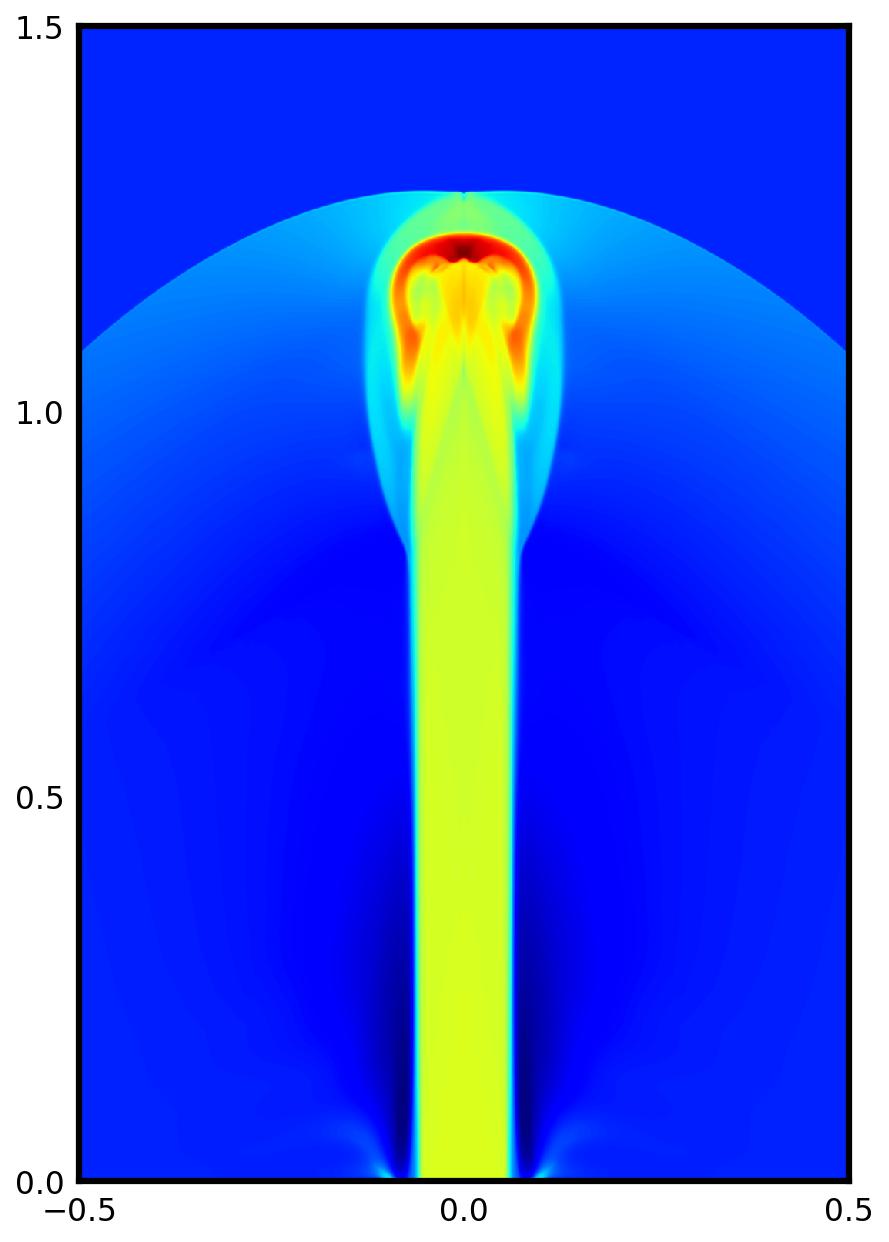}
			\end{subfigure}
			
			\begin{subfigure}{0.32\textwidth}
				\includegraphics[scale=0.32]{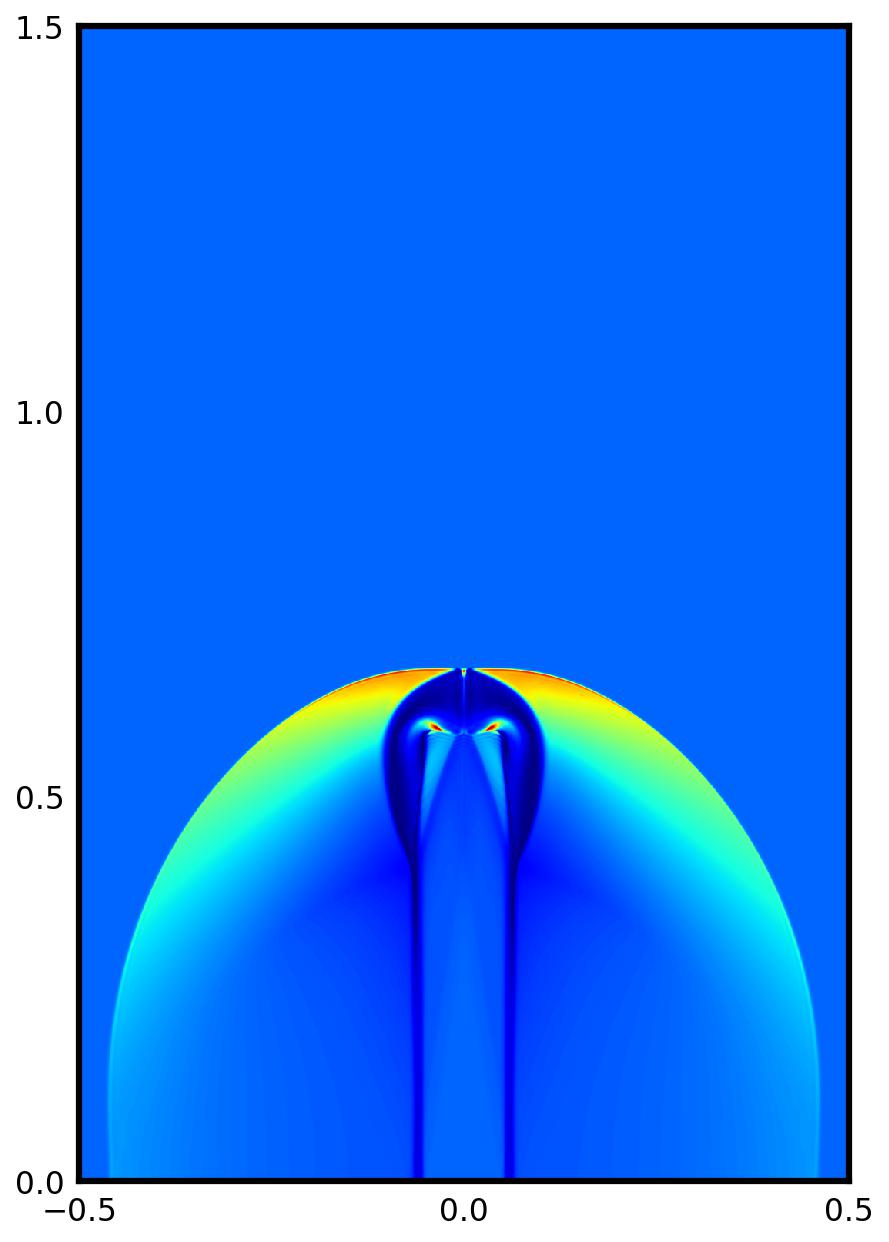}
			\end{subfigure}
			\hfill
			\begin{subfigure}{0.32\textwidth}
				\includegraphics[scale=0.32]{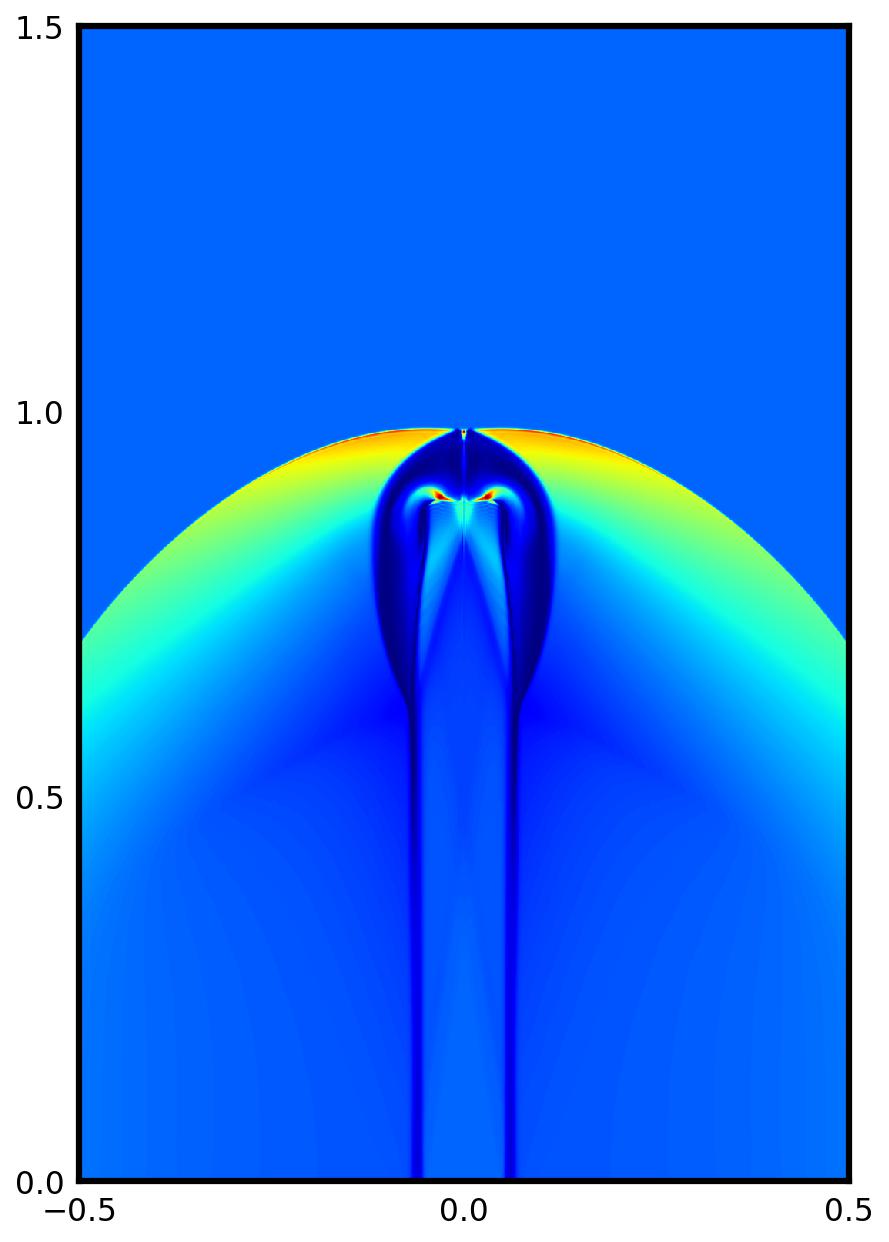}
			\end{subfigure}
			\hfill
			\begin{subfigure}{0.32\textwidth}
				\includegraphics[scale=0.32]{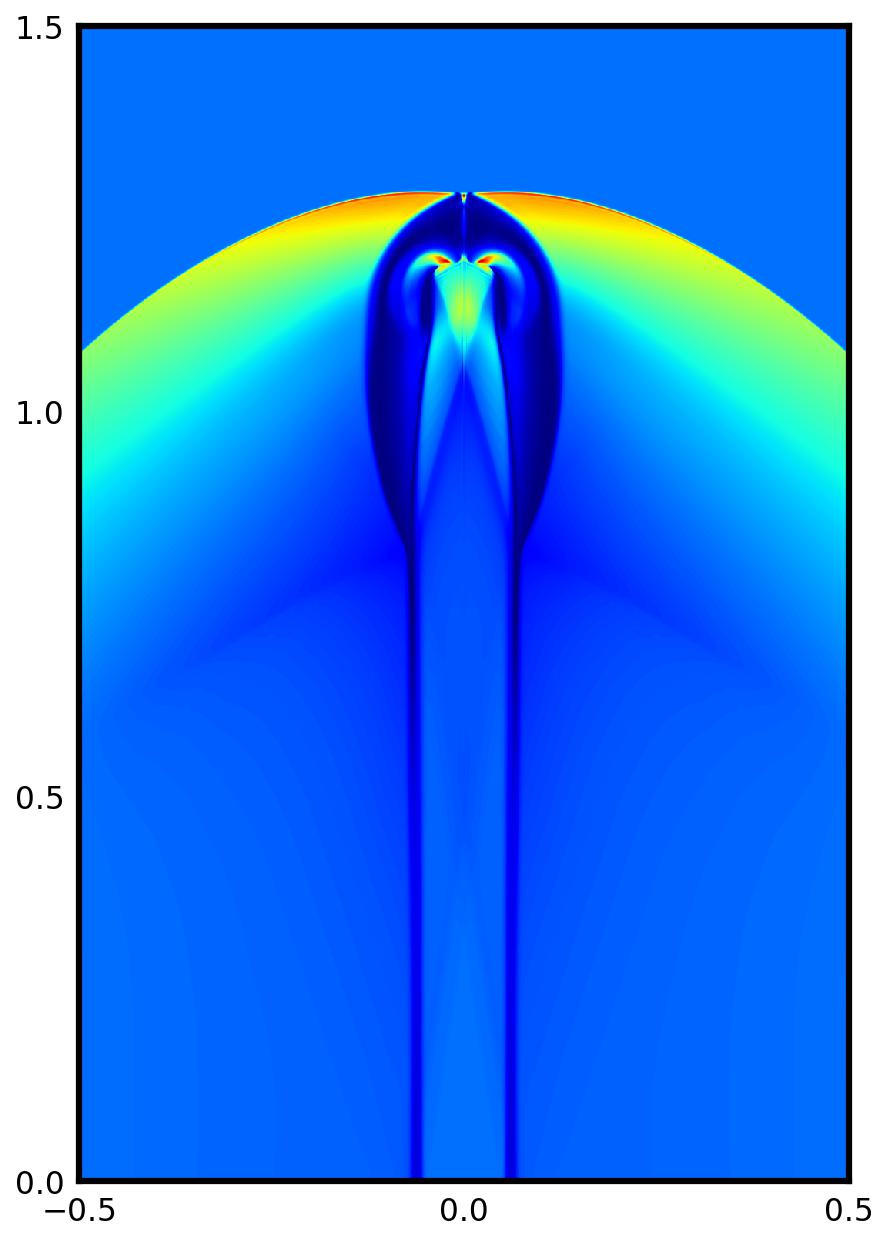}
			\end{subfigure}
			
			\caption{Mach 800 jet with $B_0 = \sqrt{20000}$: Density logarithm (top) and magnetic pressure (bottom) at $t = 0.001, 0.0015$, and $0.002$ (from left to right).
			}
			\label{fig:Ex-Jet_800_20000}
		\end{figure}

		\begin{figure}[!htb]
			\centering
			\begin{subfigure}{0.32\textwidth}
				\includegraphics[scale=0.32]{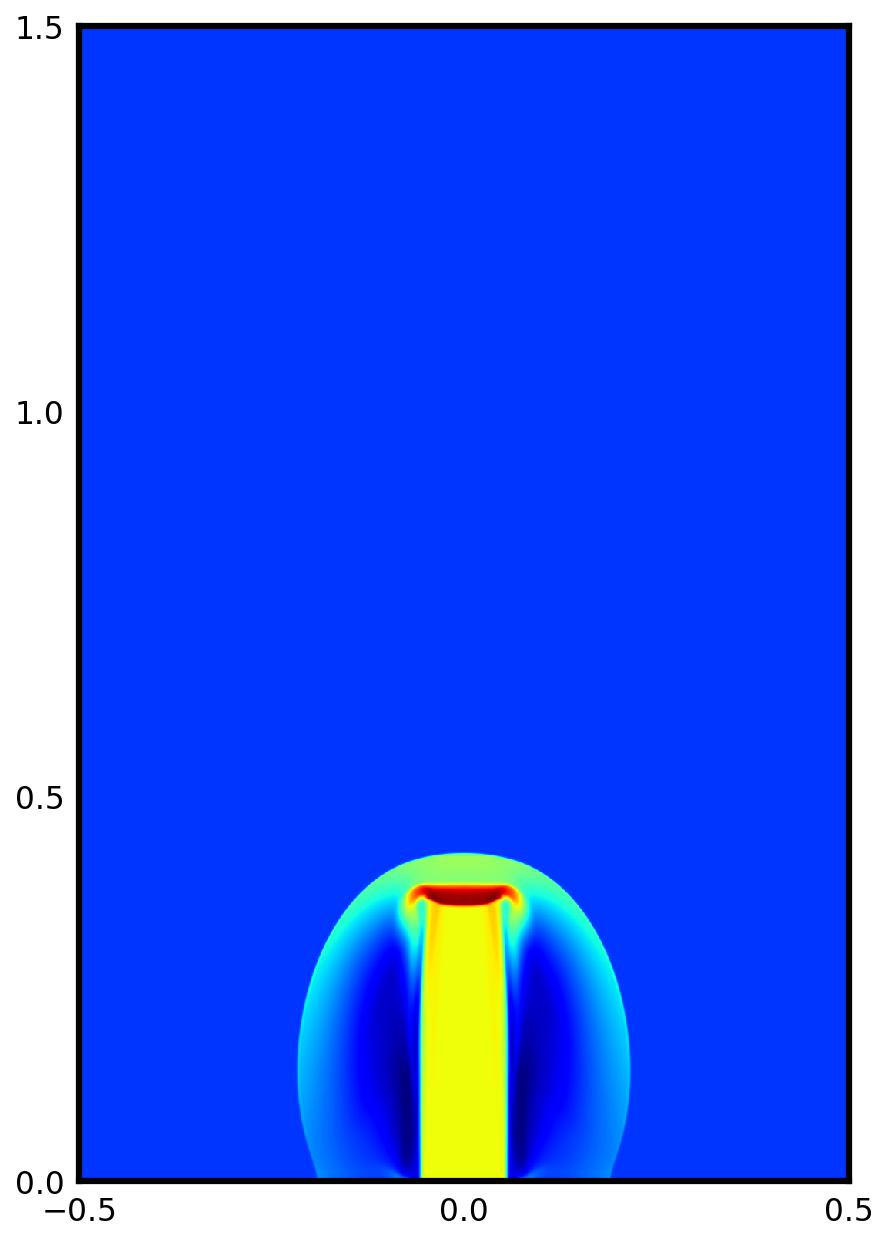}
			\end{subfigure}
			\hfill
			\begin{subfigure}{0.32\textwidth}
				\includegraphics[scale=0.32]{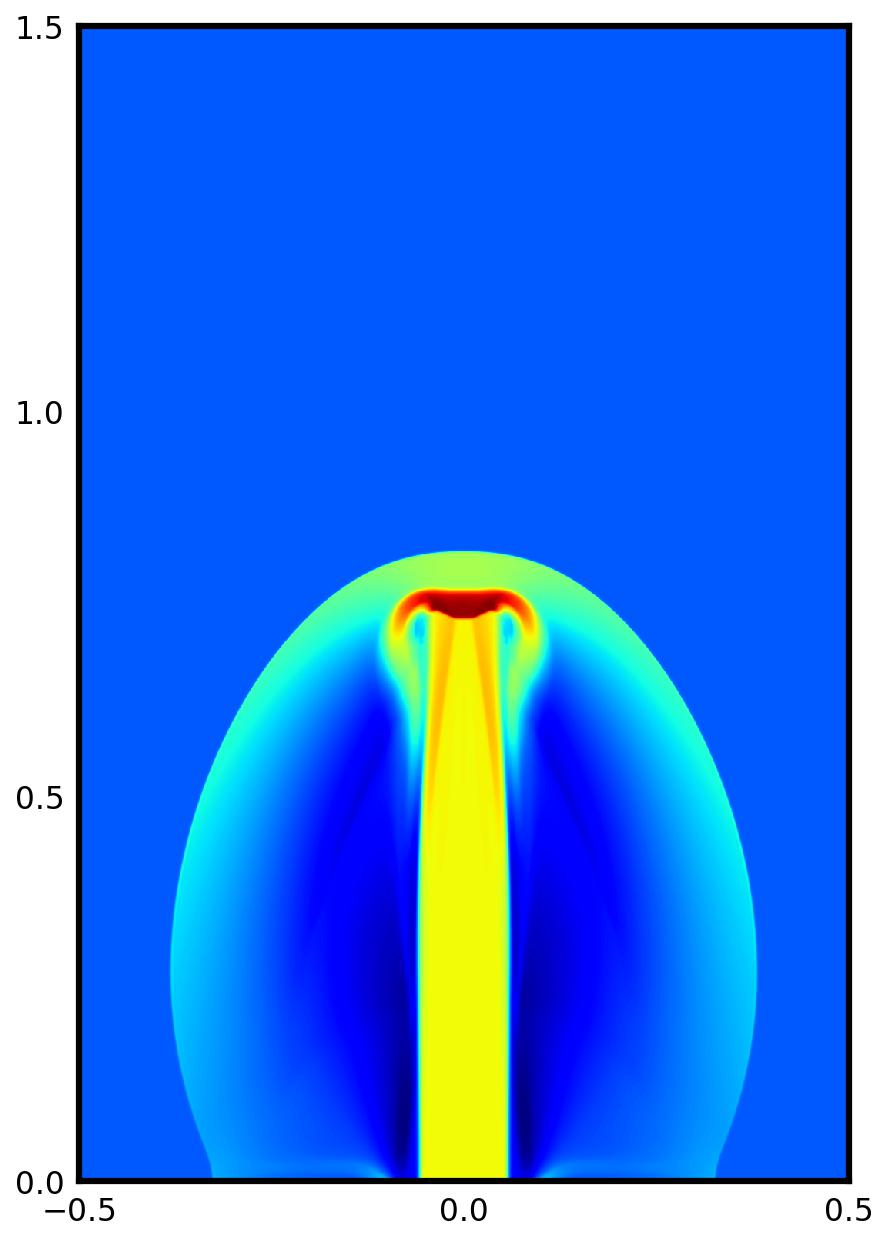}
			\end{subfigure}
			\hfill
			\begin{subfigure}{0.32\textwidth}
				\includegraphics[scale=0.32]{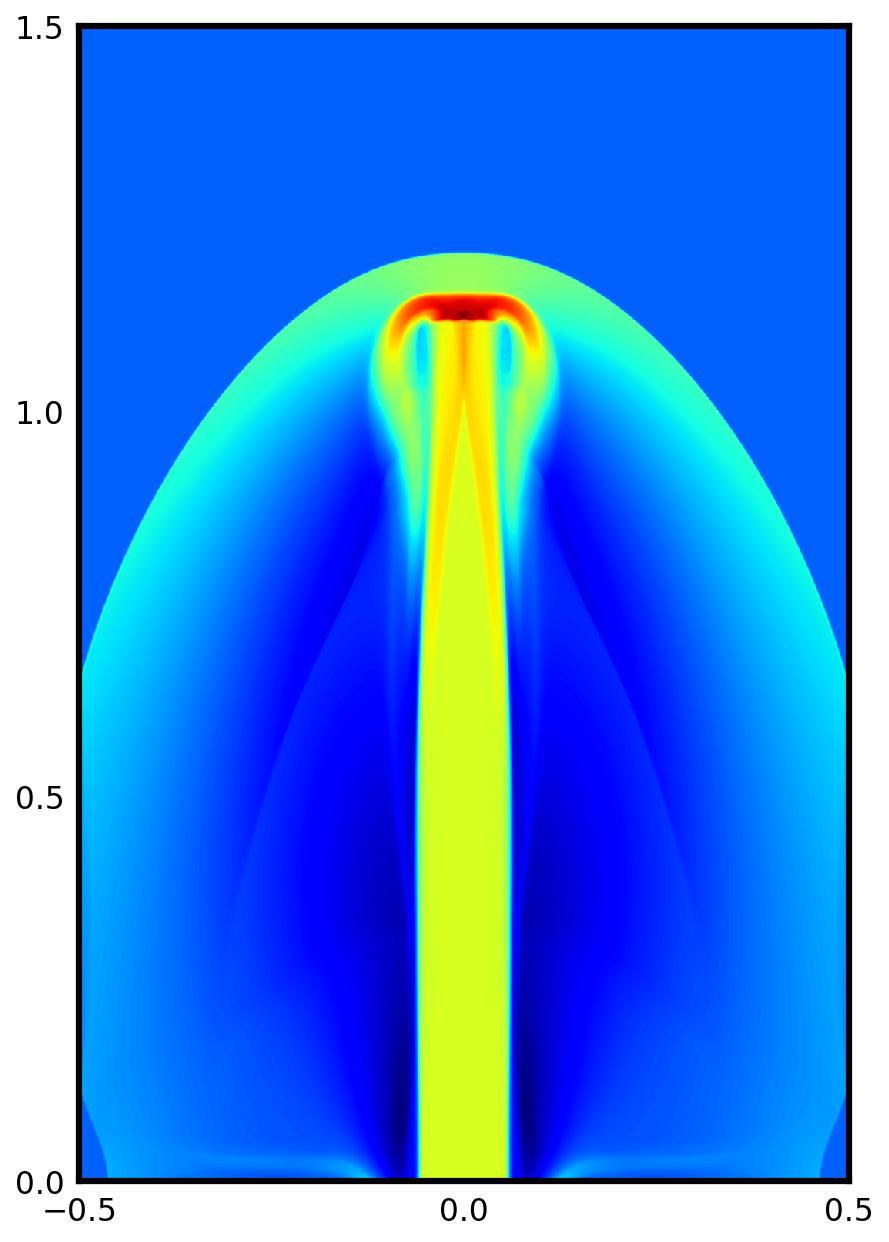}
			\end{subfigure}
			
			\caption{Mach 2000 jet problem with $B_0 = \sqrt{20000}$: Density logarithm at $t = 0.00025, 0.0005$, and $0.00075$ (from left to right).
			}
			\label{fig:Ex-Jet_2000_20000}
		\end{figure} 
		
		\begin{figure}[!htb]
			\centering
			\begin{subfigure}{0.32\textwidth}
				\includegraphics[scale=0.32]{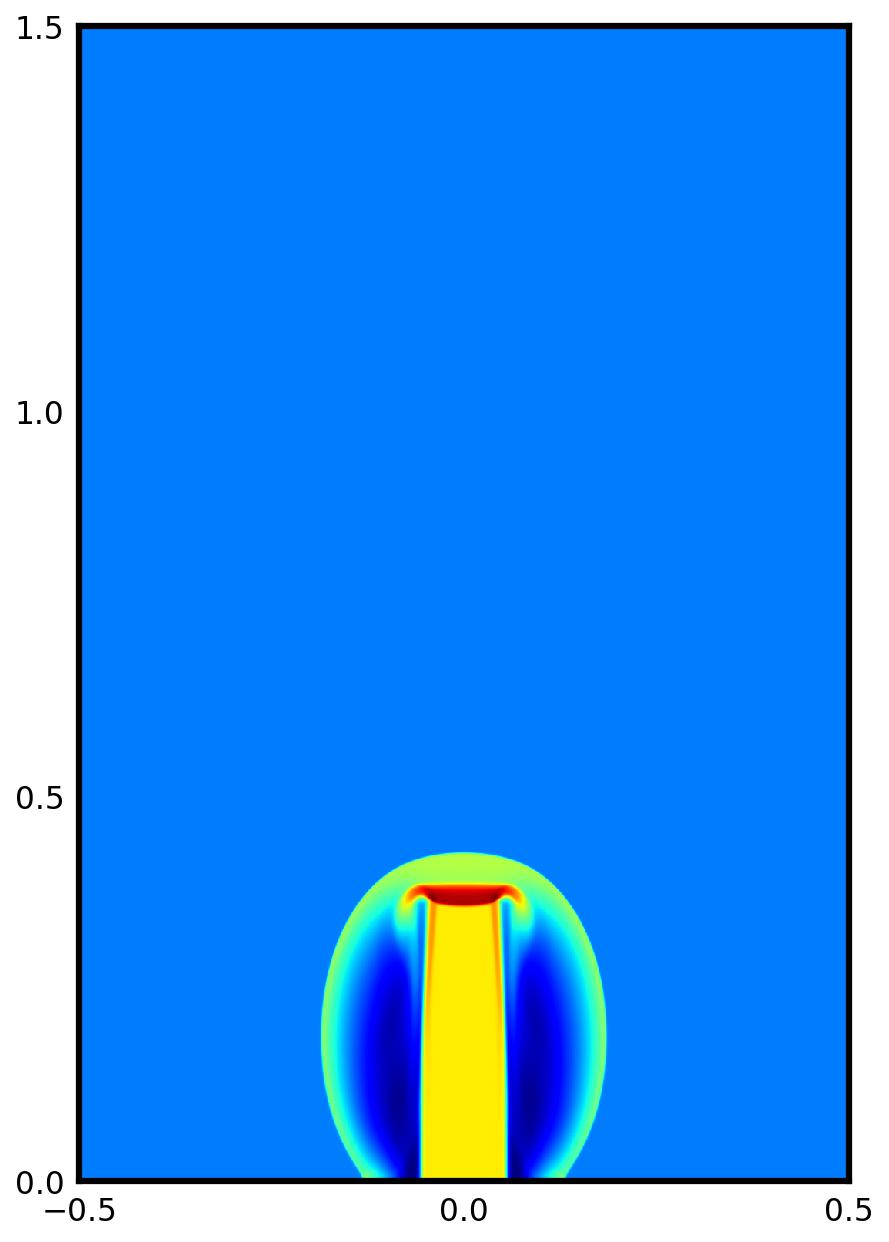}
			\end{subfigure}
			\hfill
			\begin{subfigure}{0.32\textwidth}
				\includegraphics[scale=0.32]{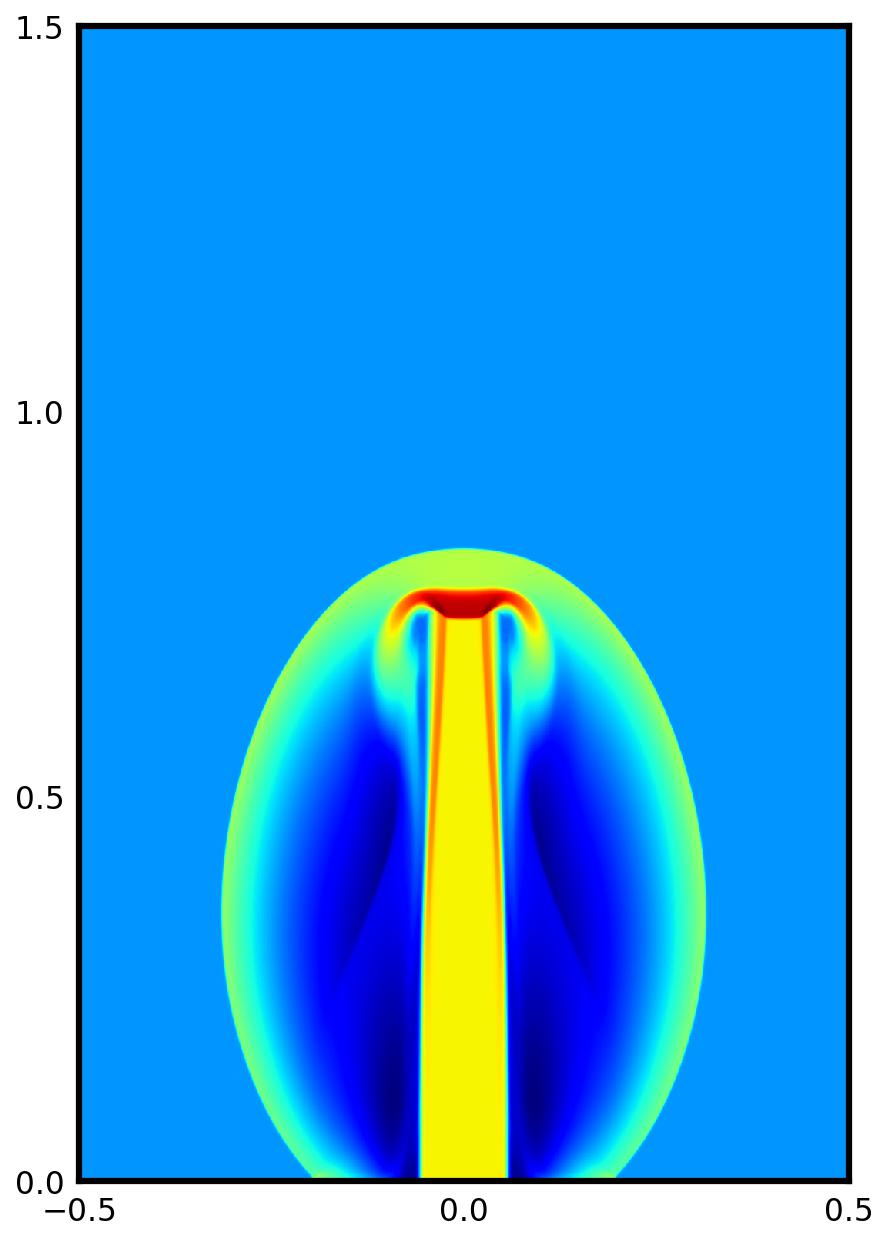}
			\end{subfigure}
			\hfill
			\begin{subfigure}{0.32\textwidth}
				\includegraphics[scale=0.32]{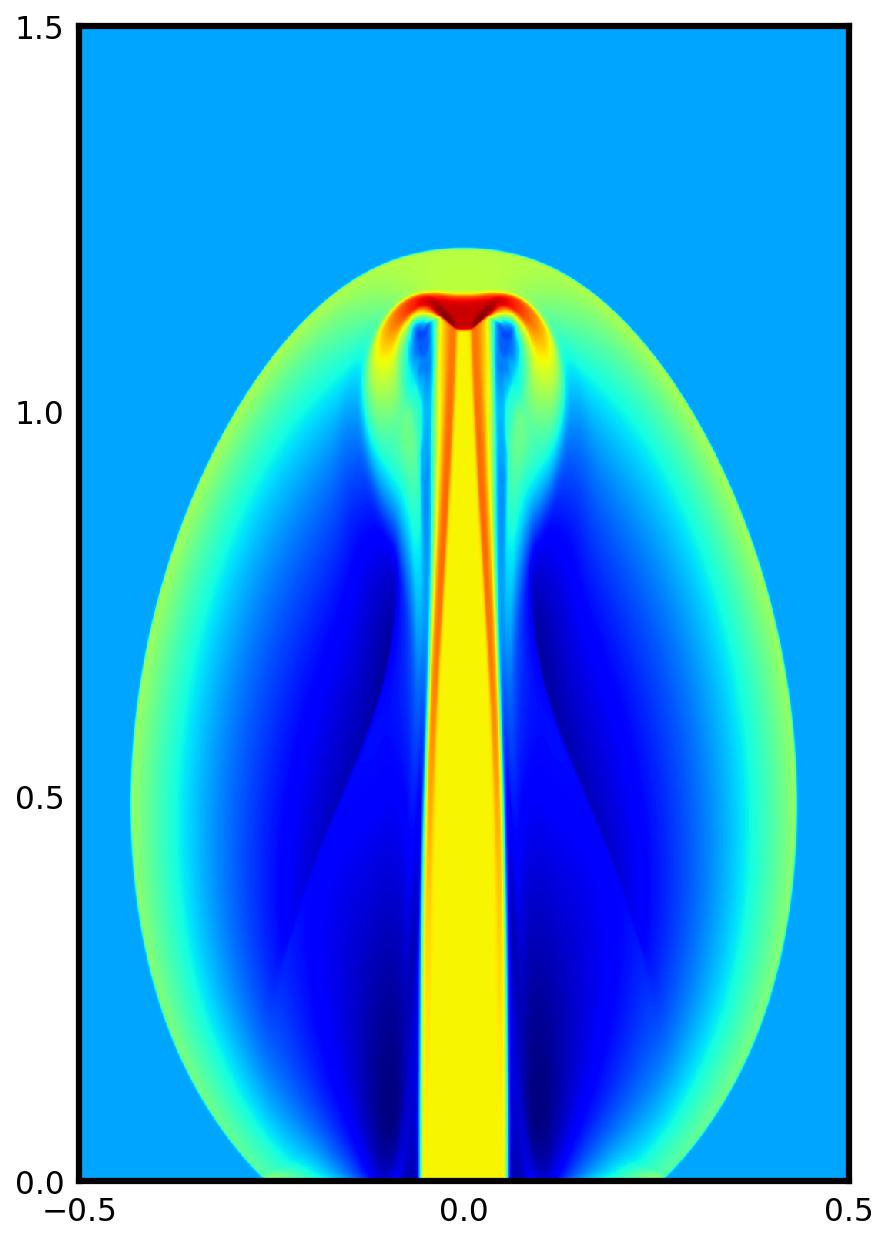}
			\end{subfigure}
			
			\caption{Mach 10000 jet problem with $B_0 = \sqrt{20000}$: Density logarithm at $t = 0.00005, 0.0001$, and $0.00015$ (from left to right).
			}
			\label{fig:Ex-Jet_10000_20000}
		\end{figure} 	 
	\end{expl}
	
	\section{Conclusion}\label{S.6}

	In this paper, we have developed and analyzed a novel, efficient, and robust second-order positivity-preserving constrained transport (PPCT) scheme for ideal magnetohydrodynamics (MHD) on non-staggered Cartesian meshes. This scheme achieves two essential physical properties: positivity of density and pressure, and a globally discrete divergence-free (DDF) condition for the magnetic field. The approach builds on the splitting strategy in \cite{Dao2024Structure}, which decomposes the MHD system into a compressible Euler subsystem with a steady magnetic field and a magnetic subsystem with steady density and internal energy. To handle these subsystems, we propose a hybrid finite volume-finite difference (FV-FD) method, combining a provably positivity-preserving (PP) finite volume method for the Euler subsystem with an unconditionally stable finite difference constrained transport (CT) method for the magnetic subsystem. A second-order Strang splitting technique is used to alternate between these two methods, ensuring both positivity and global DDF properties efficiently.
	
	The finite volume method for the Euler subsystem incorporates a second-order reconstruction based on primitive variables, which avoids the overshoots and nonphysical oscillations that tend to occur when reconstructing with conservative variables. A specially designed PP limiter ensures the positivity of reconstructed limiting values and updated cell averages during time-stepping. This limiter maintains second-order accuracy and includes a velocity slope suppression mechanism that provides an a priori condition for positivity of updated cell averages. Our rigorous theoretical proof is based on the geometric quasilinearization (GQL) approach \cite{WuShu2021GQL} to transform the nonlinear positivity constraint into a set of linear conditions, enabling us to overcome the challenges posed by nonlinearity.
	
	For the magnetic subsystem, we develop an implicit finite difference CT method that naturally enforces the globally DDF condition on non-staggered grids. Unlike the finite element method \cite{Dao2024Structure}, our method is computationally efficient and avoids matrix computations. We introduce an iterative algorithm to solve the resulting nonlinear system, which demonstrates exponential convergence. The residual error is typically reduced to machine precision in 5 to 9 iterations and never exceeds 20 steps. Unique solvability and convergence of this algorithm have been proven in theory under a CFL-like condition.  Furthermore, we rigorously establish that this CT scheme conserves energy. Since the finite difference method for the magnetic subsystem is unconditionally stable and preserves both density and internal energy, the time step for stability and the positivity-preserving property of the PPCT scheme is restricted by a mild CFL condition that applies to the Euler subsystem.
	
	The PPCT scheme has been extended to three dimensions while maintaining all of its structure-preserving properties, including positivity, global DDF enforcement, and energy conservation. To validate the scheme, we tested it on several benchmark and challenging problems. The numerical results demonstrate that the PPCT method achieves high resolution and robust performance across a variety of MHD scenarios.
	
	In conclusion, the PPCT scheme is an efficient, robust second-order FV-FD hybrid method with a mild time step restriction, making it easy to implement on non-staggered Cartesian meshes. It effectively preserves both positivity and the globally DDF property, and it ensures conservation of mass and total energy. Future work will focus on exploring the convergence analysis of fast solvers for the nonlinear algebraic system arising from the implicit CT discretization and extending the PPCT scheme to unstructured meshes and other MHD models.

\bibliographystyle{siamplain}
\bibliography{references}
\end{document}